\documentclass{amsart}

\title{A model theoretic study of right-angled buildings}
\date{August 16, 2017}
\author{Andreas Baudisch, Amador Martin-Pizarro and Martin Ziegler}
\address{Institut f\"ur Mathematik,Humboldt-Universit\"at zu Berlin,
  D-10099 Berlin, Germany
\newline
\indent Universit\'e de Lyon, CNRS UMR 5208, Universit\'e Lyon
1, Institut Camille Jordan, 43 boulevard du 11 novembre 1918,
F--69622 Villeurbanne Cedex, France
\newline
\indent Mathematisches Institut, Albert-Ludwigs-Universit\"at
Freiburg, D-79104 Freiburg, Germany}
\email{baudisch@mathematik.hu-berlin.de}
\email{pizarro@math.univ-lyon1.fr} \email{ziegler@uni-freiburg.de}
\thanks{The second author conducted research with support of programme
  ANR-09-BLAN-0047 Modig, ANR-13-BS01-0006 Valcomo as well as an
Alexander von Humboldt-Stiftung Forschungsstipendium f\"ur erfahrene
Wissenschaftler 3.3-FRA/1137360
  STP}


\usepackage[latin1]{inputenc}
\usepackage{amssymb, amsmath, tikz}
\usepackage{expdlist}
\usetikzlibrary{matrix,arrows, calc}
\usetikzlibrary{fit}
\usetikzlibrary{positioning}

\usepackage{enumerate,xspace}
\usepackage{mdwlist}
\RequirePackage{ifpdf}
\ifpdf
   \usepackage[pdftex]{hyperref}
\else
   \usepackage[hypertex]{hyperref}
\fi

\theoremstyle{plain}
\newtheorem{theorem}{Theorem}[section]
\newtheorem{cor}[theorem]{Corollary}
\newtheorem{prop}[theorem]{Proposition}
\newtheorem*{claim}{Claim}
\newtheorem{lemma}[theorem]{Lemma}
\newtheorem*{thmA}{Theorem A}
\newtheorem*{thmB}{Theorem B}
\newtheorem*{thmC}{Theorem C}
\newtheorem*{thmD}{Theorem D}

\theoremstyle{definition}
\newtheorem{remark}[theorem]{Remark}

\newtheorem{definition}[theorem]{Definition}

\newtheorem*{notation}{Notation}

\theoremstyle{definition}

\newcommand{\nc}{\newcommand}
\nc{\Q}{\mathbb{Q}}
\nc{\C}{\mathfrak{C}}
\nc{\N}{\mathbb{N}}
\nc{\cb}{\operatorname{Cb}}
\nc{\Kom}{\mathbb{K}} 
\nc{\K}{\mathcal{K}}
\nc{\tp}{\operatorname{tp}}
\nc{\stp}{\operatorname{stp}}

\def\dcl{\mathrm{dcl}}
\def\acl{\mathrm{acl}}
\def\aclq{\mathrm{acl}^\mathrm{eq}}

\nc{\RM}{\operatorname{RM}}
\nc{\U}{\operatorname{U}}
\nc{\p}{\operatorname{p}}

\nc{\A}{\mathcal{A}} 
\nc{\Cent}{\operatorname{C}} 

\nc{\lmto}[1]{\xrightarrow[{#1}]{}}
\nc{\ggen}[1]{\langle #1\rangle} 
\nc{\rest}[2]{#1\!\!\upharpoonright_{#2}}

\nc{\restr}[1]{\!\upharpoonright\!#1}

\nc{\cox}{\mathrm{Cox}(N)} 
\nc{\coxnu}{\mathrm{Cox}^0} 
\nc{\wob}{\mathrm{Wob}} 
\nc{\s}{\mathcal{S}} 
\nc{\sr}{\s_\mathrm{R}} 
\nc{\sL}{\s_\mathrm{L}} 
\nc{\str}{\xrightarrow{\ast}} 
\nc{\inv}{^{-1}}
\DeclareMathOperator{\m}{M}
\nc{\mi}{\m_\infty}
\nc{\mo}{\m_0(\Gamma)}
\nc{\psg}{\mathrm{PS}_{\,\Gamma}}
\nc{\leqr}{\leq_\mathrm{r}}

\renewcommand{\iff}{if and only if\xspace}

\def\Ind#1#2{#1\setbox0=\hbox{$#1x$}\kern\wd0\hbox to 0pt{\hss$#1\mid$\hss}
\lower.9\ht0\hbox to 0pt{\hss$#1\smile$\hss}\kern\wd0}
\def\Notind#1#2{#1\setbox0=\hbox{$#1x$}\kern\wd0\hbox to
0pt{\mathchardef\nn="0236\hss$#1\nn$\kern1.4\wd0\hss}\hbox
to 0pt{\hss$#1\mid$\hss}\lower.9\ht0
\hbox to 0pt{\hss$#1\smile$\hss}\kern\wd0}
\def\ind{\mathop{\mathpalette\Ind{}}}
\def\nind{\mathop{\mathpalette\Notind{}}}

\nc{\sop}[1]{\xrightarrow[#1]{}}
\nc{\op}[1]{\xrightarrow{#1}}
\DeclareMathOperator{\nh}{N} 
\DeclareMathOperator{\w}{w} 
\DeclareMathOperator{\pp}{P} 

\nc{\ez}{einfach zusammenhängend\xspace}
\nc{\SC}{simply connected\xspace}

\DeclareMathOperator{\Rd}{R_{div}}
\nc{\ks}{\prec^*}

\DeclareMathOperator{\Rkl}{R_\prec}

\nc{\arcThroughThreePoints}[4][]{
\coordinate (middle1) at ($(#2)!.5!(#3)$);
\coordinate (middle2) at ($(#3)!.5!(#4)$);
\coordinate (aux1) at ($(middle1)!1!90:(#3)$);
\coordinate (aux2) at ($(middle2)!1!90:(#4)$);
\coordinate (center) at ($(intersection of middle1--aux1 and
middle2--aux2)$);
\draw[#1]
 let \p1=($(#2)-(center)$),
      \p2=($(#4)-(center)$),
      \n0={veclen(\p1)},       
      \n1={atan2(\x1,\y1)}, 
      \n2={atan2(\x2,\y2)},
      \n3={\n2>\n1?\n2:\n2+360}
    in (#2) arc(\n1:\n3:\n0);
}

\begin{document}
\begin{abstract}
  We study the model theory of right-angled buildings with
  infinite residues. For every Coxeter graph we obtain a complete
  theory with a natural axiomatisation, which is $\omega$-stable and
  equational. Furthermore, we provide sharp lower and upper bounds for
  its degree of ampleness, computed exclusively in terms of the
  associated Coxeter graph. This generalises and provides an
  alternative treatment of the free pseudospace.
\end{abstract}
\maketitle

\tableofcontents

\section{Introduction}

A \emph{Coxeter group} $(W,\Gamma)$ consists of a group $W$ with a
fixed set $\Gamma$ of generators and defining relations
$(\gamma\cdot\delta)^{m_{\gamma,\delta}}=1$, where
$m_{\gamma,\gamma}=1$ and $m_{\gamma,\delta}=m_{\delta,\gamma}$, for
$\gamma\neq \delta$, is either $\infty$ or an integer larger than
$1$. We will exclusively consider finitely generated Coxeter groups,
with $\Gamma$ finite. A \emph{word} $w$ is a
finite sequence on the generators from $\Gamma$, and $w$ is
\emph{reduced} if its length is minimal with respect to all words
representing the same element of $W$. A \emph{chamber system}
$(X,W,\Gamma)$ for the Coxeter group $(W,\Gamma)$ is a set $X$,
equipped with a family of equivalence relations $(\sim_{\gamma}$,
$\gamma\in\Gamma$). If $w=\gamma_1\cdots \gamma_n$ is a reduced word,
a \emph{reduced path of type} $w$ from $x$ to $y$ in $X$ is a
sequence $x=x_0,\ldots,x_n=y$ such that $x_{i-1}$ and $x_i$ are
$\sim_{\gamma_i}$-related and different for every $1\leq i\leq n$. A
chamber system $(X,W,\Gamma)$ is a \emph{building} if each
$\sim_{\gamma}$-class contains at least two elements and such that,
for every pair $x$ and $y$ in $X$, there exists a element $g\in W$
such that there is a reduced path of type $w$ from $x$ to $y$ \iff
the  word $w$ represents $g$. It follows that $g$ is uniquely
determined by $x$ and $y$, and that the reduced path connecting $x$
and $y$ is uniquely determined by its type $w$. We refer the reader to
\cite{tG02} for a pleasant introduction to buildings.

The Coxeter group $(W,\Gamma)$ is \emph{right-angled} if for every
$\gamma\neq\delta$, the value $m_{\gamma,\delta}$ is either $2$ or
$\infty$. So $W$ is determined by its \emph{Coxeter diagram}: a graph
with vertex set $\Gamma$ such that $\gamma$ and $\delta$ have an edge
connecting them, which we denote by $R(\gamma,\delta)$, if
$m_{\gamma,\delta}=\infty$. In an abuse of notation, we will denote
this graph by $\Gamma$ as well. The elements of $\Gamma$ will be
referred to as \emph{colours} or \emph{levels}. Note that, for
involutions $\gamma$ and $\delta$, the relation
$(\gamma\cdot\delta)^2=1$ means that $\gamma$ and $\delta$ commute. We
call a word $w'$ a \emph{permutation} of $w$ if it can be obtained
from $w$ by a sequence of swaps of commuting consecutive generators.
In right-angled Coxeter groups, a word $w$ is reduced \iff no
permutation of $w$ has the form $w_1\cdot \gamma\cdot\gamma\cdot w_2$,
for some generator $\gamma$. Every element of $W$ is represented by a
unique reduced word, up to permutation. In a building $(X, W,\Gamma)$,
the the class $x/\!\sim_{\gamma}$ is the $\gamma$\emph{-residue} of
$x$. A right-angled Coxeter group admits a unique (up to isomorphism)
countable building $\mathrm{B}_0(\Gamma)$ with infinite residues
\cite[Proposition 5.1]{HP03}, which we call \emph{rich}.

A right-angled building can also be described in terms of an incidence
geometry or, as we will refer to, a \emph{coloured graph}. The
vertices of colour $\gamma$ are equivalence classes of $\sim^\gamma$,
the transitive closure of all $\sim_{\gamma'}$, for
$\gamma'\not=\gamma$. Two vertices are connected by an edge if the
corresponding classes intersect. The coloured graph associated to the
rich building given by the diagram

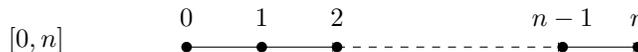
{\begin{figure}[h]
   \centering   \begin{tikzpicture}[->=latex,text height=1ex,text
       depth=1ex]

     \fill (-2,0) node {$[0,n]$};
     \fill  (0,0)  node [above]  {$0$}  circle (2pt);
     \fill   (1,0)  node [above]  {$1$} circle (2pt);
     \fill   (2,0)  node [above] {$2$} circle (2pt);
      \fill   (5,0)  node [above] {$n-1$} circle (2pt);
       \fill   (6,0)  node [above] {$n$} circle (2pt);

     \draw  (0,0) --  (1,0) --  (2,0) ;
     \draw[dashed] (2,0) -- (5,0) ;
     \draw (5,0) -- (6,0) ;
   \end{tikzpicture}
 \end{figure}}

\noindent is, as noticed by Tent \cite{kT14}, the prime model of the
theory of the \emph{free $n$-dimensional pseudospace}. A model of a
theory is prime if it elementarily embeds into every model. The
$n$-dimensional pseudospace, also considered by the authors in
\cite{BMPZ13}, witnessed the strictness of the ample hierarchy for
$\omega$-stable theories. We are indebted to Tent for pointing out the
connection between the free pseudospace and Tits buildings, which was
the starting point of the present work.

Recall that a countable complete theory is $\omega$-stable if
there is a \emph{rank function} $R$ defined on the collection of definable
sets of a (sufficiently saturated) model $M$ with the following
principle: If $X\subset M^n$ is definable, then $R(X)>\alpha$ \iff $X$
contains an infinite family of pairwise disjoint definable sets $Y_i$ with
$R(Y_i)\geq \alpha$. The smallest rank function is called \emph{Morley
rank}. The Morley rank of a type is the smallest Morley rank of
its formulae. This notion agrees with the Cantor-Bendixson rank on the
space of types over an $\omega$-saturated model, equipped with the Stone
topology. For an algebraically closed field with no additional structure,
definable sets are exactly the Zariski constructible ones, and Morley rank
coincides with the Zariski
dimension.

If $M$ is a \emph{group of finite Morley rank}, that
is, it carries a definable group structure and the Morley rank is always
finite, this notion of dimension is well-behaved. For example, given a
definable fibration $S\subset X\times Y\to Y$, the subset consisting of
those $y$ in $Y$ such that the fibre over $y$ has dimension $k$ is
definable for every $k$ in $N$. If all fibres have constant Morley rank
$k$, then $\RM(X)=\RM(Y)+k$.

Motivated by a famous conjecture about the structure of \emph{strongly
  minimal} sets (that is, irreducible definable sets of rank $1$), the
Algebraicity Conjecture states that every simple group of finite
Morley rank can be seen as an algebraic group over an algebraically
closed field, which is itself interpretable in the mere group
structure. Though the general conjecture on strongly minimal sets was
proven to be false \cite{Hr93}, work on the Algebraicity Conjecture,
which remains open, has become a fruitful research area, combining
ideas from model theory as well as the classification of simple finite
groups.

 If an $\omega$-stable theory
does not interpret a certain incidence configuration present in
euclidean space, then it interprets neither infinite fields nor
specific possible counterexamples to the Algebraicity Conjecture,
called \emph{bad groups} \cite{Pi95}. The notion of $n$-ampleness
\cite{Pi00, Ev03} for a theory generalises the incidence configuration
given in euclidean $(n+1)$-space by flags of affine subspaces of
increasing dimension, from a single point to a hyperplane. Ampleness
introduces thus a geometrical hierarchy, according to which
algebraically closed fields are $n$-ample for every $n$. The first two
levels of this hierarchy suffice to describe the structure of
definable groups \cite{HrPi87, Pi95}: they are virtually abelian, if
the $\omega$-stable group is not $1$-ample, and virtually nilpotent if
the group has finite Morley rank and is not $2$-ample. However, little
is known from $2$ onwards. In particular, whether the ample hierarchy
was strict remained long unknown. Evans conjectured that his example
could be accordingly modified to illustrate the strictness. Extending
the construction in \cite{BP00}, where a $2$-ample theory was produced
which interprets no infinite group (and thus, no infinite field is
interpretable), the aforementioned free $n$-dimensional pseudospace
was constructed \cite{BMPZ13,kT14} for every $n$, whose theory is
$\omega$-stable and $n$-ample yet not $(n+1)$-ample. In particular,
the $n$-dimensional pseudospace is a graph with $n+1$ many colours,
labelled from $0$ to $n$ such that the induced subgraph on consecutive
colours is an infinite pseudoplane, whose theory was known to be
$1$-ample but not $2$-ample.

In this article, we will provide an alternative approach to the above
construction, which incorporates as well the rich buildings of every
right-angled Coxeter group. As explained in Section \ref{S:Buildings},
a right-angled building $B$ can be recovered from its associated
coloured graph $M$: The elements of $B$ correspond to the \emph{flags}
of $M$, coloured subgraphs of $M$ isomorphic to $\Gamma$. This article
deals with the model theory of $\mo$, the coloured graph associated to
the rich building $\mathrm{B}_0(\Gamma)$ given by $\Gamma$. Though
models of the theory of $\mo$ need not arise from buildings, given
flags $F$ and $G$ in a model, a notion of a \emph{reduced path}
between $F$ and $G$ can be defined, whose corresponding word consists
of letters which are non-empty connected subsets of $\Gamma$. We show
that the coloured graphs associated to rich buildings are \emph{simply
  connected}: any two reduced paths between two given flags have the
same word, up to permutation. A combinatorial study of the reduction
of a path between two flags allows us to show the following crucial
result (\emph{cf.}\ Theorem \ref{S:SC_elementar}):

\begin{thmA}
  Simple connectedness is an elementary property.
\end{thmA}

Let $\psg$ denotes the collection of sentences stating the following
two elementary properties: Simple connectedness and that, given any
flag $G$ and a colour $\gamma$, there are infinitely many flags which
differ from $G$ only at the vertex of colour $\gamma$. The following
theorem (\emph{cf.}\ Theorem \ref{T:el} and Corollaries \ref{C:stab}
and \ref{C:Ranks}) yields that $\psg$ axiomatises the complete theory
of $\mo$:

\begin{thmB}
  The theory $\psg$ is complete and $\omega$-stable of Morley rank
  $\omega^{(K-1)}$, where $K$ is the cardinality of a connected
  component of\/ $\Gamma$ of largest size. The coloured graph $\mo$,
  associated to the countable rich building $\mathrm{B}_0(\Gamma)$, is
  the unique prime model of $\psg$.
\end{thmB}

Morley rank defines a notion of independence, which agrees in the
$\omega$-stable case with non-forking, as introduced by Shelah. A
remarkable feature of non-forking independence, which rules out the
existence of infinite definable groups, is \emph{total triviality}:
whenever we consider a base set of parameters $D$, given tuples $a$,
$b$ and $c$ such that $a$ is independent both from $b$ and from $c$
over $D$, then it is independent from $b,c$ over $D$. Recall that a
canonical basis of a type $p$ over a model is some set, fixed
pointwise by exactly those automorphisms $\alpha$ of a sufficiently
saturated model $N$ fixing the global non-forking extension
$\mathbf{p}$ of $p$ over $N$. If $\cb(p)$ exists, then it is unique,
up to interdefinability, though generally canonical bases only exist
as \emph{imaginary} elements in the expansion $T^\mathrm{eq}$ of an
$\omega$-stable theory $T$. Within the wider class of stable theories,
there is a distinguished subclass consisting of the \emph{equational}
ones, where each definable set in every cartesian product $N^n$ of a
model $N$ is a boolean combination of instances of $n$-equations
$\varphi(x,y)$, that is, the tuple $x$ has length $n$ and the family
of finite intersections of instances $\varphi(x,a)$ has the descending
chain condition. Whilst all known examples of stable theories arising
naturally in nature are equational, the only stable non-equational
theory constructed so far \cite{Hr91} is an expansion of the free
$2$-dimensional pseudospace. We obtain the following
(\emph{cf.}\ Corollaries \ref{C:Eq} and \ref{C:WEI} and Proposition
\ref{P:totally_trivial}):

\begin{thmC}
  The theory $\psg$ is equational, totally trivial with weak
  elimination of imaginaries: every type over a model has a canonical
  basis consisting of real elements.
\end{thmC}

To conclude, we provide lower and upper bounds on the ample degree of
the theory $\psg$, which can be described in terms of the underlying
Coxeter graph $\Gamma$. Set $r$ to be the minimal valency of the
non-isolated points of $\Gamma$ and $n$ the largest integer such that
the graph $[0,n]$, as before, embeds as a full subgraph of $\Gamma$.
We deduce (\emph{cf.}\ Theorem \ref{T:amplebounds}):

\begin{thmD}
  The theory $\psg$ is $n$-ample but not $(|\Gamma|-r+1)$-ample
\end{thmD}

\noindent These bounds are sharp and attained by $[0,n]$, by the
circular graph on $n+2$ points or by the extremal case of the complete
graph $\mathrm{K}_N$ on $N\geq 2$ elements, whose theory is $1$-ample
but not $2$-ample.

\begin{center}
  {\sc acknowledgements}
\end{center}

We thank the referee for a careful reading and for providing many valuable
comments, which helped us improve the presentation of this article.

\section{Buildings and geometries}\label{S:Buildings}

\begin{definition}
  A graph consists of set of vertices together with a symmetric and
irreflexive binary relation. Two vertices $a$ and $b$ are
\emph{adjacent} if the pair $(a,b)$ lies in the relation.

  Given a finite graph $\Gamma$, its associated \emph{right-angled
    Coxeter group} $(W,\Gamma)$ consists of the group $W$ generated by
  the elements of $\Gamma$ with defining relations:
  \begin{align*}
    \gamma^2&=1 &\text{for all $\gamma$ in
      $\Gamma$}\\ \gamma\delta&=\delta\gamma &\text{if $\gamma$ and
      $\delta$ are not adjacent.}
  \end{align*}
\end{definition}
\noindent As a convention, no element $\gamma$ \emph{commutes} with
itself.\\

{\bf From now on, all Coxeter groups are right-angled.}\\

Fix a Coxeter group $(W,\Gamma)$. A word $v=\gamma_1\cdots \gamma_n$
in the generators is \emph{reduced}
if there is no pair $i\neq j$ such that $\gamma_i$ equals $\gamma_j$
and commutes  (i.e.\ is not adjacent) with every letter occurring
between $\gamma_i$ and $\gamma_j$. Two words are \emph{equivalent} if
they
represent the same element of $W$. Clearly, every word is equivalent
to a reduced one. A reduced word $w$ \emph{commutes} with
$\gamma\in\Gamma$ if every element of $w$ does.

A word $v'$ is a \emph{permutation} of $v$ if it can be obtained from
$v$ by a sequence of commutations on pairs of commuting generators. A
permutation of a reduced word is again reduced.

The following is easy to see.
\begin{lemma}
  Two reduced words $u$ and $v$ are equivalent \iff $u$ is a
  permutation of $v$.
\end{lemma}

\noindent For $s$ a subset of $\Gamma$, let $\ggen s$ denote the
subgroup of $W$
generated by $s$.
\begin{cor}\label{C:GG_Inter}
  Given two subsets $s$ and $t$ of\/ $\Gamma$,
  \[\ggen{s\cap t}=\ggen s \cap \ggen t.\]
\end{cor}

\begin{definition}
  A \emph{chamber system} $(X,W,\Gamma)$  for the Coxeter group
  $(W,\Gamma)$ consists of a set $X$ equipped with a family of
equivalence
  relations $\sim_{\gamma}$ for each $\gamma\in\Gamma$. Given a word
  $w=\gamma_1\cdots \gamma_n$, a \emph{path of type} $w$ from $x$ to
  $y$ in $X$ is a sequence $x=x_0,\ldots,x_n=y$ such that $x_{i-1}$
  and $x_i$ are different and $\sim_{\gamma_i}$-related for every
  $1\leq i\leq n$. A path of type $w$ is \emph{reduced} if $w$ is.

  A chamber system $(X,W,\Gamma)$ is a \emph{building} if each
  $\sim_{\gamma}$-class contains at least two elements, and such that,
  for every pair $x$ and $y$ in $X$, there exists a element $g\in W$
  with the property that there is a reduced path of type $w$ from $x$
  to $y$ \iff the word $w$ represents $g$.
\end{definition}
\noindent We will refer to a chamber system $(X,W,\Gamma)$ uniquely
by the underlying set $X$ if the corresponding Coxeter group
$(W,\Gamma)$ is clear. The following two lemmas can be easily shown.

\begin{lemma}\label{L:unique_path}
  In a building $X$, a reduced path of type $w$ connecting $x$ and
  $y$ is uniquely determined by $w$, $x$ and $y$.
\end{lemma}

We will denote the existence of a path of type $w$ connecting $x$ to
$y$ by $x\op{w}y$. In particular, we have that $x\op{\gamma}y$ \iff
$x\neq y$ are $\sim_\gamma$-related.

\begin{lemma}\label{L:building_equiv}
  A chamber system $X$ is a building \iff the following four
  conditions hold:
  \quad\begin{enumerate}[\quad (a)]
  \item Every $\sim_\gamma$-class has at least two elements.
  \item\label{L:building_equiv:connect} Every two elements of $X$ are
    connected by a path.
  \item\label{L:building_equiv:commute} Given two commuting
generators $\gamma$ and $\delta$, if the elements $x$ and $y$ are
connected by a path of type $\gamma\delta$, then $x$ and $y$ are also
connected by a path  of type $\delta\gamma$.
  \item\label{L:building_equiv:simple} There is no non-trivial closed
    reduced path.
  \end{enumerate}
\end{lemma}

\noindent A chamber system satisfying conditions
$(\ref{L:building_equiv:connect})$ and
$(\ref{L:building_equiv:commute})$ is called \emph{strongly
  connected}. A strongly connected chamber system is a
\emph{quasi-building} if it satisfies condition
$(\ref{L:building_equiv:simple})$. Lemma \ref{L:unique_path}
holds for quasi-buildings, as well.

\begin{remark}\label{R:unique_b*}
Given  elements $b \op{w} a$ in a quasi-building $A$
and $\lambda \in \Gamma$ commuting with $w$, if $a\sim_\lambda a^*\in A$,
then there is a unique $b^*$ in $A$ with $b\sim_\lambda b^*\op{w} a^*$.
\end{remark}

\begin{proof}

 Iterating \ref{L:building_equiv} $(\ref{L:building_equiv:commute})$, the
reduced path $b\op{w} a \sim_\lambda a^*$ yields some element $b^*$ such
that $b\sim_\lambda b^* \op{w} a^*$. Furthermore, this element is uniquely
determined by $b$, $a^*$ and the reduced word $\lambda\cdot w$, by Lemma
\ref{L:unique_path}.
\end{proof}

We will now produce certain  extensions of a given
quasi-building $A$. Fix some $\lambda$ in $\Gamma$ and an equivalence
class $a/{\sim_\lambda}$ in $A$. We will extend $A$ to a
quasi-building containing a new element in
$a/{\sim_\lambda}$. Let $B$ be the set of elements $x$ in $A$ which
are connected to $a$ by a reduced path of type $w$, where $w$
and $\lambda$ commute. In particular, the generator $\lambda$
does
not occur in $w$. Furthermore, if $b$ and $c$ in $B$ are
$\sim_\gamma$-related, then $\lambda$ and $\gamma$ commute.

Observe that $a$ lies in $B$. For every $b\in B$, introduce a new element
$b^*$. Denote  \[A(a^*)=A\cup\{b^*\}_{b\in B},\]
 \noindent and extend the
chamber structure of $A$ to $A(a^*)$ by setting
\[ b^*\sim_\gamma c^* \;\Leftrightarrow\; b\sim_\gamma
c\]
for all $b$, $c$ in $B$, and
\[b^*\sim_\gamma a' \;\Leftrightarrow\; \lambda=\gamma
\text{ and }b\sim_\lambda a',\]
for all $b\in B$ and $a'\in A$. In particular, if
$b$ is in $B$ and $b\op{w} a$, then we obtain a reduced path $b^*\op{w}
a^*$.

The extension $A(a^*)$ is called \emph{simple}.

\begin{lemma}\label{L:Aastar}
   The chamber system $A(a^*)$ is a quasi-building.
\end{lemma}
\begin{proof}
  Properties $(\ref{L:building_equiv:connect})$ and
  $(\ref{L:building_equiv:commute})$ of Lemma \ref{L:building_equiv}
  can be easily shown. For example, suppose that $\delta$ and
  $\gamma_1$ commute, suppose $b^*\sim_{\gamma_1} d^*\sim_\delta c^*$.
  Then $b\sim_{\gamma_1} d\sim_\delta c$, so there is some $d'$ in $A$
  such that $b\sim_\delta d'\sim_{\gamma_1} c'$, whence
  $b^*\sim_\delta (d')^*\sim_{\gamma_1} c^*$. The other cases are
  treated in a similar fashion. For property
  $(\ref{L:building_equiv:simple})$, note first that any two elements
  of $B$ are connected by a word which commutes with $\lambda$. Thus, a
reduced path cannot change sides twice between $A$
  and $A(a^*)\setminus A$, for otherwise it would contain a reduced
subpath of the form $\gamma\cdot w\cdot\gamma$, where $w$ commutes
with $\gamma$. A closed reduced path is hence either fully contained in
$A$ -- and thus trivial since $A$ is a quasi-building -- or fully
contained in $A(a^*)\setminus A$, in which case it is in bijection with a
closed reduced path in $A$, which must be then trivial.
\end{proof}

\begin{cor}\label{C:infte_resid}
   For every Coxeter group $(W,\Gamma)$, there is a countable building
   in which all $\sim_\gamma$-equivalence classes are infinite.
\end{cor}

We will now show that every subset of a quasi-building has a strongly
connected hull, which is attained by a sequence of simple extensions.

\begin{prop}\label{P:str_conn_bldgs}
  Given a strongly connected subset $A$ of a quasi-building $X$ and
  $a^*$ in $X\setminus A$ with $a^*\sim_{\lambda} a \in A$, the
  smallest strongly connected subset of $X$ containing
  $A\cup\{a^*\}$ is isomorphic to $A(a^*)$.
\end{prop}
\begin{proof}
Observe that $A$ is a quasi-building, since $X$ is. Let $B\subset A$
be as in the construction of $A(a^*)$, that is, the
set of elements $b\in A$, with $b\op{w} a$ and $w$ commutes with $\lambda$.

Every $b$ in $B$ yields a unique $b^*$ in $X$ such that
$b\op{\lambda}b^*\op{w}a^*$, by Remark \ref{R:unique_b*}. Since $w$ is
uniquely determined, up to permutation, by $b$ (and $a$), the element
$b^*$ depends only on $b$. By symmetry,
the element $b$ is determined by $b^*$. Thus, the elements $b^*$ are
pairwise distinct. Note that none of the $b^*$'s belong to $A$, since
$b^*\op{w}a^*\op{\gamma}a$ and $A$ is strongly connected.

Therefore $A(a^*)$ may be identified with a subset of $X$, which is
contained in every strongly connected extension of $A\cup\{a^*\}$. We
need only show that the chamber structure of $A\cup\{a^*\}$ agrees
with the structure induced by $X$. Given distinct elements $b$ and $c$
in $B$ with $b\sim_\delta c$, we need only show that $b^*\sim_\delta
c^*$, since the converse follows by replacing $a$, $b$ and $c$ with
$a^*$, $b^*$ and $c^*$.

The set $A$ is strongly connected, so there are reduced words $u$ and
$v$ such that $a\op{u}b$ and $a\op{v}c$ in $A$.

\begin{claim}
  Let $a$, $b$ and $c$ be distinct elements of a quasi-building $X$
  with $b\sim_\delta c$. Suppose that $a\op{u}b$ and $a\op{v}c$ for reduced
  words $u$ and $v$. Then there are three possibilities:
  \begin{enumerate}
  \item $u\cdot\delta$ is reduced.
  \item\label{L:dreieck:wdelta} $v\cdot\delta$ is reduced.
  \item\label{L:dreieck:gabel} There exists a reduced word
$w\cdot\delta$ and an element $a'$ such that $a\op{w}a'$,
$a'\op{\delta}b$ and $a'\op{\delta}c$.
  \end{enumerate}
The third case implies that both $u$ and $v$ are equivalent to
$w\cdot\delta$
\end{claim}

\begin{proof}[{\it Proof of the claim:}]
\renewcommand{\qed}{\noindent\emph{End of the proof of the
claim.}}
 If $u\cdot\delta$ is not reduced, up to permutation, we may assume
that
$u=w\cdot \delta$. Choose $a'$ in $X$ with
$a\op{w}a'\op{\delta}b$. Thus $a'\sim_\delta c$. Either $a'=c$, so $w$
is equivalent to $v$, which gives case
$(\ref{L:dreieck:wdelta})$, or $a'\op{\delta}c$, which gives case
  $(\ref{L:dreieck:gabel})$.

\end{proof}

If we apply the previous claim to our situation, we obtain three
possibilities:
  \begin{enumerate}
  \item The word $u\cdot\delta$ is reduced. Up to permutation,
    we have that $v=u\cdot\delta$. In particular, the elements
$\gamma$ and $\delta$ commute. Since $b\sim_\delta c \sim_\gamma
c^*$, there is some $c'$ in $X$ with $b\sim_\delta c' \sim_\gamma
c$. Note that, since $a^*\op{v} c^*$ and $a^*\op{u}b^*\op{\delta} c'$, we
have that $a^*\op{v} c'$ as well. Thus $c^*=c'$, by Remark
\ref{R:unique_b*}, so  $c^*\sim_\delta b^*$.

  \item The word $v\cdot\delta$ is reduced, which is treated
similar to the first case.
  \item
    For some reduced word $w\cdot\delta$, there is an $a'$ such
that
    $a\op{w}a'$, $a'\op{\delta}b$ and $a'\op{\delta}c$. Since $A$ is
strongly connected, the element $a'$ lies in $A$. By the
    first case, we have that $(a')^*\sim_\delta b^*$ and
$(a')^*\sim_\delta c^*$, whence $b^*\sim_\delta c^*$.
  \end{enumerate}

  If $b^*\sim_\delta a'$ for some $b$ in $B$ and $a'$ in $A$, then
the path $b\sim_\lambda b^* \sim_\delta a'$ cannot be reduced, since
$b^*$ does not lie in $A$. Thus $\delta=\lambda$ and
$b^*\sim_\lambda a'$.

\end{proof}

Corollary \ref{C:infte_resid} and Proposition \ref{P:str_conn_bldgs}
yield immediately the following result:

\begin{cor}(\emph{cf.} \cite[Proposition
    5.1]{HP03})\label{C:b_unten_0} Every Coxeter group $(W,\Gamma)$
  has, up to isomorphism, a unique countable building
  $\mathrm{B}_0(\Gamma)$, in which all $\sim_\gamma$-equivalence
  classes are infinite.
\end{cor}

We will study the model theory of buildings using the following
expansion of the natural language:

\begin{definition}\label{D:equ}
  Let $X$ be a chamber system for $(W,\Gamma)$ and $s$ a subset of
  $\Gamma$. By $\sim_\emptyset$, we denote the diagonal in $X\times
  X$. Otherwise, for $\emptyset\neq s\subset\Gamma$, the relation
  $\sim_s$ is the transitive closure of all $\sim_{\gamma}$, with
  $\gamma \in s\subset \Gamma$. The $\sim_s$-class of an element is
  called its $s$\emph{-residue}. In particular, its
  $\gamma$\emph{-residue} is its $\sim_\gamma$-class, which is often
  called $\gamma$\emph{-panel} in the literature.

For $\gamma$ in $\Gamma$, set $\sim^\gamma=\sim_{\Gamma\setminus
\{\gamma\}}$. The chamber system  $(X,\sim^{\gamma})_{\gamma\in\Gamma}$ is
called the associated  \emph{dual chamber system} of $X$.
\end{definition}

It is easy to see that $x\sim_s y$ \iff $x\op{w}y$ for some
$w\in\ggen{s}$. If $X$ is a quasi-building, the word $w$ is uniquely
determined as an element of $W$, so Corollary \ref{C:GG_Inter} implies
\begin{align*}
\label{E:GG_Inter} \sim_{s_1\cap
  s_2}&=\sim_{s_1}\cap\sim_{s_2},\\ \intertext{and particularly}
\tag{\dag}\sim_{\gamma} &=\bigcap\limits_{\beta\neq \gamma} \sim^{\beta}\/.
\end{align*}

\noindent Thus, for a quasi-building $X$, the chamber system
$(X,\sim_\gamma)_{\gamma\in\Gamma}$ is definable in its associated
dual chamber system $(X,\sim^\gamma)_{\gamma\in\Gamma}$. Clearly, the
latter is only definable if countable disjunctions are allowed.

\noindent The aim of this article is to study the complete theory of
$\mathrm{B}^0(\Gamma)$, \label{seite:definition_b_oben_0} the
associated dual chamber system of $\mathrm{B}_0(\Gamma)$, which was
defined in Corollary \ref{C:b_unten_0}.

\begin{lemma}\label{L:buildings}
  The dual chamber system $(X,\sim^\gamma)_{\gamma\in\Gamma}$ of a
  quasi-building $X$ has the following elementary properties:
  \begin{enumerate}
  \item\label{L:buildings:equals_all_gamma} Given $x$ and $y$ in $X$
    with $x\sim^{\gamma} y$ for all $\gamma \in \Gamma$, then $x=y$.
  \item\label{L:buildings:lift_flag} If
    $(x_{\gamma})_{\gamma\in\Gamma}$ is a \emph{coherent} sequence in
    $X$, i.e.\ whenever $\gamma$ and $\delta$ are adjacent, there
exists some $y_{\gamma,\delta}$ in $X$ with
    $x_\gamma\sim^{\gamma} y_{\gamma,\delta} \sim^{\delta}
x_{\delta}$, then there exists an
    element $z\in X$ with $z\sim^{\gamma} x_\gamma$ for all
    $\gamma\in\Gamma$.
  \end{enumerate}
\end{lemma}

\begin{proof}
  Property (\ref{L:buildings:equals_all_gamma}) clearly follows from
  condition (\ref{E:GG_Inter}).

  In order to show Property $(\ref{L:buildings:lift_flag})$, note that
  any singleton is a quasi-building and has property
  $(\ref{L:buildings:lift_flag})$. By Proposition
  \ref{P:str_conn_bldgs}, we need only show now that, if $A$ is a
  quasi-building with Property $(\ref{L:buildings:lift_flag})$, then
  so is $A(a^*)$.

  Let $(x_{\gamma})_{\gamma\in\Gamma}$ be a coherent sequence in
  $A(a^*)$, where $a^*\sim_{\lambda} a \in A$. Each $x_\gamma$ lies
  either in $A$ or is $\sim_\lambda$ connected to some element in $A$.
  Since only the $\sim^\gamma$-class of $x_\gamma$ matters, we may
  assume that all $x_\gamma$ belong to $A$, for $\gamma\neq\lambda$.
  If $x_\lambda$ is in $A$, then the result follows, since Property
  (\ref{L:buildings:lift_flag}) holds in $A$. Otherwise, if
  $x_\lambda$ does not belong to $A$, we have that $x_\lambda
  \op{w}a^* $, for a reduced word $w$ which commutes with $\lambda$.
Thus $x_\lambda \sim^\lambda a^*$ and we may assume that
  $x_\lambda=a^*$.

  By assumption, for every $\gamma$ adjacent with $\lambda$, there is
  some $y_{\lambda,\gamma}$ in $A(a^*)$ such that $a^*\sim^\lambda
  y_{\lambda,\gamma}\sim^\gamma x_\gamma$. The element
  $y_{\lambda,\gamma}$ cannot lie in $A$, for otherwise the reduced
path $a\sim_\lambda a^*\sim^\lambda y_{\lambda,\gamma}$ implies that
so is $a^*$. Thus $y_{\lambda,\gamma}$ is of the form
  $b^*_{\lambda,\gamma}$ for some $b_{\lambda,\gamma}\in A$.

  It follows that $a\sim^\lambda b_{\lambda,\gamma}\sim^\gamma
  x_\gamma$. Replacing the element $x_\lambda$ in the sequence
$(x_\gamma)$  by $a$, yields a new sequence contained in
$A$ and coherent. Thus, we find an element $c$ in $A$ such that
$c\sim^\lambda a$ and
  $c\sim^\gamma x_\gamma$ for $\gamma\neq\lambda$. Observe that
  $c^*\sim^\lambda a^*$ and $c^*\sim^\gamma x_\gamma$, so $c^*$ is
the desired element.

\end{proof}

\begin{definition}\label{D:pre-building}
  A chamber system $(X,\sim^\gamma)_{\gamma \in \Gamma}$ is a
  \emph{dual quasi-building} if it has properties
  $(\ref{L:buildings:equals_all_gamma})$ and
  $(\ref{L:buildings:lift_flag})$ from Lemma \ref{L:buildings}.
\end{definition}

The following remark can easily be verified.
\begin{remark}
  A chamber system $(X,\sim^\gamma)_{\gamma \in \Gamma}$ is a dual
  quasi-building \iff for every coherent sequence
  $(x_\gamma)_{\gamma\in\Gamma}$ there is a unique $z$ with
  $z\sim^{\gamma} x_\gamma$ for all $\gamma\in\Gamma$.
\end{remark}

\begin{definition}\label{D:colsp}
  A $\Gamma$\emph{-graph} $M$ is a coloured graph with colours
  $\A_{\gamma}(M)$ for $\gamma$ in $\Gamma$, and no edges between
  elements of $\A_{\gamma}(M)$ and $\A_{\delta}(M)$ if $\gamma$ and
  $\delta$ are not adjacent.

  A \emph{flag} $F$ of the $\Gamma$-graph $M$ is a subgraph
  $F=\{f_{\gamma}\}_{\gamma\in\Gamma}$, where each $f_{\gamma}$ lies
  in $\A_{\gamma}(M)$, such that the map $\gamma\mapsto
f_{\gamma}$ induces a
  graph isomorphism between $\Gamma$ and $F$.

  The $\Gamma$-graph $M$ is a $\Gamma$\emph{-space} if the following
  two additional properties are satisfied:
  \begin{enumerate}
  \item Every vertex belongs to a flag of $M$.
  \item Any two adjacent vertices in $M$ can be expanded to a flag of
    $M$.
  \end{enumerate}
\end{definition}

In particular, if $\Gamma$ is the complete graph $\Kom_3$, then the
following $\Kom_3$-graph is not a $\Kom_3$-space:

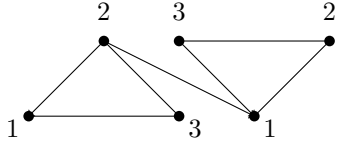
\begin{figure}[h]
  \centering

  \begin{tikzpicture}[->=latex,text height=1ex,text depth=1ex]

     \fill  (0,0)  node [below left]  {$1$}  circle (2pt);
     \fill   (2,0)  node [below right]  {$3$} circle (2pt);
     \fill   (1,1)  node [above] {$2$} circle (2pt);
      \fill   (3,0)  node [below right] {$1$} circle (2pt);
       \fill   (2,1)  node [above] {$3$} circle (2pt);
	\fill   (4,1)  node [above] {$2$} circle (2pt);

     \draw  (0,0) --  (2,0) --  (1,1) -- (0,0) ;
     \draw (1,1) -- (3,0) ;
     \draw (3,0) -- (4,1) -- (2,1) -- (3,0) ;

  \end{tikzpicture}
\end{figure}

\begin{theorem}\label{T:biinterpretierbar}
  The class of dual quasi-buildings for $(W,\Gamma)$ and the class of
  $\Gamma$-spaces are bi-interpretable.
\end{theorem}
\begin{proof}
  In a chamber system $(X,\sim^\gamma)_{\gamma\in \Gamma}$, we
  interpret a $\Gamma$-graph $\mathcal{M}(X)$ as follows: for every
  $\gamma \in \Gamma$, the colour $\A_{\gamma}$ is $X/\sim^\gamma$,
  the set of $\sim^\gamma$-classes of elements in $X$. We consider the
  $\A_{\gamma}$ as being pairwise disjoint. For the graph structure on
  $\mathcal{M}(X)$, we impose that two elements $u$ and $v$ are
  adjacent if $u \in \A_{\gamma}$ and $v \in\A_\delta$, the colours
  $\gamma$ and $\delta$ are adjacent and there is some $z \in X$ with
  $z \sim^{\gamma} u$ and $z \sim^{\delta} v$.

  Since every $x\in X$ gives rise to the flag
  $\phi(x)=\{x/\!\!\sim^\gamma\,\mid\gamma\in\Gamma\}$ of
  $\mathcal{M}(X)$, it is straight-forward to see that
  $\mathcal{M}(X)$ is a $\Gamma$-space.

  Given a $\Gamma$-graph $M$, we define a chamber system
  $(\mathcal{X}(M),\sim^\gamma)_{\gamma\in \Gamma}$, whose underlying
  set is the collection of flags of $M$. Two flags $F$ and $G$ are
  $\sim^{\gamma}$-related if their $\gamma$-vertices agree.

  We will first show that $\mathcal{X}(M)$ is a dual quasi-building.
  We need only check Property (\ref{L:buildings:lift_flag}). Assume
  that $(F_\gamma)$ is a coherent system of flags. Let $f_\gamma$
  denote the $\gamma$-vertex of $F_\gamma$. If $\gamma$ and $\delta$
  are adjacent, there is a flag $G$ such that $F_\gamma\sim^\gamma
  G\sim^\delta F_\delta$. In particular, the flag $G$ contains an edge
  between $f_\gamma$ and $f_\delta$. So $H=(f_\gamma)$ is a flag of
  $M$ and $H\sim^\gamma F_\gamma$ for all $\gamma\in\Gamma$.

  For a chamber system $(X,\sim^\gamma)_{\gamma\in \Gamma}$, the
  correspondence $x\mapsto\phi(x)$ defines a map
  $\phi:X\to\mathcal{X}(\mathcal{M}(X))$. It is easy to see that
  $\phi(x)\sim^\gamma\phi(y)$ \iff $x\sim^\gamma y$. If $X$ is a dual
  quasi-building, Property (\ref{L:buildings:equals_all_gamma})
  implies that $\phi$ is injective and Property
  (\ref{L:buildings:lift_flag}) that $\phi$ is surjective. Thus $X$
  and $\mathcal{X}(\mathcal{M}(X))$ are definably isomorphic.

  Given a $\Gamma$-graph $M$, the correspondence
  $a:\mathcal{M}(\mathcal{X}(M))\to M$ which associates to each class
  $F/\sim^\gamma$ the $\gamma$-vertex of $F$ is a bijection between
  $\mathcal{M}(\mathcal{X}(M))$ and the collection of vertices of $M$
  which belong to a flag of $M$. For adjacent $\gamma$ and $\delta$,
  there is an edge between $F/\sim^\gamma$ and $G/\sim^\delta$ \iff
  $a(F/\sim^\gamma)$ and $a(G/\sim^\delta)$ belong to a common flag of
  $M$. This shows that $a$ is a definable isomorphism if $M$ is a
  $\Gamma$-space.

  Thus, the classes of dual quasi-buildings and of
  $\Gamma$-spaces are bi-interpretable, as desired.
\end{proof}

In order to describe the model-theoretical properties of the dual
quasi-building $\mathrm{B}^0(\Gamma)$ (introduced right before Lemma
\ref{L:buildings}), we may therefore consider the
first-order theory of

\[\label{seite:definition_m_0} \mo =
\mathcal{M}(\mathrm{B}^0(\Gamma)),\]

\noindent its associated $\Gamma$-space. Our reason to do this is that
$\Gamma$-spaces, for certain Coxeter groups $\Gamma$, will be familiar
to the readers of \cite{BP00,  kT14, BMPZ13}. In particular, many of
the tools developped in \cite{BMPZ13} can be easily generalised and
adapted to this context. However, the whole model-theoretical
study of $\mathrm{B}^0(\Gamma)$ could be  done without passing to its
corresponding $\Gamma$-space.

\section{Simply connected $\Gamma$-spaces}\label{S:Simply_connected}

Recall that by a Coxeter group we mean a right-angled finitely
generated Coxeter group. From now on, fix a Coxeter group
$(W,\Gamma)$, which we will denote simply by $W$, with underlying
Coxeter graph $\Gamma$. In order to describe the first-order theory of the
structure $\mo$ obtained before, we will need to study non-standard paths
between flags.

\begin{notation}
   A \emph{letter} is a non-empty connected subset of the
graph $\Gamma$. Characters such as $s$ and $t$ will exclusively refer to
  letters. A \emph{word} $u$ is a finite sequence of letters.
\end{notation}
Every generator $\gamma$ in $\Gamma$ defines the letter $\{\gamma\}$.
In this way, every word in the generators can be considered as a word
in the above sense.

\begin{definition}\label{D:comm}
  Two letters $s$ and $t$ \emph{commute} if $s\cup t$ is not a letter,
  i.e.\ if the elements of $s$ commute with all elements of $t$. In
  particular, no letter commutes with itself. Two words commute if
  their letters respectively do. A word is \emph{commuting} if it
  consists of pairwise commuting letters. A \emph{permutation} of a
  word is obtained by repeatedly permuting adjacent commuting letters.
  Two words $u$ and $v$ are \emph{equivalent}, denoted by $u\approx
  v$, if one can be permuted into the other.
\end{definition}
The following is easy to see.
\begin{remark}\label{R:commuting_word}
  A commuting word $w=s_1\cdots s_n$ is determined up to equivalence
  by its \emph{support} \[|w|=s_1\cup\cdots\cup s_n,\] where the
$s_i$'s are the connected components of $|w|$.
\end{remark}
\noindent We will often write $w$ instead of $|w|$ if $w$ is a
commuting word.\\

Throughout this section, we will work inside some ambient
$\Gamma$-space.

\begin{definition}\label{D:Flageq}
  A \emph{weak flag path} $P$ from the flag $F$ to the flag $G$ is a
finite sequence  $F=F_0,F_1,\ldots,F_n=G$ of flags such that the colours
where
$F_{i+1}$ and $F_i$
  differ form a letter $s_{i+1}$. To such a path, we associate the
  word $u=s_1\cdots s_n$ and denote this by $F\sop{u} G$.
\end{definition}

In the light of Theorem \ref{T:biinterpretierbar}, we transfer to
$\Gamma$-spaces Definition \ref{D:equ} and say that two flags $F$
and $G$ are \emph{$A$-equivalent},
\[F\sim_AG,\]
if the set of colours where $F$ and $G$ differ is contained in
$A\subset\Gamma$. Similarly as in \cite[Lemma 6.3]{BMPZ13}, by
decomposing any subset of $\Gamma$ as a disjoint union of its
connected components, we obtain the following.

\begin{lemma}\label{L:modlemma}
  Two flags $F$ and $G$ are $A$-equivalent \iff they can be connected
  by a weak flag path whose word consists of letters contained in $A$.
  In particular, setting $A=\Gamma$, any two flags can be connected by
  a weak flag path.
\end{lemma}
\begin{proof}
  For $F=\{f_\gamma\mid\gamma\in\Gamma\}$ and
  $G=\{g_\gamma\mid\gamma\in\Gamma\}$, let $s_1,\ldots,s_n$ be the
  connected components of $\{\gamma\in\Gamma\mid
  f_\gamma\not=g_\gamma\}$. Set
  \[F_i=\{f_\gamma\mid\gamma\not\in s_1\cup\ldots\cup
  s_i\}\cup\{g_\gamma \mid\gamma\in s_1\cup\ldots\cup s_i\}.\] Then
  $F_0,\ldots,F_n$ is a weak path which connects $F$ and $G$, and it
  has word $s_1\cdots s_n$.
\end{proof}

The proof yields the following two corollaries.

\begin{cor}\label{C:permute_path}
  Given flags $F$ and $G$, there exists a commuting word $u$, unique
  up to equivalence, such that $F\sop{u}G$. For every permutation $u'$
  of $u$, there is a unique weak flag path from $F$ to $G$ with word
  $u'$.
\end{cor}
Uniqueness of $u$ follows from fact that the letters of $u$ are the
connected components of the set of colours where $F$ and $G$ differ.
\begin{cor}\label{C:path_with_commuting_words}
  In a path $F\sop{u}H\sop{v}G$, where $u$ and $v$ commute, the
  flag $H$ is uniquely determined by $F$, $G$, $u$ and $v$.
\end{cor}

To each relation $F\sop{s} G$ associate the natural bijection
$\A_{s}(F)\to\A_{s}(G)$. It is easy to see, that a weak flag path
$P:F_0\sop{s_0}F_1\cdots F_{n-1}\sop{s_{n-1}}F_n$ is completely
determined by $F_0$ (likewise by $F_n$) and the sequence of associated
maps. If $s_1\cdots s_n$ is commuting, this sequence of maps is
equivalent to the collection $\A_{s_i}(F_0)\to\A_{s_i}(F_n)$, which
gives an alternative proof of \ref{C:permute_path}. More generally,
the following lemma holds.

\begin{lemma}[Permutation of a path]\label{L:path_permutation}
  Given a weak flag path $P:F_0\sop{u}F_n$ and  a permutation $u'$
of $u$, there is a unique weak flag path $P':F_0\sop{u'}F_n$
such that the associated map of each letter in $P$ is the same as the
  associated map of the corresponding occurrence of that letter in
  $P'$.
\end{lemma}
\noindent Such a path $P'$ is a \emph{permutation} of $P$.

\begin{definition}\label{D:splitting}
  \par\noindent
  \begin{enumerate}[$\bullet$]\setlength{\itemsep}{.5em}
  \item A \emph{splitting} of a letter $s$ is a (possibly trivial)
    word, whose letters are properly contained in $s$. Given words $u$
    and $v$, we say that $u\prec v$ if $u$ is equivalent to a word
    obtained from $v$ by replacing at least one occurrence of a letter
    in $v$ by a splitting. We write $u\preceq v$ if either $u\prec v$
    or $u\approx v$.
   \item Whenever $F\sop{s}G$ and there is no weak flag
     path from $F$ to $G$ whose word is a splitting of $s$, write
      $F\op{s}G$. A  \emph{flag path} from $F$ to $G$ with word
$u=s_1\cdots s_n$,
     denoted by $F\op{u} G$,
     is a weak flag path
     $F=F_0,\ldots,F_n=G$ such that $F_i\op{s_{i+1}}F_{i+1}$ for
     $i=0,\ldots,n-1$.
  \end{enumerate}
\end{definition}

\noindent It is easy to see that the relation $\prec$ is transitive,
irreflexive and well-founded (\emph{cf.}\ \cite[Lemma 5.26]{BMPZ13}.)
A permutation of a flag path is again a flag path. If $F\sop{s}G$,
whether $F\op{s}G$ depends on the ambient $\Gamma$-space.

\vspace{1em}

\begin{lemma}\label{L:Flagpath}
 If $F\sop{u} G$, then $F\op{v} G$ for some $v\preceq u$.
\end{lemma}

\begin{proof}
  Suppose $F\sop{u}G$. If this is not a flag path, it contains a step
  $F'\sop{s}G'$ which can be replaced by $F'\sop{w}G'$, where $w$ is a
  splitting of $G$. This yields a weak flag path $F\sop{u'}G$ with
  $u'\prec u$. Since $\prec$ is well-founded, this procedure stops
  with a flag path $F\op{v} G$ for some word $v\preceq u$.
\end{proof}

\vspace{1em}

\begin{notation}
  The notation $s\subset t$ means that $s$ is a subset of $t$,
  possibly with $s=t$. We will use the notation $s\subsetneq t$ to
  emphasise that $s$ is a proper subset of $t$.
\end{notation}

\vspace{1em}

\begin{definition}\label{D:Op_redpath}
  \par\noindent

  \begin{enumerate}[$\bullet$]\setlength{\itemsep}{.5em}
  \item A word $v=s_1\cdots s_n$ is \emph{reduced} if there is no pair
    $i\neq j$ such that $s_i\subset s_j$ and $s_i$ commutes with all
    letters in $v$ between $s_i$ and $s_j$.

  \item A flag path is \emph{reduced} if its associated word is.

  \item The reduced word $v$ is a \emph{reduct} of $u$ if it can
    obtained from $u$ by the following rules\\[-.5em]
    \begin{description}\setlength{\itemsep}{.5em}
      \setlength{\itemindent}{0em}
      \item[\sc Commutation] Permute consecutive commuting
        letters.
      \item[\sc Absorption] If $s$ is contained in $t$, replace a
        subword $s\cdot t$ (or $t\cdot s$) by $t$.
      \item[\sc Splitting] Replace a subword $s\cdot s$ by a splitting
        of $s$.
    \end{description}
    \vspace{.4em}
    We will denote this by $u\str v$ (\emph{cf.}\ \cite[Definition
      5.24]{BMPZ13}). Clearly $u\str v$ implies $v\preceq u$.
  \end{enumerate}
\end{definition}
It is easy to see that a word $u$ is reduced \iff any permutation of
$u$ is. Similarly, a path $P$ is reduced \iff any permutation of
$P$ is.

\vspace{1em}

Consider indexes $i\neq j$ in a word $v=s_1\cdots s_n$ such that
$s_i\subset s_j$ and $s_i$ commutes with all letters in $v$ between
$s_i$ and $s_j$. Using Commutation and Absorption, we can delete the
occurrence of the letter $s_i$. If $s_i=s_j$, we may also replace $s_j$
by a splitting of $s_j$. We call such an operation a
 \emph{generalised} Absorption or Splitting. It is easy to see
that
every reduct of a word can be obtained by a sequence of generalised
Absorptions and Splittings, followed by a permutation.

\begin{lemma}\label{L:RedPath}
  If $F\op{u} G$, then $F\op{v} G$ for some reduced $v$ with $u\str
  v$.
\end{lemma}

\begin{proof}
  If the path $F\op{u} G$ is not reduced, possibly after permutation,
we may
  assume that it contains a subpath $F'\op{s}H'\op{t}G'$, where
  $s\subset t$ (or $t\subset s$). One of the following reduction
  steps now applies:
  \vspace{.4em}
  \begin{description}\setlength{\itemsep}{.5em}
  \item[\sc Proper Absorption] If $s\subsetneq t$, remove $H'$
    since $F'\op{t}G'$, for otherwise, there would be a a splitting
$x$ of
    $t$ such that $F'\sop{x}G'$, which implies $H'\sop{s\cdot x}G'$,
    contradicting $H'\op{t}G'$.

  \item[\sc Absorption/Splitting] If $s=t$, note that $F'\sim_s G'$.
    Lemmata \ref{L:modlemma} and \ref{L:Flagpath} yield:
\vspace{.4em}
    \begin{description}\setlength{\itemsep}{.3em}
    \item[\sc Absorption] $F'\op{s}G'$, or
    \item[\sc Splitting] $F'\op{x}G'$ for some splitting $x$ of $s$.
    \end{description}
    Therefore, the flag $H'$ can be removed.
  \end{description}
    \vspace{.3em}
  \noindent Note that both Absorption and Splitting yield words which
are $\prec$-smaller than $u$. Thus, the process must eventually
stop.
\end{proof}

\begin{remark}
  We will see in Remark \ref{R:paths_exists} that, for every reduction
  $u\str v$, there is a flag path of word $u$ in a suitable
  $\Gamma$-space which can be reduced to a path with word $v$ by the
above procedure.
\end{remark}

\begin{cor}
  Any two flags can be connected by a reduced path.
\end{cor}

The following property of the ambient space will ensure that all
reduced paths between two given flags have equivalent words
(\emph{cf.}\ Proposition \ref{P:Eindwort}).

\begin{definition}\label{D:sc}
  A $\Gamma$-space $M$ is \emph{\SC} if there are no
  non-trivial closed reduced flag paths.
\end{definition}

\begin{lemma}\label{L:SCequiv}
  The $\Gamma$-space $M$ is \SC \iff the word of any closed flag path
  can be reduced to the trivial word\/ $1$.
\end{lemma}
\begin{proof}
  Suppose the condition on the right holds. Given a closed reduced
  flag path with word $u$, since $u\str 1$, then $u=1$, as $u$ is
  already reduced. For the other direction, given a closed path $P$
  with word $u$, apply Lemma \ref{L:RedPath} to obtain a closed
  reduced path whose word $v$ is a reduct of $u$. If $M$ is \SC, the
  word $v$ must be $1$, thus $u\str 1$.
\end{proof}

\begin{theorem}\label{T:M0_SC}
  The $\Gamma$-space $\mo$ is \SC.
\end{theorem}

\begin{proof}
  By the definition of $\mo$, as explained in the proof of
\ref{T:biinterpretierbar}, two flags $F$ and $G$ in $\mo$ have the same
$\gamma$-vertex if they can be connected by a flag path, whose letters are
singletons different from $\gamma$. Thus, if $F\op{s} G$, then $s$ must be
a singleton. All paths are singleton paths. Since
  $\mathrm{B}_0(\Gamma)$ is a building, there are no non-trivial
  closed reduced singleton paths.
\end{proof}

An interesting feature of \SC $\Gamma$-spaces is that the word of a
reduced flag path connecting two given flags is unique, up to
equivalence. For the proof of the following proposition, we need a
definition and a lemma.
\begin{definition}\label{D:abs}
  The letter $t$ is (properly) \emph{left-absorbed} by the word
$s_1\cdots
  s_n$, resp.\ \emph{right-absorbed}, \iff $t$ is
  (properly) contained in some $s_i$ and commutes with $s_1\cdots
s_{i-1}$, resp.\ with $s_{i+1}\cdots
  s_{n}$. A word $u$ is
  \emph{left-absorbed},  resp.\ \emph{right-absorbed}, by $v$ if each
letter in $u$ is.
\end{definition}

A word is reduced \iff it cannot be written as $u\cdot t\cdot v$,
where $t$ is either right-absorbed  by $u$ or left-absorbed  by
$v$.
The word $u=s_1\cdots
s_n$ left-absorbs $t$  \iff $u^{-1}$
right-absorbs $t$, where $u^{-1}=s_n\cdots s_1$.

\begin{lemma}\label{L:reduce_one_letter}
  Let $u\cdot s$ be reduced and $x$ be a splitting of $s$. Every
  reduct of $u\cdot x$ is equivalent to $u\cdot x_1$ for some
  $x_1\preceq x$.
\end{lemma}
\begin{proof}
  Since $u\cdot s$ is reduced, a generalised Absorption or Splitting
  for $u\cdot x$ cannot happen for a pair $s_i\subset s_j$, where
  $s_i$ is contained in $u$. So letters contained in $u$ will never be
  deleted.
\end{proof}

\begin{prop}\label{P:Eindwort}
  If the flags $F$ and $G$ in a \SC $\Gamma$-space $M$ are connected
  by reduced flag paths with respective words $u$ and $v$, then
  $u\approx v$.
\end{prop}

\begin{proof}
  We prove it by $\prec$-induction on $u$ and $v$. If $F=G$, the claim
  is equivalent to simple connectedness. We may therefore assume that
  $F\neq G$. Since $u\cdot v\inv$ belongs to a closed non-trivial flag
  path, it cannot be a reduced word. So assume that $u=u_1\cdot s$ and
  $s$ is right-absorbed by $v$. Splitting the first path accordingly,
  \[ F\op{u_1}F_1\op{s}G,\]

  \noindent  we distinguish two cases:
  \begin{enumerate}[(1)]
  \item
    The letter $s$ is properly right-absorbed by $v$.
Then $v$  is the only reduct of $v\cdot s$ and therefore
$F\op{v}G\op{s}F_1$
    gives $F\op{v}F_1$.

    
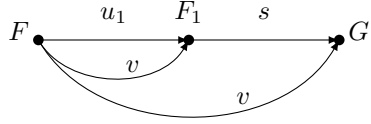
\begin{figure}[h]
      \centering

      \begin{tikzpicture}[>=latex,text height=1ex,text depth=1ex]

        \fill  node [left] (0,0) (F) {$F$}  circle (2pt);
        \fill   (2,0)  node [above]  {$F_1$} circle (2pt);
        \fill   (4,0)  node [right] (G) {$G$} circle (2pt);

     \draw[->]  (F) -- node[pos=.5,above] {$u_1$} (2,0);
     \draw[->]  (2,0) -- node [pos=.5, above] {$s$} (G);

     \draw[bend right=60,->]  (0,0) to node [pos=.7, left] {$v$}
     (2,0);
     \draw[bend right=60,->]  (0,0) to node [pos=.7, left] {$v$}
     (4,0);

      \end{tikzpicture}

    \end{figure}
    \noindent Since $u_1\prec u$, induction yields that $u_1 \approx v$.
    In particular, the letter $s$ is right-absorbed by $u_1$,
    contradicting that $u$ is reduced.  \vskip3mm

  \item After a permutation $v$ has the form $v_1\cdot s$. We split
    the second path as \[F\op{v_1}G_1\op{s}G.\] We have then either
    $G_1\op{s}F_1$ or $G_1\op{x}F_1$ for a reduced splitting $x$ of
    $s$. If $G_1\op{s}F_1$, then $F\op{v}F_1$, which contradicts as before
that $F\op{u_1}F$. So $G_1\op{x}F_1$.
    
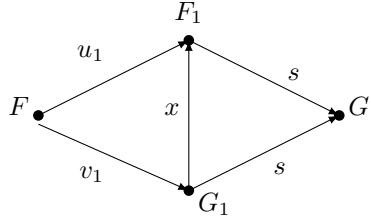
\begin{figure}[h]
      \centering

      \begin{tikzpicture}[>=latex,text height=1ex,text depth=1ex]

        \fill  node [left] (0,0) (F) {$F$}  circle (2pt);
        \fill   (2,1)  node [above]  {$F_1$} circle (2pt);
        \fill   (2,-1)  node [below right]  {$G_1$} circle (2pt);
     \fill   (4,0)  node [right ] (G) {$G$} circle (2pt);

     \draw[->]  (0,0)  --  node[pos=.5,above left] {$u_1$}(2,1);
     \draw[->] (2,1)  --   node [pos=.8, above left] {$s$}(4,0);

     \draw[->]  (F) -- node [pos=.5,below left] {$v_1$} (2,-1);
     \draw[->] (2,-1) -- node [pos=.5, below right] {$s$} (4,0);
     \draw[->]  (2,-1) -- node [pos=.5, left] {$x$} (2,1);
      \end{tikzpicture}

    \end{figure}

    \noindent By Lemma \ref{L:reduce_one_letter} the path
    $F\op{v_1}G_1\op{x}F_1$ reduces to a path $F\op{v_1\cdot x_1}F_1$
    with $x_1\preceq x$. Since $v_1\cdot x_1 \prec v$, induction
    yields that $v_1\cdot x_1\approx u_1$. So $v_1\cdot x_1\cdot
    s\approx u$ is reduced, which is only possible if $x_1=1$. Whence
    $v_1\approx u_1$ and therefore $v\approx u$.
  \end{enumerate}
\end{proof}

\begin{definition}\label{D:w_FG}
  Given $F$ and $G$ flags in a \SC $\Gamma$-space $A$, we say that
  the word $u$ \emph{connects $F$ and $G$}, if $u$ is the word of a
  reduced path from $F$ to $G$. Since $u$ is uniquely determined up
  to equivalence we denote it by $\w_A(F,G)$
\end{definition}

In order to show that simple connectedness is an elementary property,
we will first give a general description of a reduction of a flag
path.\\

For $i\neq j$, a pair of letters $s_i$, $s_j$, occurring in $v$ is
called \emph{reduced} if either $s_i$ and $s_j$ are not comparable, or
neither $s_i$ nor $s_j$ commute with all letters in between $s_i$ and
$s_j$. The word $v$ is reduced \iff all pairs of letters occurring in
$v$ are. A pair of two disjoint subwords $w_1$ and $w_2$ of a word
$v$, possibly not reduced, is \emph{reduced in $v$} if all pairs of
letters $s$ and $t$, where $s$ occurs in $w_1$ and $t$ occurs in
$w_2$, are reduced in $v$. By a sequence of generalised Absorptions
and Splittings applied to letters in $w_1$ and $w_2$, we may replace
$w_1$, $w_2$ by a pair $w_1^*$, $w_2^*$, which is reduced in the
resulting word $v^*$. We call such a process a \emph{reduction} of
$w_1$, $w_2$ in $v$. If $v$ is the word of a flag path, we call a
corresponding transformation of the path also a reduction of $w_1$,
$w_2$.

\begin{lemma}[Reduction Lemma]{\label{L:red}}
  Let $w_1\cdot w\cdot w_2$ be the (possibly non-reduced) word of a
flag path, where both $w_1$ and $ w_2$ are reduced. Then there are
words $u_1,v_1,u_2,v_2$  and a reduction of $w_1$, $w_2$ in the path
with resulting  word $w_1^*\cdot w\cdot w_2^*$, such
that:
  \begin{enumerate}[\quad $\bullet$]\setlength{\itemsep}{.5em}
  \item $w_1\approx u_1\cdot v_1$ and $v_2\cdot u_2\approx w_2$,
  \item $w_1^*=u_1\cdot x_1$ for some $x_1\preceq v_1$ and
    $x_2\cdot u_2=w_2^*$ for some $x_2\preceq v_2$,
  \item $v_1$ and $v_2$ commute with $w$,
  \item $|v_1|$ and $|v_2|$ are contained in $|w_1|\cap|w_2|$.
  \end{enumerate}
\end{lemma}
\begin{proof}
  Since both $w_1$ and $ w_2$ are reduced, we may assume that the
  words $w_1^*$ and $w_2^*$ are obtained by a sequence of generalised
  Absorptions and Splitting, each one involving a letter in the first
  word and a letter in the last word, where after every step we apply
  a reduction on both the first and the last word. It is enough to
  show by induction that, at every intermediate step of the reduction,
  the word $w_1'\cdot w\cdot w_2'$ satisfies the conclusion of the
  lemma. Start by setting $u_j=w_j$ and $v_j=x_j$ the empty word,
  for $j=1,2$.
  Assume that $u_1,v_1,x_1,v_2,u_2,x_2$ witness this at the
$i^\text{th}$ step. In particular $w_1'=u_1\cdot  x_1$ and $x_2\cdot
u_2=w_2'$, where $x_i\preceq v_i$ for $i=1,2$.

 We will treat the case of a generalised Splitting and leave to the
 reader the easier case of a generalised
 Absorption. Suppose hence
 that $w_1''\cdot w\cdot w_2''$ is obtained from $w_1'\cdot w\cdot
 w_2'$ by a generalised Splitting followed by a reduction of the first
 and last word. Then there is a letter $s$ occurring both in $w_1'$ and
 $w_2'$, which commutes with $w$ as well as with the letters of $w_1'$
 on its right (resp.\ the letters of $w_2'$ on its left). Suppose
 furthermore that the word $w_1''$ is obtained from $w_1'$ by deleting
 $s$. Note that $w_1''$ is reduced.

  The word $w_2''$ is obtained from $w_2'$ by replacing $s$ by a
  splitting $y$ of $s$ followed by a further reduction. By Lemma
  \ref{L:reduce_one_letter}, we may assume that $w_2''$ is reduced and
  obtained, up to a permutation, by replacing $s$ with some word
  $y_2\preceq y\prec s$. If $s$ occurred in $u_2$, then set $u_2'$ the
  word obtained by removing $s$ from $u_2$ as well as $v_2'=v_2\cdot
  s$ and $x_2'=x_2\cdot y_2$. If $s$ occurred in $x_2$, then replace
  $s$ by $y_2$ and leave $v_2$ and $u_2$ unchanged.

  Likewise,  modify the words $u_1$, $v_1$ and $x_1$ accordingly.
\end{proof}

\begin{definition}\label{D:PropertyE}
  Given a letter $s$ and a natural number $n$, the reduced word $w$
  \emph{satisfies} $E(s,n)$ if $|w|\subset s$ and no permutation of
  $w$ is a product of $n$ many words $u_i$,  with $|u_i|\subsetneq s$.
\end{definition}
\noindent The properties $E(s,n)$ get stronger as $n$ increases.
In
particular, the word $w$ satisfies $E(s,0)$ \iff $|w|\subset s$ and
$w\neq 1$. Similarly, the word $w$ satisfies $E(s,1)$ \iff $|w|=s$.

\begin{cor}\label{C:paarreduktion}
  Let $u=s_1\cdots s_n$ be a reduced word and $w_1,\ldots,w_n$ reduced
  words with $|w_k|=s_k$. Consider two indices $i<j$ and a reduction
  $w_i^*,w_j^*$ of the pair $w_i,w_j$ in $w_i\cdots w_j$ as in the
  Reduction Lemma \ref{L:red}. Then the following holds
  \begin{enumerate}[(1)]
  \item\label{C:paarreduktion:neues_E} If $w_i$ satisfies $E(s_i,m)$
    for some $m>0$, then $w_i^*$ satisfies $E(s_i,m-1)$, similarly for
    $w_j$ and $w_j^*$.
  \item\label{C:paarreduktion:bleibt_reduziert} Assume that $w_i$ and
    $w_j$ satisfy $E(s_i,2)$ and $E(s_j,2)$, respectively. If a pair
    $w_{i'},w_{j'}$ is already reduced in $w=w_1\cdots w_n$, then the
    corresponding pair remains reduced in $w^*=w_1\cdots w_i^*\cdots
    w_j^*\cdots w_n$.
  \end{enumerate}
\end{cor}
\begin{proof}
  In order to prove $(\ref{C:paarreduktion:neues_E})$, choose
  $u_i,v_i, x_i, u_j, v_j, x_j$ as in the Reduction Lemma. Since
  $|v_i|$ is contained in $s_i\cap s_j$ and $v_i$ commutes with
  $s_{i+1}\cdots s_{j-1}$, then $|v_i|$ is a proper subset of $s_i$,
  for otherwise the pair $s_i,s_j$ would not be reduced in $u$. As
  $w_i\approx u_i\cdot v_i$, if $w_i$ has property $E(s_i,m)$, then
  $u_i$ has property $E(s_i,m-1)$, and so does $w_i^*$.

  For $(\ref{C:paarreduktion:bleibt_reduziert})$, assume that $w_i$
  and $w_j$ satisfy $E(s_i,2)$ and $E(s_j,2)$, respectively. Therefore
  $w_i^*$ and $w_j^*$ have property $E(s_i,1)$ and $E(s_j,1)$,
  respectively, so $|w_i^*|=s_i$ and $|w_j^*|=s_j$. Suppose now
  that the pair $w_{i'},w_{j'}$ is reduced in $w$. Since the words $w_i^*$
  and $w_j^*$ commute with the same letters as $w_i$ and $w_j$,
  respectively, the pair $w_{i'},w_{j'}$ remains reduced in $w^*$ if
  $\{i',j'\}$ and $\{i,j\}$ are disjoint. By symmetry, it suffices to
  consider the following three other cases:

  Case $i'<j'=i$. We have to show that the pair $w_{i'},u_i\cdot x_i$
  is reduced in $w^*$ if the pair $w_{i'},u_i\cdot v_i$ is reduced in
  $w$. This follows easily from $|x_i|\subset s_i=|u_i|$, since $w_i$
satisfies $E(s_i,2)$.

  Case $i=i'<j'<j$. Here we have to show that the pair $u_i\cdot
  x_i,w_{j'}$ is reduced in $w^*$ if the pair $u_i\cdot v_i,w_{j'}$ is
  reduced in $w$. This follows easily from the fact that $v_i$ and
  $x_i$ commute with $w_{j'}$.

  Case $i=i'<j<j'$. Again we have to show that the pair $u_i\cdot
  x_i,w_{j'}$ is reduced in $w^*$ if the pair $u_i\cdot v_i,w_{j'}$ is
  reduced in $w$. This follows easily from
  $|x_i|,|v_i|,|v_j|\subset|u_j|$.

\end{proof}

\begin{prop}\label{P:hauptprop}
  Let $u=s_1\cdots s_n$ be a reduced word. Given reduced flag paths
  $F_{i-1} \op{w_i} F_i$ such that each $w_i$ has property $E(s_i,n)$,
  then the path \[F_0 \op{w_1} F_1 \cdots F_{n-1} \op{w_n} F_n\] has a
  reduction of length $\geq n$.
\end{prop}
\begin{proof}
  Choose any enumeration of all pairs $i<j$ of indices between $1$ and
  $n$ and apply the Reduction Lemma to each pair in order. Observe
that $k$ occurs in at most $n-1$ reductions. At every step of the
reduction, the resulting words satisfy $E(s_k,2)$, so the resulting
path
\[F_0 \op{w_1^*} F_1^* \cdots F_{n-1}^*\op{w_n^*} F_n,\]
is reduced, by Corollary
\ref{C:paarreduktion} $(\ref{C:paarreduktion:bleibt_reduziert})$. None
of the words $w_i^*$ is trivial, by Corollary \ref{C:paarreduktion}
$(\ref{C:paarreduktion:neues_E})$.

\end{proof}

Together with Remark \ref{R:paths_exists}, the
previous proposition will imply the following (a priori) stronger
form.

\begin{remark}
  Let $u=s_1\cdots s_n$ be a reduced word and $w_1,\ldots,w_n$ reduced
  words with property $E(s_i,n)$. Then every reduct of $w_1\cdots w_n$
  has length at least $n$.
\end{remark}

\begin{theorem}\label{S:SC_elementar}
  Simple connectedness is an elementary property for\/ $\Gamma$-spaces.
\end{theorem}
\begin{proof}
  For each natural number $n$ and letter $s$, consider the following
  elementary property of two given flags $F$ and $G$: We have that
$F\sim_s G$, but there exist no proper subsets $A_1,\ldots,A_n$ of $s$ and
flags $F_1,\ldots,F_{n-1}$ such that
  \[F\sim_{A_1}F_1\sim_{A_2}\cdots \sim_{A_{n-1}}F_{n-1}\sim_{A_n} G.\]
  We denote this by $F\op{s,n}G$. Observe that, if $F\op{s,n}G$, then
there is a path $F\op{w}G$ for some reduced $w$, which satisfies property
$E(s,n)$. Indeed, since $F\sim_s G$, there exists a reduced word $w$ with
support contained in $s$ connecting $F$ to $G$. Any such word
satisfies $E(s,n)$.

If suffices to show that a $\Gamma$-space $M$ is \SC \iff for all
natural numbers $n$ and all non-trivial reduced words $u=s_1\cdots s_n$,
there is no sequence
  $F_0\op{s_1,n}F_1\cdots F_{n-1}\op{s_n,n}F_n=F_0$.

 Clearly, right-to-left is obvious, since $F\op{s}G$ implies
$F\op{s,n}G$, by Lemma \ref{L:modlemma}. Suppose now that $M$ is \SC and
let $F_0\op{s_1,n}F_1\cdots  F_{n-1}\op{s_n,n}F_n$ be a weak flag path
for some non-trivial reduced word $u=s_1\cdots s_n$. By the above
discussion, there are
  words $w_i$, each satisfying property $E(s_i,n)$, respectively, such that
  \[F_0\op{w_1}F_1\dotsb
  F_{n-1}\op{w_n}F_n.\] Proposition \ref{P:hauptprop} yields that this
  path has a reduction of length at least $n$, so $F_0\not=F_n$, since
$M$ is \SC.
\end{proof}

Simple connectedness allows us to generalise \cite[Remark
  4.9]{BMPZ13}, which will be needed for the proof of Proposition
\ref{P:nice=nosplit}.

\begin{lemma}\label{L:xy-Kreis}
  Given two adjacent colours $\gamma$ and $\delta$, if $M$ is a \SC
  $\Gamma$-space, the subgraph
  $\A_{\gamma,\delta}(M)=\A_\gamma(M)\cup\A_{\delta}(M)$ has no
  non-trivial circles.
\end{lemma}

\begin{proof}
  Since any edge in $\A_{\gamma,\delta}$ lies within a flag in $M$, by
  property $(2)$ of Definition \ref{D:colsp}, a path with no
  repetitions in $\A_{\gamma,\delta}$ induces a (possibly non-reduced)
  flag path of the form $$ F_0 \op{\text{$w_1$}} F_1 \op{\text{$w_2$}}
  \cdots F_n,$$ where the reduced words $w_1,\ldots,w_n$ have the
  following properties: \renewcommand{\theenumi}{\alph{enumi}}
\begin{enumerate}
 \item  $\delta \in |w_{2k+1}|\subset \Gamma\setminus\{\gamma\}$,
 \item $\gamma \in |w_{2k}|\subset \Gamma\setminus\{\delta\}$.
\end{enumerate}

\noindent By repeatedly applying Lemma \ref{L:red} to each pair
$w_i,w_j$ for $i\neq j$, it is easy to see that the above conditions
remain in the reduct $w^*_1 \cdots w^*_n$ of the word $w_1\cdots w_n$.
In particular, the word $w^*_1 \cdots w^*_n$ is not trivial and thus
$F_0\neq F_n$. Hence, the original path in $\A_{\gamma,\delta}$ was
not closed.

\end{proof}

\section{The Theory $\psg$}\label{S:PS}

\begin{definition}\label{D:axiome}
  In the language of graphs enriched with unary predicates for the
  colours $\{\A_\gamma\}_{\gamma\in\Gamma}$, let the theory $\psg$ be
  a collection of sentences stating that the structure is a
  $\Gamma$-space with the following properties:
  \begin{enumerate}
  \item\label{D:axiome:SC} simple connectedness,
  \item\label{D:axiome:unendlich} for any colour $\gamma$ in $\Gamma$,
    the $\sim_\gamma$-class of any flag $G$ is infinite (Observe that the
relation $\sim_\gamma$ is definable in this language).
  \end{enumerate}
\end{definition}
Axiom (\ref{D:axiome:SC}) is a first-order property, by Theorem
\ref{S:SC_elementar}. Clearly, so is Axiom (\ref{D:axiome:unendlich}).
The $\Gamma$-space $\mo$, as defined on page
\pageref{seite:definition_m_0}, is a
model of $\psg$ by Theorem \ref{T:M0_SC}, so $\psg$ is consistent.

The rest of this section is devoted to proving
the completeness of $\psg$. \\

We first generalise \cite[Definition 4.3]{BMPZ13}.
\begin{definition}
  Fix some letter $s$, and let $F$ be a flag in a $\Gamma$-graph $A$.
  Create a new flag $F^*=\{f^*_\gamma\}_{\gamma\in\Gamma}$ which
  agrees with $F$ on the colours of $\Gamma\setminus s$ but
  $f^*_\gamma\notin A$ for $\gamma\in s$. We define a
  $\Gamma$-graph \[A(F^*)\] with vertices $A\cup F^*$ and edges those
  of $A$ and of $F^*$. A $\Gamma$-graph $B\supset A$ is a \emph{simple
    extension} of $A$ of \emph{type} $(s,F)$ if it is $A$-isomorphic
  to $A(F^*)$.
\end{definition}
\noindent Note that $F^*\sop{s}F$ by construction.

\begin{remark}\label{R:A*space}
  If $A$ is a $\Gamma$-space, then so is $A(F^*)$.
\end{remark}

These simple extensions generalise those simple extensions defined
after Remark \ref{R:unique_b*}, as the following easy remark shows.
\begin{remark}\label{R:newflags}
  Let $\partial s= s\cup \{ \delta \in \Gamma\,| \text{ adjacent to
    some } \gamma \in s\}$ denote the set of all $\gamma$ in $\Gamma$
  which do not commute with $s$. Let $H$ be a flag in $A$ which agrees
  with $F$ on the colours in $\partial s$. For each $\gamma\in s$, replace
  $h_\gamma$ in $H$ by $f^*_\gamma$ in order to obtain a
  flag $H^*$ of $A(F^*)$. It is easy
  to see, since $s$ is connected, that that this construction defines a a
1-to-1-correspondence between the flags $H$ of $A$ with
$H\sim_{\Gamma\setminus\partial
    s}F$ and the new flags $H^*$ of $A(F^*)$. Note that $H^*$ is
  uniquely determined by
  \[H^*\sim_{\Gamma\setminus \partial s}F^*\text{ and } H^*\sim_sH.\]
\end{remark}

\vspace{1em}

In order to prove that the theory $\psg$ is complete, we will need the
appropriate interpretation of a strongly connected subset in this
context.

\begin{definition}\label{D:nice}
  A non-empty subgraph $D$ of a $\Gamma$-space $M$ is \emph{nice} if
  it satisfies the following conditions:
  \begin{itemize}
  \item Any point $a$ in $D$ lies in a flag in $D$.
  \item Given flags $F$ and $G$ in $D$ and a letter $s$, if $F\op{s}
    G$ in $D$, then $F\op{s} G$ in $M$.
\end{itemize}
\end{definition}

Any nice set  is the union of all the flags contained in it. Niceness
is a transitive property. Proposition \ref{P:Eindwort}
yields that a non-empty subset $D$ of a \SC $\Gamma$-space $M$ is
nice \iff the following hold:
\begin{itemize}
\item Any point $a$ in $D$ lies in a flag in $D$.
\item Given flags $F$ and $G$ in $D$ and a reduced word $u$, if
  $F\op{u}G$ in $M$, then there exists such a path in $D$ with the same
word.
\end{itemize}
\noindent In particular, if $D$ is nice in $M$, then $D$ is \SC
whenever $M$ is.

\begin{remark}\label{R:nice}
  \begin{enumerate}
  \item\label{R:nice:A*} The $\Gamma$-space $A$ is nice in $A(F^*)$.
  \item\label{R:nice:space} A nice subset of a $\Gamma$-space is
    itself a $\Gamma$-space.
  \end{enumerate}
\end{remark}

\begin{proof}
  \noindent For $(\ref{R:nice:A*})$, given flags $G$ and $H$ in $A$
with $G\op{t}H$ in $A$, suppose there is a splitting $x=t_1\cdots t_n$
of
$t$ such that $$G=F_0\sop{t_1}F_1\cdots
  F_{n-1}\sop{t_n} F_n=H$$ in $A(F^*)$. We replace each $F_i$ by a
flag
$F'_i$ in $A$ as follows: If $F_i$ belongs to $A$, set $F_i'=F_i$.
Otherwise, by
  Remark \ref{R:newflags}, the flag $F_i$ has the form $H^*_i$ and we
set
  $F_i'=H_i$. Note that $F_{i-1}'\sim_{t_i}F_i'$, so
  $G$ and $H$ can be connected in $A$ be a weak flag path
  whose word is $\preceq$-smaller than $t_1\cdots t_n$, contradicting
$G\op{t} H$.

 \noindent In order to show $(\ref{R:nice:space})$, consider two
elements $a$  and $b$ in a nice subset $A$ of $M$, connected by a
flag $G$ in $M$. Choose
flags $F$
and
$H$ in $A$ containing $a$ and $b$ respectively. Let $\gamma$ be the
colour
of $a$ and $\delta$ the colour of $b$. Since $F
  \sim_{\Gamma\setminus\{\gamma\}} G\sim_{\Gamma\setminus\{\delta\}}
  H$, there is a reduced path $F\op{u}G'\op{v}H$ in $M$ such that
  $\gamma$ does not occur in $u$ and $\delta$ does not occur in $v$.
By
  niceness, we may therefore assume that $G'$ belongs to $A$. Clearly
  $G'$ contains $a$ and $b$.
\end{proof}

We can now prove the analogue version of \cite[Lemma 4.21]{BMPZ13}.

\begin{prop}\label{P:nice=nosplit}
 Given a flag $F$ in a nice subset $A$ of a  \SC $\Gamma$-space $M$,
and a flag $F^*$ in $M$ which is $s$-equivalent to $F$ for some
letter $s$, the following are equivalent:
\renewcommand{\theenumi}{\alph{enumi}}
\begin{enumerate}
  \item\label{P:nice=nosplit:niceagain} The $\Gamma$-graph $A\cup F^*$
    is a simple extension of $A$ of type $(s,F)$ and is nice in $M$.
 \item\label{P:nice=nosplit:nosplit} Whenever $G$ is a flag in $A$ and
   $x$ a splitting of $s$, then $F^*\not\sop{x} G$ in $M$.
\end{enumerate}
\end{prop}

\begin{proof}
  \noindent $(\ref{P:nice=nosplit:niceagain})\to
  (\ref{P:nice=nosplit:nosplit})$: Set $B=A\cup F^*$. If $$F^* \sop{x}
  G$$ in $M$, then we cannot have that $F^* \op{s} G$ in $B$, for $B$
  is nice. Thus, there is a splitting $x'$ of $s$ such that
  $F^*\op{x'} G$ in $B$. All flags in $B$ which are $s$-equivalent to
  $F$ are either $F^*$ itself or contained in $A$, by Remark
  \ref{R:newflags}, so $F^*\sim_{t} G'$ for some $G'$ in $A$, where
  $t\subsetneq s$ is the first letter of $x'$. This is impossible, for
  no vertex of $F^*$ with colour in $s$ lies in $A$.

  \noindent $(\ref{P:nice=nosplit:nosplit})\to
  (\ref{P:nice=nosplit:niceagain})$: Since $F$ lies in $A$, the
  hypothesis implies that $F^*\op{s} F$. We will first show that, for
  any flag $G$ in $A$, if $F^*\op{u} G$ in $M$, then $s\preceq u$. We
  may assume that $u$ is reduced. Since $A$ is nice, there is a
  reduced path $F\op{v}G$ in $A$, which remains reduced in $M$. The
  word $u$ is thus a reduct of $s\cdot v$. If in the reduction
  splitting ever occurs, it produces a flag in $A$ which connects to
  $F^*$ by a splitting of $s$, contradicting the assumption. Thus, the
  reduction involves only commutation and possibly absorption of $s$
  by $u$. Hence $s\preceq u$.

  In particular, for any flag $G$ in $A$ with $F^*\op{u} G$, there
  exists a letter $t$ in $u$ containing $s$. Actually, by Lemma
  \ref{L:Flagpath} it suffices to assume $F^*\sop{u} G$.

  In order to show that $B=A\cup F^*$ is a simple extension of $A$, we
  need to show that there is no new edge consisting of an element $b$
  in $\A_{\gamma}(B)\setminus A$ and some $a$ in
  $\A_{\delta}(A)\setminus F$. Suppose otherwise that there exists a
  flag $F'$ in $M$ passing through $a$ and $b$. Take a flag $G$ in $A$
  containing $a$.

  Note that $\gamma$ is in $s$. Suppose first that $\delta$ lies in
  $s$ as well. Since $F^* \sim_{\Gamma\setminus\{\gamma\}}
  F'\sim_{\Gamma\setminus\{\delta\}} G$, we obtain reduced words $u$
  and $v$ such that $\gamma$ does not occur in $u$, the colour
  $\delta$ does not occur in $v$ and

  \[ F^* \op{u} F' \op{v} G.\]

  The reduction $F^*\op{w} G$ in $M$ satisfies that every letter in
  $w$ does not contain either $\gamma$ or $\delta$, contradicting the
  previous discussion.

  If $\delta$ does not lie in $s$, then, with the choice of flags as
  before, we obtain the following path

  \[ F \op{s}F^* \op{u}
  F'\op{v} G\]

  As before, since $A$ is nice, this implies that $F\op{w} G$ in $A$,
  where each letter in $w$ avoids either $\gamma$ or $\delta$. This
  induces a path in $\A_{\gamma,\delta}(A)$ between $a'$ and $a$,
  where $a'$ is the $\delta$-vertex of $F$. This, together with the
  connection $a-b-a'$, yields a non-trivial circle, contradicting Lemma
  \ref{L:xy-Kreis}.

  Let us now show that $B$ is nice in $M$. Given  flags $G_1\op{t}G_2$
  in $B$, we distinguish the following cases:

  \begin{itemize}
  \item Both flags lie in $A$. Then $G_1\op{t} G_2$ also in $A$, and
    thus in $M$, since $A$ is nice.
  \item None of the flags lies in $A$. By Remark
\ref{R:newflags}, we have $G_1\sim_{\Gamma\setminus\partial s} G_2$
and whence  $t\subset \Gamma\setminus\partial s$. Thus we find $H_1$ and
$H_2$
    in $A$ such that $H_1\sim_{\Gamma\setminus\partial s} H_2$ and
    $G_i\sim_s H_i$ for $i=1,2$. This implies $H_1\op{t}H_2$ in $B$
    and also in $A$. Therefore $H_1\op{t}H_2$ in $M$, for $A$ is nice,
    which implies that $G_1\op{t}G_2$ in $M$ as well.
  \item Exactly one flag, say $G_1$, is not fully contained in
$A$. Again by
    Remark \ref{R:newflags}, we have that $F^*\op{w}G_1$ for a word
    $w$ which commutes with $s$ and a flag $H_1$ in $A$ with
    $G_1\sop{s}H_1$. Since $F^*\sop{w\cdot t}G_2$, some letter of
    $w\cdot t$ must contain $s$, so $s\subset t$. If $s=t$ but $G_1\op{x}
G_2$ in $M$ for some $x\prec t=s$, then $F^*\op{w}G_1\op{x} G_2$ in $M$,
whose reduction yields a word where no letter contains $s$. Thus
    $G_1\op{t}G_2$ in $M$. Otherwise, if
    $s\subsetneq t$, then $H_1\op{t}G_2$ in $B$. Since $A$ is nice, we
    have $H_1\op{t}G_2$ in $M$, which implies $G_1\op{t}G_2$ in $M$.
  \end{itemize}
\end{proof}

In particular, setting $s=\{\gamma\}$, for $\gamma$ in $\Gamma$, we
deduce the following result.
\begin{cor}\label{C:one_pt_ext}
  Given a flag $G$ in a nice subset $A$ of a \SC $\Gamma$-space $M$
  and $\gamma$ in $\Gamma$, if the flag $F$ is $\{\gamma\}$-equivalent
  to $G$ and the $\gamma$-vertex of $F$ does not lie in $A$, then the
  set $B=A\cup F$ is nice and a simple extension of $A$ of type
  $(\{\gamma\},G)$.
\end{cor}

The next proposition shows that a \SC space is the increasing union of
simple extensions of nice subsets. sets (\emph{cf.}
\cite[Theorem 4.22]{BMPZ13}).

\begin{prop}\label{P:nice-bandf}
  Given a nice subset $A$ of a \SC $\Gamma$-space $M$ and $b$ in $M$,
  there exists a nice subset $B$ containing $b$ which can be obtained
  from $A$ by a finite number of simple extensions.
\end{prop}
\begin{proof}
  Given a flag $F$  in $M$ containing $b$, choose a reduced path
  \[F=F_0\op{s_1}F_1\cdots \op{s_n}F_n\]
  connecting $F$ with a flag $F_n$ in $A$ such that the word
  $u=s_1\cdots s_n$ is $\prec$-minimal. We prove the claim by
  $\prec$-induction on $u$. If $u=1$, there is nothing to show.
  Otherwise, minimality of $u$ implies that there is no path which
  connect $F_{n-1}$ to a flag in $A$ whose word is a splitting of
  $s_n$. It follows from Proposition \ref{P:nice=nosplit} that
  $A'=A\cup F_{n-1}$ is a simple extension of $A$ of type $(s_n,F_n)$,
  so $A'$ is nice in $M$. Now $F$ can be connected to some flag in
  $A'$ by a reduced word path, whose word $u'$ is $\prec$-minimal
  such, with $u'\preceq s_1\cdots s_{n-1}\prec u$. By induction, the
  element $b$ is contained in a nice set $B$ which can be obtained
  from $A'$ (and thus, from $A$) by a finite number of simple
  extensions.
\end{proof}

\begin{lemma}\label{L:Zus}
Let $s$ be a letter in $\Gamma$ which is not a singleton. There
are $\gamma$ and $\gamma'$ in $s$ distinct such that both $s\setminus
\{\gamma\}$ and $s\setminus\{\gamma'\}$ are connected.
\end{lemma}

\begin{proof}
  By repeatedly removing edges, it is enough to prove it for \emph{a
    spanning tree} $s$, that is, whenever we remove an edge between
  two points in $s$, the resulting graph is no longer connected. In
  particular, such an $s$ contains no cycles. The assertion now
  follows, since any non-trivial tree has at least two extremal
  points.
\end{proof}

\begin{cor}\label{C:equiv_forms}
  Given a letter $s$ and a flag $G$ contained in some finite nice
  subset $A$ of an $\omega$-saturated model $M$ of $\psg$, then $M$
  contains a simple extension of $A$ of type $(s,G)$ which is nice in
  $M$.
\end{cor}

\begin{proof}
   For $s=\{\gamma\}$, pick any flag $G^*$ which is
   $\gamma$-equivalent to $G$ and its $\gamma$-vertex does not lie in
   (the finite set) $A$, by property (\ref{D:axiome:unendlich}). The
   set $A\cup G^*$ is nice in $M$ and a simple of $A$ type $(s,G)$, by
   Corollary \ref{C:one_pt_ext}.

  Suppose now that $|s|\geq 2$. By Proposition \ref{P:nice=nosplit}
  and saturation, it is enough to produce, for every $n$, a flag $G_n$
  with $G_n\sop{s}G$ and $G_n\not\sop{x}G'$, whenever $G'$ is a flag
  in $A$ and $x$ a splitting of $s$ of length at most $n$.

  By Lemma \ref{L:Zus}, find two subletters $s_0$ and $s_1$ of $s$ of
  cardinality $|s|-1$ and not commuting with each other. Since $A$
  contains only finitely many flags, simple connectedness yields an
  upper bound $N$ for the length of the word of any reduced flag path
  between any two flags in $A$. Set $A_0=A$ and $G_0=G$. By induction
  on $|s|$, there is a sequence of pairs $\{(G_i,A_i)\}_{i\leq N+n}$
  such that $G_i$ is a flag in the (finite) nice set $A_i$ and
  $A_{i+1}=A_i \cup G_{i+1}$ is a simple extension of type
  $(s_{f(i)},G_i)$, where $f(i)$ in $\{0,1\}$ is the residue of $i$
  modulo $2$. The flag path

  $$G_{N+n+1}\op{s_{f(N+n)}} G_{N+n}\cdots G_{1}\op{s_0} G_0$$

  \noindent is reduced, since $s_0$ and $s_1$ do not commute.

  Clearly $G_{N+n+1}\sop{s}G_0$. Suppose there is some flag $G'$ in $A$
such that \[G_{N+n+1}\sop{x}G'\] in $A$ for some splitting $x$ of $s$ of
length at most $n$. Reducing this path, we obtain a reduced path
$G'\op{u}G$, where $u$ is also a splitting of $s$. By
  niceness of $A$, we may assume that this path lies in $A$, so $u$
  has length at most $N$. Proposition \ref{P:Eindwort} implies that $x
  \cdot u \str s_{f(N+n)}\cdots s_0$. However, at every step of the
  reduction, the number of letters of size exactly $|s|-1$ is bounded
  by $N+n$, contradicting our choice of $ s_{f(N+n)}\cdots s_0$.
\end{proof}

We can now conclude that the theory $\psg$ is complete and
that the type of a nice set is determined by
its  quantifier-free type.

\begin{theorem}\label{T:el}
Any two $\omega$-saturated models of $\psg$ have the back-and-forth
property with respect to the collection of partial isomorphisms
between finite nice substructures. In particular, any partial
isomorphism $f:A\to A'$
between two finite nice subsets of two models of $\psg$ is
elementary. The theory $\psg$ is complete.
\end{theorem}

\begin{proof}
  Let $M$ and $M'$ be two $\omega$-saturated models and let $f:A\to
  A'$ be a partial isomorphism between two finite nice substructures.
  Given $b$ in $M$, Proposition \ref{P:nice-bandf} yields a nice
  subset $B$ containing $A\cup\{b\}$ such that $A\leq B$ in finitely
  many steps. By $\omega$-saturation of $M'$ and Corollary
  \ref{C:equiv_forms} (finitely many times), we obtain a nice subset
  $B'$ of $M'$ containing $A'$ such that $f$ extends to isomorphism
  between $B$ and $B'$.

  Since any model $M$ is nice in any elementary extension, replacing
  the models by appropriate saturated extensions, we produce a
  back-and forth system. Completeness of $\psg$ then follows, since
  any two flags have the same quantifier-free type.
\end{proof}

\begin{cor}\label{C:nicetype}
The type of a nice set $A$ is determined by its quantifier-free type.
\end{cor}
\begin{proof}
 For finite sets, this follows from Theorem \ref{T:el}. For infinite
nice sets, note that they are direct unions of finite
nice subsets.
\end{proof}

\begin{cor}\label{C:stab}
The theory $\psg$ is $\omega$-stable and the model $\mo$ is the
unique (up to isomorphism) countable prime model.
\end{cor}

\begin{proof}
  In order to show that $\psg$ is $\omega$-stable, we need to count
  $1$-types over a countable subset of $A$, which we may assume nice
  inside a given saturated model (\emph{cf.} \cite[Theorem
    5.2.6]{TZ12}). Every simple extension of $A$ is uniquely
  determined, up to $A$-isomorphism, by its type $(s,G)$, by Corollary
  \ref{C:nicetype}. Therefore, if $A$ is countable, there are, up to
  $A$-isomorphism, only countably many simple extensions. Now, every
  $1$-type is realised in a finite tower of simple extensions over
  $A$, by Proposition \ref{P:nice-bandf}, so there are only countably
  many types, as desired.

  In order to show that the countable model $\mo$ is the the prime
  model of $\psg$, it suffices to show that it is constructible. It
  follows from the proof of Theorem \ref{T:M0_SC} that the only words
  of reduced paths in $\mo$ are finite products of singletons. Given
  $\gamma$ in $\Gamma$ and a flag $G$ in a nice subset $A$, the type
  over $A$ of the simple extension $B=A \cup F$ of type $(\{\gamma\},
  G)$ is determined by its quantifier-free type, which amounts to
  saying, by Corollary \ref{C:one_pt_ext}, that $F$ and $G$ are
  $\gamma$-equivalent but that the $\gamma$-vertex of $F$ does not lie
  in $A$. Therefore, if $A$ is finite, the type of $B$ over $A$ is
  isolated.
\end{proof}

Uniqueness of prime models and Theorem \ref{T:biinterpretierbar} yield
the uniqueness result in Corollary \ref{C:b_unten_0}.

\begin{remark}\label{R:paths_exists}
  Let $M$ be an $\omega$-saturated model of $\psg$. For every word
$v$,  there is a path $F\op{v}G$ in $M$. Whenever a (possibly
non-reduced)   word $u$ can be reduced to $v$, then there is a flag
path from $F$  to $G$ with word $u$.
\end{remark}

\section{Non-splitting reductions}\label{S:NS_Reductions}

In order to describe the geometrical complexity of $\psg$, we will
need several auxiliary results on the combinatorics of reduction of
words when no splitting occurs, generalising some of the results of
\cite{BMPZ13}. For the sake of self-containment, we will provide,
whenever possible, different proofs.

\begin{definition}\label{D:beginning}
  A letter $s$ is a \emph{beginning} (resp.\ \emph{end}) of the word
  $u$ if $u\approx s\cdot v$ (resp.\ $u\approx v\cdot s$). The
  \emph{initial segment} (resp.\ \emph{final segment}) of $u$ is the
  commuting subword whose letters are beginnings (resp.\ ends) of $u$.
\end{definition}

By abuse of the language, we say that the word $u$ is an \emph{initial
  subword} of $v$ if $u$ is an initial subword (in the proper sense)
of some permutation of $v$. Likewise for \emph{final subword}. The
initial segment of $v$ is the largest commuting initial subword of
$v$. Inductively on the sum of their lengths, it is easy to see that
any two words $u$ and $v$ have a \emph{largest common initial
  subword}, resp.\ a \emph{largest common final subword}, none of
which need be commutative.

Common initial subwords can be removed, as seen easily.
\begin{lemma}\label{L:init_subw_kuerzen}
  If $u\cdot v'\approx u\cdot v''$, then $v'\approx
v''$. Likewise, if $ v'\cdot u\approx v''\cdot u$, then $v'\approx
v''$.
\end{lemma}

\begin{lemma}\label{L:initial}
  Let $a$ be a final subword of $c\cdot d$ such that every end of $a$
commutes with $d$. Then $a$ and $d$ commute
  and $a$ is a final subword of $c$.
\end{lemma}

\begin{proof}
  Proceed by induction on the length of $a$. If $a$ is the trivial
  word, there is nothing to show. Otherwise, write $a= s\cdot a_1$.
By induction, the subword $a_1$ commutes with $d$ and is
a final subword of $c\approx c'\cdot a_1$. Since
$c\cdot d\approx c' \cdot d \cdot a_1$,  Lemma
\ref{L:init_subw_kuerzen} implies that $s$ is an end of $c'\cdot d$.

If $s$ occurs in $d$, then $s$ and $a_1$ commute, so $s$ is an end of
$a$ and hence it must commute with $d$, by hypothesis, which is a
contradiction, since no letter commutes with itself. Therefore, the
letter $s$ must occur in $c'$, so in particular it must commute with
$d$. Thus, the word $a$ commutes with $d$ and is a final subword
of $c$, as desired.
\end{proof}

\begin{definition}
  A \emph{non-splitting reduction} of the word $u$, denoted by $u\to
v$, is a reduced word $v$
  obtained from $u$ by a reduction process, where only {\sc
    Commutation} and {\sc Absorption} occur (\emph{cf.}\ Definition
  \ref{D:Op_redpath}).
\end{definition}
By induction on the length of the reduction $u\to v$, the following can be
easily shown.
\begin{lemma}\label{L:non-split_vs_preq}
  If $v$ is a non-splitting reduction of $u$, then every $u'\preceq u$
  has a non-splitting reduction $v'\preceq v$.
\end{lemma}

\begin{cor}[{\emph{cf.\ }\cite[Proposition
  5.3]{BMPZ13}}]\label{C:non-split-unique} Every word admits exactly
  one non-splitting reduction, up to permutation.
\end{cor}
\begin{proof}
  Assume that $v_1$ and $v$ are two non-splitting reducts of $u$. In
particular $v_1\preceq u$, so  $v_1\to v'$ for some $v'\preceq
  v$, by Lemma \ref{L:non-split_vs_preq}. Thus $v'$ must be equivalent
to $v_1$, and hence $v_1\preceq  v$. Similarly, we obtain $v\preceq
v_1$, so $v_1\approx v$.
\end{proof}
\begin{notation}
  Given reduced words $u$ and $v$, we denote by $[u\cdot v]$ the
  \emph{non-splitting reduct} of $u\cdot v$, which is defined up to
  permutation.
\end{notation}
\begin{cor}[{\emph{cf.\ }\cite[Lemma 5.29]{BMPZ13}}]\label{C:Ord_Mon}
  Given reduced words $u$, $v$ and $x$ with $x\preceq v$, then
  $[u\cdot x]\preceq [u\cdot v]$.
\end{cor}
\begin{cor}\label{C:kleinergleich}
  If $u$ and $v$ are reduced words, then  $u\preceq [u\cdot
    v]$.\qed
\end{cor}

Recall (\emph{cf.}\
Definition \ref{D:abs}) that a letter $t$ is properly
left-absorbed, resp.\ right-absorbed, by the word $s_1\cdots
  s_n$ \iff $t$ is properly contained in some $s_i$ and commutes with
$s_1\cdots s_{i-1}$, resp.\ with $s_{i+1}\cdots
  s_{n}$.
\begin{prop}\label{P:Dec}(Symmetric Decomposition Lemma,
\emph{cf.}\ \cite[Corollary 5.23]{BMPZ13})
 Given two reduced words $u$ and $v$, there are unique
  decompositions (up to permutation):
  \begin{align*}
    u&=u_1\cdot u'\cdot w&
    w\cdot v'\cdot v_1&=v,
  \end{align*}
  such that:
  \renewcommand{\theenumi}{\alph{enumi}}
  \begin{enumerate}
  \item $w$ is a commuting word,
  \item $u'$ is properly left-absorbed by $v_1$,
  \item $v'$ is properly right-absorbed by $u_1$,
  \item $u'$, $w$ and $v'$ pairwise commute,
  \item $u_1\cdot w\cdot v_1$ is reduced.
  \end{enumerate}
  Furthermore,
  \[[u\cdot v]=u_1\cdot w\cdot v_1.\]
\end{prop}
\begin{proof}
  We show the existence of such a decomposition by induction on the
  sum of the lengths of $u$ and $v$. If $u \cdot v$ is already
  reduced, then set $u_1=u$ and $v_1=v$, and define $v'$, $u'$ and $w$
  to be the trivial word.

  Otherwise, up to permutation and changing the roles of $u$ and $v$ , we
may assume that $u=\tilde u\cdot s$, where $s$ is left-absorbed by
$v$. In
  particular $[s\cdot
v]=v$. By induction, we
  find words $\tilde u_1$, $\tilde u'$, $\tilde w$, $\tilde v'$ and
  $\tilde v_1$ with the desired properties such that

 \begin{align*}
    \tilde u&=u_1\cdot \tilde u'\cdot \tilde w&
    \tilde w\cdot v'\cdot \tilde v_1&= v.
  \end{align*}

  Observe that $s$ cannot be left-absorbed by $\tilde w$ nor by
$\tilde
  v'$, for $u$ is reduced. Thus $s$ commutes with $\tilde
  w\cdot \tilde v'$ and is  absorbed by $\tilde v_1$. There
are
  two possibilities:
  \begin{enumerate}
  \item The letter $s$ is properly absorbed by $\tilde v_1$. Set
    $u'=\tilde u'\cdot s$, $w=\tilde w$ and $v_1=\tilde v_1$.
  \item Up to permutation, the letter $s$ is a beginning of $\tilde
    v_1=s\cdot v^*_1$. Set $u'=\tilde u'$, $w= \tilde w\cdot s$ and
    $v_1=v^*_1$.
  \end{enumerate}

  Let us finish by showing the uniqueness of the above decomposition. Note
  first that $[u\cdot v]=u_1\cdot w\cdot v_1$, since $u\cdot v\to
  u_1\cdot w\cdot v_1$. The word $w$ is exactly the intersection of the
  final segments of $u$ and $v\inv$, since $u_1\cdot w\cdot v_1$ is
  reduced. Furthermore, the word $u_1\cdot w$ is the largest common
  initial subword of $u$ and $[u,v]$, for otherwise, there exists a
  letter $s$ with $v_1=s\cdot \tilde v_1$ and $u'=s\cdot\tilde u'$,
  contradicting that $u'$ is properly left-absorbed by $v_1$.
  Similarly, the word $w\cdot v_1$ is the largest common final subword
  of $v$ and $[u,v]$, providing the desired result.
\end{proof}
\begin{cor}[{\emph{cf.\ }\cite[Corollary 5.14]{BMPZ13}}]
  \label{C:absorb=stab}
  Let $u$ and $v$ be reduced words. Then $u$ is left-absorbed by $v$
  \iff $[u\cdot v]=v$.
\end{cor}
\begin{proof}
  Note that $[u\cdot v]=v$ \iff $u_1\cdot w\cdot v_1\approx
  v'\cdot w\cdot v_1$. Lemma \ref{L:init_subw_kuerzen} implies
  $u_1\approx v'$. So $u_1$ is properly right-absorbed by itself, which
   can only happen if  $u_1=1$. Thus $u\approx u'\cdot w $ is
  left-absorbed by $v$.
\end{proof}
\begin{cor}[{\emph{cf.\ }\cite[Corollary 5.22]{BMPZ13}}]
  \label{C:kommwort}
  A reduced word is commuting \iff it left-absorbs itself.
\end{cor}
\begin{proof}
 Suppose first that we have two reduced words $u$ and $v$, such that $u$
right-absorbs $v$ and $v$ left-absorbs $u$. The proof of
\ref{C:absorb=stab} yields both  $u_1=v'=1=v_1=u'$, so $u=v=w$ is
commuting. The statement now follows easily, since Corollary
\ref{C:absorb=stab} yields that if $u$ left-absorbs itself, then it also
right-absorbs itself.
\end{proof}
\begin{cor}\label{C:proper_absorb}
  If the reduced word $u$ is left-absorbed by the reduced word $v$, we
  can write (up to permutation)~:
  \begin{align*}
    u&=u'\cdot w& w\cdot v_1&=v,
  \end{align*}
  such that
  \begin{enumerate}
  \item $w$ is a commuting word,
  \item $u'$ is properly left-absorbed by $v_1$,
  \item $u'$ and $w$ commute.
  \end{enumerate}
  \qed
\end{cor}

\begin{cor}\label{C:Red_inkompr}
  Given reduced words $u$, $v$ and $x$ such that $[u\cdot v]=u\cdot
  x$, then $x$ is a final subword of $v$. In particular, the word $x$
  commutes with any word commuting with $v$.
\end{cor}

\begin{proof}
  Decompose  \begin{align*}
    u&=u_1\cdot u'\cdot w&
    w\cdot v'\cdot v_1&=v,
  \end{align*} as in Proposition \ref{P:Dec}. Hence $u_1\cdot w\cdot
  v_1=[u\cdot v]=u_1\cdot w\cdot u'\cdot x$, so $v_1\approx u'\cdot
  x$, by Lemma \ref{L:init_subw_kuerzen}. Thus $u'=1$ and $x$ is a
  final subword of $v$, as desired.
\end{proof}

\begin{lemma}\label{L:reordering}
  If the reduced word $u$ is left-absorbed by $v_1\cdot v_2$, then
  $u\approx u_1 \cdot u_2$, where each $u_i$ is left-absorbed by
  $v_i$, and $u_2$ commutes with $v_1$ (and therefore with $u_1$).
\end{lemma}
\begin{proof}
  If $u$ is empty, there is nothing to show. Otherwise, write
  $u=u'\cdot s$. By induction on the length of $u$, obtain a
  decomposition $u'\approx u'_1\cdot u'_2$, such that $u'_i$ is
  absorbed by $v_i$ and $u'_2$ commutes with $v_1$. We distinguish two
  cases: If $s$ is absorbed by $v_1$, then $u'_2$ commutes with $s$,
  so decompose $u\approx (u'_1\cdot s)\cdot u'_2$.
  Otherwise, the letter $s$ commutes with $v_1$ and is left-absorbed
  by $v_2$. Set $u\approx u_1\cdot(u_2\cdot s)$.
\end{proof}

\begin{cor}\label{C:hinundher}
  Given a reduced word $u\approx u_1\cdot \tilde u $, where $\tilde u$
  denotes the final segment of $u$, then $[u\cdot u\inv]=u_1\cdot
  \tilde u \cdot {u_1}\inv$.
\end{cor}
\begin{proof}
  It suffices to show that the word $u_1\cdot \tilde u \cdot
{u_1}\inv$ is reduced. Otherwise, there is an end $s$ of $u_1$ which
commutes with $\tilde u$. In that case, the letter $s$ is an end of
$u$ and hence it occurs in $\tilde u$, which is a  contradiction.
\end{proof}

For the proof of the next lemma, we will require the following
notation: Recall (\emph{cf.\ }Remark \ref{R:newflags}) that
$\Gamma\setminus\partial s$ is the set of those colours commuting
with $s$. For a letter $t$, denote by $\Cent_t(s)$ the commuting word
with support $t\setminus\partial s$ (\emph{cf.\ }Remark
\ref{R:commuting_word}). For a word $v=t_1\cdots t_n$, set
    \[\Cent_v(s)=\Cent_{t_1}(s)\cdots\Cent_{t_n}(s).\]

\begin{lemma}\label{L:div}(Division Lemma)
  Given reduced words $u\preceq v$, there exists a reduced word $w$,
  unique up to permutation, such that for every reduced word $x$,

  \[ [ x\cdot u] \preceq v \Leftrightarrow x\preceq w.\]

\end{lemma}

We denote $w$ by $v/u$. Since $v/u\preceq v/u$, Corollary
\ref{C:kleinergleich} implies that $v/u\preceq v$. Furthermore, since
$1\preceq v/u$, the condition $u\preceq v$ is necessary for the
existence of $v/u$.

\begin{proof}
  Note that we may assume that $u$ consists of a single letter:
  given $u = u_1\cdot u_2$, suppose that the statement holds for both
  $u_1$ and $u_2$. Since $u\preceq v$, we have $u_2\preceq v$, so
$v/u_2$ exists. Now $u_1\cdot u_2\preceq v$ implies $u_1\preceq
v/u_2$. Thus, set $v/u=(v/u_2)/u_1$, which exists.   Note that
  \[[x\cdot u]=[ [x\cdot u_1]\cdot u_2]\preceq v \Leftrightarrow
    [x\cdot u_1]\preceq v/u_2 \Leftrightarrow x\preceq v/u,\] as
desired. Therefore, assume $u=s\preceq v$.

    We can write $v=v_1\cdot t\cdot v_2$, such
    that $s\subset t$ and no letter in $v_2$ contains $s$. Set $w=
    [v_1\cdot t\cdot \Cent_{v_2}(s)]$. No letter in $v_1\cdot t$
    is left-absorbed during the non-splitting reduction
    $v_1\cdot t\cdot \Cent_{v_2}(s)\to w$, for $v$ is reduced.

    Clearly $[w\cdot s]=w\preceq v$. Given $x\preceq w$  reduced, Corollary
    \ref{C:Ord_Mon} implies that  $ [ x\cdot s] \preceq [w\cdot s]
\preceq v$, which gives one implication. For the other, assume that
$[x\cdot s]\preceq v$. We
distinguish two cases~: if $s$ is absorbed by $x$, then
$x=[x\cdot s]\preceq v=v_1\cdot t\cdot v_2$, so we may decompose
$x\approx x_1 \cdot x_2$, where $x_1\preceq v_1\cdot t\cdot$ and
$x_2\preceq v_2$. Since $x_2$ does not right-absorb $s$, Lemma
\ref{L:reordering} implies that  $x_2$ commutes with $s$, which is
right-absorbed by $x_1$. This implies $x_2\preceq\Cent_{v_2}(s)$ and thus
$x\preceq w$.

  If $s$ is not absorbed by $x$, then Proposition \ref{P:Dec} yields a
decomposition $x\approx  x_1\cdot x'$, where $x'$ is properly absorbed by
  $s$ and $[x\cdot s]=x_1\cdot s$. The word $x_1\cdot s$ right-absorbs
$s$, so $x_1\cdot s \preceq w$ by the previous discussion. Since
  $x'\prec s$, we conclude that $x\approx  x_1 \cdot x'\preceq w$, as
desired.
\end{proof}

\begin{definition}\label{D:sr}
 According to Lemma \ref{L:div}, we denote by $\sr(u)$ the largest
reduced word, unique up to permutation, such that for every reduced
word $x$,

\[ [ u\cdot x] \preceq u \Leftrightarrow x\preceq \sr(u).\]
\end{definition}

\begin{cor}\label{C:sr}
Given a reduced word $u$, the word $\sr(u)\preceq u$ is commutative. A
reduced word $x$ is right-absorbed by $u$ \iff $x\preceq \sr(u)$.

 In particular, if $\tilde u$ denotes the final segment of $u$, then
$\tilde u\preceq \sr(u)$.
\end{cor}
\begin{proof}
Note that $x$ is right-absorbed by $u$ \iff $[u\cdot x]=u$. Since
$u\preceq [u\cdot x]$, by Corollary \ref{C:kleinergleich}, this is
equivalent to  $[ u\cdot x] \preceq u$, that is $x\preceq \sr(u)$.

In particular, the word $\sr(u)$ is right-absorbed by $u$. Thus
$[u\cdot [\sr(u)\cdot \sr(u)]]=[[u\cdot \sr(u)]\cdot \sr(u)]=[u\cdot
\sr(u)]=u$, so $\sr(u)\preceq [\sr(u)\cdot \sr(u)]\preceq \sr(u)$, which
implies that $\sr(u)$ is commutative, by Corollary \ref{C:kommwort}.
\end{proof}

\section{Forking}\label{S:Forking}

We work inside a big sufficiently saturated model $M$ of the theory
$\psg$, as a universal domain. Recall, by Corollary \ref{C:stab}, that
$\psg$ is $\omega$-stable,. Given a finite tuple $a$ and subsets $C\subset
B$ of $M$, the extension $\tp(a/C)\subset \tp(a/B)$ is \emph{non-forking},
if $\RM(\tp(a/B))=\RM(\tp(a/C))$. More generally, given subsets $A$, $B$
and $C$ of $M$, the set $A$ is \emph{independent from} $B$ over $C$,
denoted by \[A\ind_C B,\]
\noindent if, for every finite tuple $a$ in $A$, the extension
$\tp(a/C)\subset \tp(a/B\cup C)$ is non-forking. This gives rise to a
well-behaved notion of independence, which has, among many other,
the following remarkable properties (\emph{cf.\ }\cite[Corollary
8.5.4 and Theorem 8.5.5]{TZ12}):
\begin{description}
 \item[Symmetry]  If $A\ind_C B$, then $B\ind_C A$.
 \item[Extension] Given a tuple $a$ and subsets $C\subset B$ of $M$, there
is a non-forking extension of $\tp(a/C)$ to $B\cup C$, that is, there is
some realisation of $\tp(a/C)$ which is independent from
$B$ over $C$.
 \item[Stationarity] Every type $p$ over an elementary substructure
$N$ of $M$ is \emph{stationary}, that is, given a subset $B$ of
$M$, there is a unique non-forking extension of $p$ to $N\cup B$.
\item[Invariant Extension] Given a type $p$ over a sufficiently saturated
elementary substructure $N$ of $M$ which is \emph{invariant} over a
small subset $C\subset N$, that is, every automorphism of $N$ fixing $C$
fixes $p$ (as a collection of formulae), then $p$ is the unique
non-forking extension of $p\restr{C}$ to $N$.
\end{description}

Indeed, the last property follows from the fact that a global type,
which is invariant over $C$, does not fork over $C$ \cite[Exercise
  7.1.4]{TZ12}, plus the fact that all non-forking extensions of
$p\restr{C}$ to $N$ are conjugate under automorphisms of $N$
\cite[Theorem 8.5.6]{TZ12}.

In a similar fashion as in \cite[Section 7]{BMPZ13}, we will describe
non-forking over nice sets and canonical bases in $\psg$. We will also
show that this theory, which has weak elimination of imaginaries, has
trivial forking and furthermore is totally trivial, as defined in
\cite{jBG91}.

Recall the terminology introduced in Definition \ref{D:w_FG}: A word
$u$ connects the flag $F$ to the flag $G$ if it is the word of a reduced
path from $F$
to $G$. This word is unique, up to permutation, and denoted by
$\w(F,G)$.

The following result describes the type of a flag over a nice
set and will help us to determine the nature of non-forking.

\begin{prop}\label{P:fusspunkt}
  Given a flag $F$ and a reduced path with word $u$ which connects $F$
  to a flag $G$ lying in a nice set $D$, the following are equivalent:
  \renewcommand{\theenumi}{\alph{enumi}}
  \begin{enumerate}
   \item\label{P:fusspunkt:indep} For any flag $G'$ in $D$,
   \[ \w(F,G')= [u\cdot\w(G,G')],\]
   that is, the word connecting $F$ to $G'$ is equivalent to $[u\cdot
     v]$, the non-splitting reduct of $u\cdot v$, where $v$ is the
   reduced word connecting $G$ to $G'$.
  \item\label{P:fusspunkt:kleinstes} The word $u$ is the
    $\preceq$-smallest word connecting $F$ to some flag in $D$.
  \item\label{P:fusspunkt:minimal} The word $u$ is $\preceq$-minimal
    among words connecting $F$ to some flag in $D$.
   \end{enumerate}
\end{prop}

\noindent This generalisation of \cite[Proposition 7.2]{BMPZ13} has
essentially the same proof. Note that a word $u$ satisfying Property
$(\ref{P:fusspunkt:kleinstes})$ is unique, up to permutation, for
$\prec$ is irreflexive.
\begin{proof}
  (\ref{P:fusspunkt:indep})$\to$(\ref{P:fusspunkt:kleinstes}) follows
  from Corollary \ref{C:kleinergleich}. The implication
  (\ref{P:fusspunkt:kleinstes})$\to$(\ref{P:fusspunkt:minimal})
  is trivial.\\

  \noindent For
  (\ref{P:fusspunkt:minimal})$\to$(\ref{P:fusspunkt:indep}), let $v$ be
the word of a reduced path from $G$ to some flag $G'$ in $D$.  By niceness,
we may assume that the path is fully contained in $D$.
  Choose a decomposition $u=u_1\cdot u'\cdot w$ and $w\cdot v'\cdot
  v_1=v$, as in Proposition \ref{P:Dec}, with corresponding paths
 \[F\op{u_1\cdot u'}F_1\op{w}G\op{w}G_1\op{v'\cdot v_1}G',\]
 where $G_1$ is some flag in $D$. The word $b$ connecting $F_1$ to $G_1$
 is a reduct of $w\cdot w$, so $b\preceq w$. The word $c$
 connecting $F$ with $G_1$ is hence a reduct of $u_1\cdot u'\cdot b$, so
  \[c\preceq u_1\cdot u'\cdot b\preceq u_1\cdot u'\cdot
 w\preceq u.\] Minimality of $u$ implies that $c\approx u$, so
 $b\approx w$. Hence, in the reduction of the path $F\op{}G'$, no
 splitting occurred and the resulting word is $u_1\cdot w\cdot
 v_1=[u\cdot v]$, by Proposition \ref{P:Dec}.
\end{proof}

\begin{definition}\label{D:basept}
  A \emph{base-point} of the flag $F$ over the nice set $D$ is a flag
  $G$ in $D$ such that any of the conditions of Proposition
  \ref{P:fusspunkt} hold.
\end{definition}

Recall Definition \ref{D:sr} of $\sr(u)$. By Corollaries
\ref{C:absorb=stab} and \ref{C:sr}, we easily conclude the following:

\begin{cor}\label{C:Fusspunktgleich}
  Let $G$ be a base-point of $F$ over the nice set $D$ and $u$ the
  word which connects $F$ to $G$. Let $v$ be a reduced word connecting
  $G$ to some flag $G_1$ in $D$. The flag $G_1$ is a base-point of $F$
  over $D$ \iff $v$ is right-absorbed by $u$ \iff $u=[u\cdot
    v]$ \iff $v\preceq \sr(u)$.
\end{cor}

\begin{lemma}\label{L:alpha_kette}(\emph{cf.}\
  \cite[Lemma 7.4]{BMPZ13}) Let $G$ be a base-point of $F$ over the
  nice set $D$ and denote by $P$ a reduced path $F=F_0, \ldots, F_n=G$
  with word $u$. Then $D\cup P$ is nice. It is
  uniquely determined by $G$ and $u$ in the following strong sense: If
  $P'=F'_0,\ldots, F'_n$ is a second path with word $u$ from $F$ to $G$,
  then there is a (unique) isomorphism $D\cup P\to D\cup P'$ which is
  the identity on $D$ and maps each $F_i$ onto $F'_i$.
\end{lemma}
We will express the last property by saying that the "extension
$F_0\cdots F_n/D$" is uniquely determined, up to isomorphism.
\begin{proof}
 If $u$ is trivial, there is nothing to show. Otherwise, write
 $u=u'\cdot s$. Mi\-ni\-mality of $u$ yields that no splitting of $s$
 can connect $F_{n-1}$ to a flag in $D$. Thus Proposition
 \ref{P:nice=nosplit} implies that the set $D'=D\cup F_{n-1}$ is nice
 and a simple extension of $D$ of type $(s,G)$.

 We will now show that $F_{n-1}$ is a base-point of $F$ over $D'$, by
 Proposition \ref{P:fusspunkt} (\ref{P:fusspunkt:minimal}). Otherwise,
 there is a reduced word $v\prec u'$ which connects $F$ to a flag $H'$
 in $D'$. By minimality of $u$, the flag $H'$ cannot lie in $D$. By
 Remark \ref{R:newflags}, there exists some flag $H$ in $D$ such that
 $H'\op{s} H$. This gives a reduced path from $F$ to $H$ whose word is
some reduction of $v\cdot s$, contradicting the minimality of $u$.

 By induction, if $P_0$ denotes the subpath $F=F_0,\ldots, F_{n-1}$, then
 the set $D'\cup P_0$ is nice and the extension $F_0\cdots F_{n-1}/D'$
 is uniquely determined, up to isomorphism, by $G'$ and $u'$.
 Therefore, the extension $F_0\cdots F_n/D$ is uniquely determined, up
 to isomorphism, by $G$ and $u$.
\end{proof}

\begin{cor}\label{C:fusspunkttyp}
  Given a reduced word $u$ and a flag $G$ in a nice set $D$, there is
  a flag $F$ such that $u$ connects $F$ to $G$, which is a base-point
  of $F$ over $D$. The type of $F$ over $D$ is uniquely determined by
  $G$ and $u$.
\end{cor}
\noindent Recall that the word $\w(F,G)$ is determined only up to a
permutation. So the $\tp(F/D)$ depends only on the equivalence class
of $u$.
\begin{proof}
  Observe that, if such a flag $F$ exists and $P$ denotes the reduced
  path from $F$ to $G$ with word $u$, Lemma \ref{L:alpha_kette} and
  Corollary \ref{C:nicetype} imply that the type of $P$ over $D$ is
  uniquely determined by $G$ and $u$.

  Thus, we need only show existence of such a flag $F$, by induction
  on the length of $u$. If $u=1$, there is nothing to do. Otherwise,
  write $u=s\cdot u'$ and choose a flag $F'$ connecting to $G$ by $u'$
  such that $G$ is a base-point of $F'$ over $D$. Let $P'$ denote the
  reduced path $F'\op{u'}G$. By Lemma \ref{L:alpha_kette}, the set
  $D'=D\cup P'$ is nice. Corollary \ref{C:equiv_forms} yields a simple
  extension $D'\cup F$ of type $(s,F')$. Proposition
  \ref{P:nice=nosplit} implies that $F'$ is base-point of $F$ over
  $D'$. We need only show that $G$ is a base-point of $F$. Hence, let
  $G'$ be an arbitrary flag in $D\subset D'$. We have

  \[\w(F,G')=[s\cdot \w(F',G')]=[s\cdot[u'\cdot  \w(G,G')]]=
     [u\cdot \w(G,G')],\]
as desired.
\end{proof}

If we denote the type of $F$ over $G$,  resp.\ over $D$, by
\[\p_u(G)\text{ resp.\ } \p_u(G)|D ,\]

\vspace{1em}

\noindent we conclude the following.

\begin{prop}\label{P:nicescaffold}
  Let $G$ be a flag in the nice set $D$ and $u$ a reduced word, then
  $\p_u(G)|D$ is the unique non-forking extension of $\p_u(G)$ to $D$.
\end{prop}
\begin{proof}
  By the {\bf Extension} principle, we may replace $D$ by a sufficiently
saturated elementary substructure containing it, which is again nice in
$M$. Since $\p_u(G)|D$ is invariant over $G$, the {\bf Invariant
Extension} principle yields the desired result.
\end{proof}

Since the type $\p_u(G)$ admits a non-forking extension to $D$,
which must coincide with $\p_u(G)|D$, by the previous result, we obtain
the following immediate observation.

\begin{cor}\label{C:forking}
  (\emph{cf.}\ \cite[Lemmata 7.4 and 7.6]{BMPZ13}) Given a flag $F$ and
  a nice set $D$, the flag $G$ in $D$ is a base-point of $F$ over $D$
  \iff $F\ind_G D$.
\end{cor}

Recall that the \emph{canonical base} $\cb(p)$ of a stationary type $p$ is
some set, fixed pointwise by exactly those automorphisms of
$M$ fixing the global non-forking extension $\mathbf{p}$ of $p$ to $M$. If
$\cb(p)$ exists, then it is unique, up to interdefinability, and
$\mathbf{p}$ is the unique non-forking extension to $M$ of its
restriction to $\cb(p)$. Furthermore, if $p$ is a stationary type over
$B$ and $\cb(p)$ exists, then $p$ does not fork over $A\subset B$ if and
only if $\cb(p)$ is algebraic over $A$.

Canonical bases exist as \emph{imaginary} elements in the expansion
$T^\mathrm{eq}$ of an $\omega$-stable theory $T$ \cite[Chapter 8.4]{TZ12}.

As in Remark \ref{R:commuting_word}, given a reduced word $u$, we do
not distinguish between the word $\sr(u)$ and its support. For a flag
$G$, the class of $G$ modulo $\sr(u)$ can be identified with a subset
of the real sort, namely with the set of vertices of $G$ whose colours
do not belong to $\sr(u)$. To simplify the notation, we will denote
this class by $G/\sr(u)$.

\begin{cor}\label{C:cb_type_pG}
  The class $G/\sr(u)$ is a canonical base of\/ $\p_u(G)$.
\end{cor}

\begin{proof}
  Two types $\p_u(G)$ and $\p_u(G_1)$ have a common global non-forking
extension \iff $\p_u(G)|D=\p_u(G_1)|D$ for some nice set $D$ which
contains
  $G$ and $G_1$. This is
  equivalent, by Corollary \ref{C:Fusspunktgleich}, to
$\w(G,G_1)\preceq\sr(u)$, which is equivalent to
  $G_1\sim_{\sr(u)} G$.
\end{proof}

\begin{definition}\label{D:Rdiv}
  Given reduced words $u$ and $v$, we say that $u$ is a \emph{proper
    left-divisor} of $v$ if $u\not\approx v$ and there is a reduced
  word $w$ such that $[u\cdot w]=v$.
\end{definition}

If $u$ is a proper left-divisor of $v$, it follows that $u\prec v$. In
particular, being a proper left-divisor is a well-founded relation.
Let $\Rd$ be its foundation rank, and likewise let $\Rkl$ denote the
foundation rank of reduced words with respect to $\prec$.

The foundation rank of types associated to the forking relation is
called \emph{Lascar rank}, denote by $\U$-rank. This means that a type
has Lascar rank at least $\alpha+1$ \iff it has a forking extension of
Lascar rank at least $\alpha$.

The following result can be proved exactly as \cite[Lemmata 7.10 and
  7.11]{BMPZ13}

\begin{lemma}\label{L:URdiv}
  For every flag $G$ and every reduced word $u$,
  \[\U(\p_u(G))=\Rd(u)\leq\RM(\p_u(G))\leq\Rkl(u).\]
\end{lemma}

In general $\Rd(u)$, $\RM(\p_u(G))$ and $\Rkl(u)$ need not agree
(\emph{cf.} \cite[Remark 7.14]{BMPZ13}). They coincide however in the
following special case, the proof of which is a straight-forward
modification of the proof of \cite[Lemma 7.12]{BMPZ13}, together with
 Lemma \ref{L:Zus}.

\begin{lemma}\label{L:ranks}(\emph{cf.}\ \cite[Corollary 7.13]{BMPZ13})
  For every reduced word $u=s_1\cdots s_n$, with $|s_i|\geq |s_{i+1}|$
  for $i=1,\ldots,n-1$,
  \[\Rd(u)=\Rkl(u)=\omega^{|s_1|-1}
  +\dotsb+\omega^{|s_n|-1}.\]
\end{lemma}
\begin{remark}
  For an arbitrary reduced word $u$, the Morley rank of the type
  $\p_u(G)$ can be easily computed thanks to the following
  observation: The rank of $\p_u(G)$ is strictly larger than $\alpha$
  \iff either
  \begin{enumerate}
    \item the word $u$ has a proper left-divisor $v$ such that $\p_v(G)$
has at
      least rank $\alpha$,
  \end{enumerate}
  \noindent or
  \begin{enumerate}
    \setcounter{enumi}{1}
  \item the type $\p_u(G)$ is an accumulation point of a family of
types  $\p_v(G)$, each of rank at least $\alpha$.
  \end{enumerate}
\end{remark}

\begin{cor}\label{C:Ranks}
  The theory $\psg$ is $\omega$-stable of Morley rank $\omega^{K-1}$,
  where $K$ is the cardinality of a connected component
of\/  $\Gamma$ of largest size.
\end{cor}
\begin{proof}
  Decompose $\Gamma=\bigcup\limits_{i=1}^n\Gamma_i$ into its connected
  components. Similarly as in Corollary \ref{C:Unternice}, it is easy
  to see that each restriction $M_i=\A_{\Gamma_i}(M)$ is a model of
  $\mathrm{PS}_{\Gamma_i}$. The structure $M$ can be considered as the
  disjoint union of the structures $M_i$'s, so the Morley rank of $M$
  is the maximum of the Morley ranks of the structures $M_i$. We may
  therefore assume that $\Gamma$ is connected.

  Given any vertex $a$, choose a flag $F$ containing $a$ as well as a
  flag $G$ independent from $F$ over $\emptyset$. If $\p_u(G)$ is the
  type of $F$ over $G$, then the word $u$ must be equal to $\Gamma$,
  since the canonical base $G/\sr(u)$ is algebraic over the empty set.
  By Lemma \ref{L:ranks}, we have
  \[\U(F) = \RM(F) = \Rkl(\Gamma)=\omega^{K-1}.\]
  By Lascar inequalities (\emph{cf.\ }\cite[Exercise 8.6.5]{TZ12}),
    we have that $\U(F/a)+\U(a)\leq\U(F)=\omega^{K-1}$, so
    $\U(a)=\omega^{K-1}$, since $\U(a)>0$. Since
  \[\omega^{K-1}=\U(a)\leq \RM(a)\leq\RM(F)=\omega^{K-1},\] we have
equality, as desired.
\end{proof}

Two types $p$ and $q$, possibly over different sets of
parameters, are \emph{non-orthogonal}, if there is a common extension
$C$ of both sets of parameters, and two realisations $a$
and $b$ of the corresponding non-forking extensions
of $p$ and $q$ to $C$ such that $a$ forks with $b$ over $C$. As in
\cite[Theorem 7.15]{BMPZ13}, we conclude the following.

\begin{remark}\label{C:nonortho}
  Every type over a nice set $D$ is non-orthogonal to some
  $\p_s(G)|D$, where $G$ lies in $D$.
\end{remark}

Given a reduced flag path $P:F\op{u}G$, we will conclude this section
by describing the flags one can obtain from the collection of vertices
of the flags occurring in $P$, as well as describing how the flags in
$P$ can vary (or \emph{wobble}), whilst the endpoints are fixed.

\begin{lemma}\label{L:Flagfusspunkt}(\emph{cf.}\
  \cite[Lemma 6.18]{BMPZ13}) Let $A$ be the set of vertices of the
  flags occurring in a reduced flag path $P$. Then any flag in $A$
  occurs in some permutation of $P$ (\emph{cf.}\ Lemma
\ref{L:path_permutation}).
\end{lemma}

\begin{proof}
  Write $P:F\op{u}G$ and note first that $G$ is the base-point of $F$
  over $G$. If $u=1$, then $F=G$ is the only flag in $A$, so there is
  nothing to prove. Otherwise, write $u=s\cdot v$ and decompose the
  path $P$ as $F\op{s}H\op{v}G$. If we denote by $B$ the set of
  vertices of the flags occurring in the reduced path $H\op{v}G$,
  Lemma \ref{L:alpha_kette} yields that $B$ is nice. By induction and
  Proposition \ref{P:nice=nosplit}, the nice set $B\cup F$ is a simple
  extension of $B$ of type $(s,H)$.

  Given any flag $K$ in $A$, we distinguish two cases: if $K$ lies in
  $B$, by induction $K$ occurs in some permutation of $H\op{v}G$,
  which induces a permutation of $P$. Otherwise, by Remark
  \ref{R:newflags}, there exist a reduced word $w$ commuting with $s$
  and a flag $K_1$ in $B$ such that $F\op{s} H\op{w} K_1$ and $F\op{w}
  K \op{s} K_1$. By Corollary \ref{C:path_with_commuting_words}, we
  may assume that the second path is a permutation of the first. By
  induction, the flag $K_1$ belongs to a permutation
  $H\op{w}K_1\op{v_2}G$ of $H\op{v}G$. So $F\op{w}
  K\op{s}K_1\op{v_2}G$ is a permutation of $P$.
\end{proof}

The following generalises Remark \ref{R:newflags} and Lemma
\ref{L:Flagfusspunkt}.

\begin{lemma}\label{L:newflag2}
  Let $G$ be a base-point of $F$ over the nice set $D$ and $P$ be a
  reduced path connecting $F$ to $G$. For every flag $K'$ in the nice
  set $D\cup P$, there are flags $K$ occurring in some permutation of
  $P$ and $G'$ in $D$, such that $w=\w(K,G)$ commutes with
  $v=\w(G,G')$ and $K \op{v} K' \op{w} G'$.
\end{lemma}
\begin{proof}
  If $P$ is trivial, set $K=G$ and $G'=K'$. Otherwise, decompose $P$
 into $F\op{s}F'\op{u'}G$ and set $P':F'\op{u'}G$.  Note that $G$ is
also a base-point of $F'$ over $D$. If $K'$
  is contained in $D\cup P'$, find, by induction on the length of
  $P$, a flag $K$ occurring in some permutation of $P'$ and $G'$ in
  $D$, as desired. Otherwise, Remark \ref{R:newflags} implies the
  existence of a flag $K'_1$ in $D\cup P'$ such that $w'=\w(F',K'_1)$
  commutes with $s$ and $F\op{w'} K'\op{s} K'_1$. By induction, we
  find a flag $K_1$ occurring in a permutation of $P'$ and a flag $G'$
  in $D$ such that $w_1=\w(G,G')$ commutes with $v_1=\w(K_1,G)$ and
  $K_1\op{w_1}K'_1\op{v_1}G'$.

  Let $P_1:K_1\op{v_1} G$ be the part of a permutation of $P'$ which
connects $K_1$ to
  $G$. Lemma \ref{L:alpha_kette} implies that the set $D_1=D\cup P_1$
  is nice. Furthermore, the flag $K_1$ is a
  base-point of $F'$ over $D_1$. Set $u_1= \w(F',K_1)$. Since
  $F'\op{u_1} K_1 \op{w_1} K'_1$, we conclude that $w'=[u_1\cdot
    w_1]$. In particular, the letter $s$ must commute with both $u_1$
  and $w_1$. Since $F\op{s} F'\op{u_1} K_1$, there exists thus a
  unique flag $K$ such that $F\op{u_1} K\op{s} K_1$. Clearly, the flag
  $K$ occurs in a permutation of $P$ with word $u_1\cdot s\cdot v_1$
  and $\w(K,G)=s\cdot v_1 =\w(K',G')$. We need only show that
$\w(K,K')$ commutes with $s\cdot v_1$. Since $K\op{u_1\inv}F
  \op{w'}K'$, the word $\w(K,K')$ is some reduction of $u_1\inv\cdot
w'$ (possibly with splitting), so $\w(K,K')$ commutes with $s$, since
both $u_1$ and $w$ do. The flag path $K\op{s}K_1\op{w_1}K'_1\op{s}K'$
must reduce to one with word $\w(K,K')$, which commutes with
$s$, so the word \[  s\cdot w_1\cdot s \approx w_1\cdot s\cdot s \]
must reduce to $w_1=\w(K,K')$, which commutes with $s\cdot v_1$, as
desired.
\end{proof}

\begin{cor}\label{C:Perm}
  Let $A$ be the nice set consisting of the vertices of the
  flags occurring in a reduced flag path $P:F\op{u}G$. Every flag $K$
in $A$ is uniquely determined by $\w(K,G)$. Thus,
  the only automorphism of $A$ fixing one of the endpoints of $P$ is
  the identity.
\end{cor}
\begin{proof}
  Given flags $K$ and $K'$ in $A$ with $\w(K',G)=\w(K,G)=w$, Lemma
\ref{L:Flagfusspunkt} shows that there are permutations
$Q:F\op{v}K\op{w}G$ and $Q':F\op{v'}K'\op{w'}G$ of $P$.  Since
$w\approx w'$, Lemma
  \ref{L:init_subw_kuerzen} implies that $v'\approx v$. We may
  therefore permute $Q'$ in order to decompose it as
  $Q_1:F\op{v}K'\op{w}G$.  The correspondence between permutations of a
flag path and permutations of the associated word in Lemma
\ref{L:path_permutation} implies that $Q=Q_1$, so $K'=K$.
\end{proof}

\begin{definition}\label{D:Wob}
The \emph{wobbling} of a reduced product $u\cdot v$ is
$\wob(u,v)=\sr(u)\cap \sr(v\inv)$, that is, the collection of those
$\gamma$ in $\Gamma$ which are both
right-absorbed by $u$ and left-absorbed by $v$.
\end{definition}
\noindent Note that $\wob(u,v)$ cannot be equal to $|u|$ nor to $|v|$,
for $u\cdot v$ is reduced.

\begin{lemma}(Wobbling Lemma \emph{cf.}\ \cite[Lemma 6.19]{BMPZ13})
\label{L:Wob}
  Given two reduced paths between the flags $F$ and $G$ with the same
  word $u=s_1\cdots s_i\cdots s_n$,
\begin{figure}[!htbp]
\centering

\begin{tikzpicture}[>=latex,->]

\matrix (A) [matrix of math nodes,column
sep=1cm]
{  & H_1 & \cdots &  H_{n-1} &  \\
 F &   &  &  &   G, \\
   & H'_1 & \cdots & H'_{n-1} &   \\};

 \draw (A-2-1)  edge node[pos=.7, above left] {$s_1$} (A-1-2)
		edge node[pos=.7, below left] {$s_1$} (A-3-2) ;

\draw (A-1-2)  edge  (A-1-3) ;
\draw (A-3-2)  edge (A-3-3);

\draw (A-1-3)  edge  (A-1-4) ;
\draw (A-3-3)  edge (A-3-4);

 \draw[<-] (A-2-5)  edge node[pos=.7, above right] {$s_n$} (A-1-4)
		edge node[pos=.7, below right] {$s_n$} (A-3-4) ;

\end{tikzpicture}

\end{figure}

\noindent then $H_i$ and $H'_i$ are $\wob(s_1\cdots  s_i,
s_{i+1}\cdots s_n)$-equivalent for every $i$ in $\{1,\ldots,n-1\}$.

In particular, the tuple $H_i/\wob(s_1\cdots  s_i, s_{i+1}\cdots
s_n)$ enumerating the vertices of $H_i$ with colours
in $\Gamma\setminus \wob(s_1\cdots  s_i, s_{i+1}\cdots
s_n)$ lies in $\dcl(F,G)$.

\end{lemma}

\begin{proof}
Given two different flag paths as in the above picture, we prove the
statement by induction on the index $i<n$. For $i=1$, let $w_1$ be the
reduced word connecting $H_1$ to $H'_1$. Since $H_1\op{s_1}F\op{s_1}
H'_1$, it follows that $w_1\preceq s_1$ and $w_1$ is right absorbed
by $s_1$.
If $w_1=s_1$, it contradicts
Proposition \ref{P:Eindwort}, since $H_1\op{s_2\cdots s_n}
G$. Therefore, the word $w_1$ is a proper splitting of $s_1$. Furthermore,
since $w_1\cdot s_2\cdots s_n\str s_2\cdots s_n$, no
letter from $s_2\cdots s_n$ can be absorbed during the reduction, for
the word $u$ is reduced.
Corollary \ref{C:absorb=stab} implies that $w_1$ is completely
absorbed by $s_2\cdots s_n$, so $|w_1|\subset \wob(s_1,s_2\cdots s_n)$, as
desired.

Let now $H_i\op{w_i} H'_i$, resp.\ $H_{i+1}\op{w_{i+1}} H'_{i+1}$. By
induction, the word $w_i$ has support in $\wob(s_1\cdots
s_i,s_{i+1}\cdots s_n)$. Lemma \ref{L:reordering} yields that
$w_i\approx w_i^1\cdot w_i^2$, where $w_i^1$ is left-absorbed by
$s_{i+1}$ and $w_i^2$ commutes with $s_{i+1}$ and left-absorbed by
$s_{i+2}\cdots s_n$. Thus $w_i^1\prec s_{i+1}$. Since $$s_{i+1}\cdot
w_i \cdot s_{i+1} \to w_i^2 \cdot s_{i+1} \cdot s_{i+1}$$ reduces to
$w_{i+1}$, we conclude that $w_{i+1}\approx w_i^2\cdot x$, where
$x\preceq s_{i+1}$. Note that $x\prec s_{i+1}$, since
$H_{i+1}\op{s_{i+2}\cdots s_n}G$, so $x$ is a splitting of $s_{i+1}$
which must be absorbed by $s_{i+2}\cdots s_n$ and so is $w_{i+1}$. The
word $w_i^2$ is right-absorbed by $s_1 \cdots s_i$, by induction, and
commutes with $s_{i+1}$. Since $x$ is a proper splitting of $s_{i+1}$,
we have that and $s_1 \cdots s_{i+1} \cdot w_i^2 \cdot x \to s_1
\cdots s_{i+1}$, so $w_{i+1}$ is right-absorbed by $s_1 \cdots
s_{i+1}$, as desired.
\end{proof}

\begin{lemma}(Base-point Lemma \emph{cf.}\ \cite[Lemma
    7.18]{BMPZ13})\label{L:basept} Given a flag $H$ with base-point
  $G$ in the nice set $A$, let $H\op{v} G$ be a reduced flag path
  connecting $H$ to $G$. If $H/W$ lies in $\acl(A)$ for some subset
  $W\subset \Gamma$, then $|v|\subset W$.
\end{lemma}

\begin{proof}
  Choose a flag $H'\models \tp(H/\acl(A))$ with $H\ind_A H'$. Note
  that $G$ is a base-point in $A$ for $H'$ as well. Thus Proposition
  \ref{P:nicescaffold} and transitivity of non-forking imply that
  $H\ind_G H'$. Furthermore, the flags $H$ and $H'$ are
  $W$-equivalent, so the reduced word $w$ connecting $H$ to $H'$ has
  support in $W$. By Proposition \ref{P:fusspunkt}, the word $w$ is
  the non-splitting reduct of $v\cdot v\inv$ and equals $v_1\cdot
  \tilde v\cdot v_1\inv$, where $\tilde v$ is the final segment of
  $v\approx v_1\cdot \tilde v$, by Corollary \ref{C:hinundher}.
  Thus, the set $|v|$ is contained in $|w|\subset W$, as desired.
\end{proof}

\begin{cor}\label{C:nice_acl}
Nice sets are algebraically closed.
\end{cor}

\begin{proof}
 Let $b$ be algebraic over the nice set $A$. Suppose that $b$ has
colour $\gamma$ and choose some flag $H$ containing it. Pick some
base-point $G$ for $H$ over $A$, and let $P:H\op{u}G$ be a reduced flag
path connecting $H$ to $G$.  By assumption, the set
$H/(\Gamma\setminus\{\gamma\})=\{b\}$ lies in $\acl(A)$, so
$|u|\subset \Gamma\setminus\{\gamma\}$, by Lemma \ref{L:basept}, that
is, the element $b$ equals the $\gamma$-vertex of $G$, which lies in
$A$.
\end{proof}

\section{Equationality}\label{S:Equationality}

In this section, we will show that the theory  $\psg$ is equational.

\begin{definition}\label{D:Eq}
A parameter-free formula $\varphi(x,y)$, where the tuple $x$ has
length $n$, is an $n$-\emph{equation} if the family of finite
intersections of instances $\varphi(x,a)$ (where
$a$ belongs to a sufficiently saturated model $N$) has the descending
chain condition (DCC).

A complete theory $T$ is $n$-\emph{equational} if every definable set
in $N^n$ is a Boolean combination of instances of $n$-equations. A
theory is \emph{equational} if it is $n$-equational for every $n$ in
$\N$.
\end{definition}

Stability, a wider class class containing $\omega$-stable theories
\cite[Section 8.2]{TZ12}, is preserved under naming parameters and
bi-interpretability. The same holds for equationality \cite{Ju00}. However,
it is unknown whether equationality follows from $1$-equationality, which
itself implies stability for formulae $\varphi(x,y)$, where $x$ is a
single variable, and thus stability \cite{PiSr84}. The rest of this section
is devoted to showing that the theory $\psg$ is equational. As in the
previous section, let $M$ denote a sufficiently saturated model of $\psg$,
inside of which we
work.

\begin{definition}\label{D:P_u}
  Given a word $u=s_1\cdots s_n$ and a flag $G$ in $M$, let
$\pp_u(X,G)$ be the formula stating the existence of a
  sequence  $X=F_0, \ldots, F_n=G$ of flags with
  $F_{i-1}\sim_{s_i} F_i$ for $1\leq i\leq n$.
\end{definition}

It follows from Lemmata \ref{L:modlemma}, \ref{L:path_permutation} and
\ref{L:Flagpath} that $M\models \pp_u(F,G)$ \iff there exists some
reduction $u'\preceq u$ with $F\op{u'} G$. Thus, if $w\preceq
u$, then the sentence $\forall X \forall Y \left( \pp_w(X,Y)
\rightarrow \pp_u(X,Y)\right)$ holds in all $\Gamma$-spaces. When the word
$w$ is reduced, the converse holds on  models of $\psg$,
as shown below.

\begin{remark}\label{R:v_preceq_u}
  Let $w$ be a reduced word and $w_1,\ldots,w_n$ be arbitrary words.
Then \[\psg
  \models \forall X \forall Y\, \Bigl( \pp_w(X,Y) \to
  \bigvee_{i=1}^n\pp_{w_i}(X,Y)\Bigr)\] \iff $w\preceq w_i$ for some
  $i$.
\end{remark}

\begin{proof}
  We need only prove left-to-right.  By Corollary
\ref{C:fusspunkttyp}, there are flags $F$ and $G$
  in our saturated model $M$ such that $F\op{w} G$. Thus
$\pp_{w_i}(F,G)$ for some $i$, so $F\op{w'} G$ for some
reduced $w'\preceq w_i$. Proposition \ref{P:Eindwort} implies that
$w\approx w'$, so $w\preceq w_i$.
\end{proof}

\begin{lemma}\label{L:Komp}
  Given two arbitrary  words $u$ and $v$, there is a finite
collection of reduced words $w_1,\ldots,w_n$ such that, for any
reduced word $w$,
  \[ w\preceq u,v \Longleftrightarrow \bigvee_{i=1}^n w\preceq w_i.\]
  \noindent In particular, the words $w_i\preceq u,v$ for all $i$.
\end{lemma}

\begin{proof}
  In $M$, the formula $\pp_u(F,G)\land\pp_v(F,G)$
implies that $\pp_w(F,G)$ for some reduced word $w\preceq u,v$. By
compactness, there is a finite set $w_1,\ldots,w_n$ of reduced
words satisfying that $w_i\preceq u,v$ for $i=1,\ldots, n$ and
 \[ \psg
  \models \forall X \forall Y\, \Bigl(\pp_u(X,Y)\land\pp_v(X,Y)
\to \bigvee_{i=1}^n\pp_{w_i}(X,Y)\Bigr).\] Thus, if $w\preceq u,v$,
the
  formula $\pp_w(X,Y)$ implies $\bigvee_{i=1}^n\pp_{w_i}(X,Y)$.
  Whence  $w\preceq w_i$ for some
  $i$, by Remark \ref{R:v_preceq_u}.
\end{proof}

\begin{lemma}\label{L:divisors}
  Given flags $G_1$ and $G_2$ in $M$, and reduced words $u_1$ and
  $u_2$ such that neither $\pp_{u_1} (X,G_1)$ nor
  $\pp_{u_2}(X,G_2)$ imply the other, if $P$ denotes a reduced path
  from $G_1$ to $G_2$, then the conjunction
  $\pp_{u_1}(X,G_1)\land\pp_{u_2}(X,G_2)$ is equivalent to
  \[\bigvee_{i=1}^n\pp_{w_i} (X,H_i)\]
  for some flags
  $H_1,\ldots,H_n$ occurring in some permutation of $P$ and reduced
  words $w_1,\ldots, w_n\prec u_1,u_2$.
\end{lemma}
\noindent
If $\pp_{u_1} (X,G_1)$ and $ \pp_{u_2}(X,G_2)$ are disjoint, set
$n=0$. If $G_1=G_2$, this is the content of Lemma \ref{L:Komp}.

\begin{proof}
  Choose any realisation $F\models
\pp_{u_1}(X,G_1)\land\pp_{u_2}(X,G_2)$ and
a basepoint $H$ of $F$ over the nice set determined by the reduced
path $P$. The flag $H$ occurs in some permutation of $P$, by
Lemma \ref{L:Flagfusspunkt}.

Set $x=\w(F,H)$ and $v_i=\w(H,G_i)$ for $i=1,2$. Observe that
$\w(F,G_i)\preceq u_i$ for $i=1,2$. Proposition \ref{P:fusspunkt}
implies that $[x\cdot v_i]=\w(F,G_i)\preceq u_i$, so $x\preceq
u_i/v_i$  for $i=1,2$, by Lemma \ref{L:div}.  We obtain the following
diagram:

  \begin{figure}[h]
    \centering
    \begin{tikzpicture}[>=latex,text height=1ex,text depth=1ex]
      \fill  (0,0)  node [below left]  {$G_1$}  circle (2pt);
      \fill   (4,0)  node [below right]  {$G_2$} circle (2pt);
      \fill   (2,2)  node [above] {$F$} circle (2pt);
      \fill   (2,0)  node [below=3pt] {$H$} circle (2pt);

      \draw[->]  (2,2) -- node [left]  {$\preceq u_1$} (0,0) ;
      \draw[->]  (2,2)-- node [right] {$\preceq u_2$} (4,0) ;
      \draw[->]  (2,2)-- node [pos=.4, left] {$x$} (2,0) ;
      \draw[->]  (2,0)-- node [below] {$v_2$} (4,0) ;
      \draw[->]  (2,0)-- node [below] {$v_1$} (0,0) ;
    \end{tikzpicture}
  \end{figure}

  \noindent In particular, a flag $F$ realises
  $\pp_{u_1}(X,G_1)\land\pp_{u_2}(X,G_2)$ \iff
  there is some flag $H$ occurring in some permutation of $P$ with
  \[ \pp_{u_1/v_1}(F,H)\wedge \pp_{u_2/v_2}(F,H).\]
  Lemma \ref{L:Komp} applied to $u_1/v_1$ and $u_2/v_2$ yields reduced
  words $w_1,\ldots, w_n$ describing the above intersection. We
  need only show that $w_j\prec u_1,u_2$ for $j=1,\ldots,n$.
Clearly $w_j\preceq u_i/v_i\preceq u_i$. Suppose however that
   $w_j\approx u_1$ for some $j=1,\ldots,n$, so $u_1\preceq
u_i/v_i$ and thus $u_1\preceq [u_1\cdot v_1]\preceq u_1$. Hence
$\pp_{u_1}(X,G_1)=\pp_{u_1}(X,H)$. Also $u_1\preceq [u_1\cdot
v_2]\preceq u_2$, so $\pp_{u_1}(X,G_1)=\pp_{u_1}(X,H)\subset
  \pp_{u_2}(X,H) \subset \pp_{u_2}(X,G_2)$, contradicting our
  hypothesis.
\end{proof}

Since the relation $\prec$ is well-founded, we conclude the following.

\begin{cor}\label{C:P_uEq}
  The formulae $\pp_u(X,Y)$ are equations.
\end{cor}

By bi-interpretatibility, in order to conclude that the theory $\psg$
is equational, we need only show that every formula whose free
variables enumerate flags is a Boolean combination of formulae
$\pp_u(X,G)$. For that, we will first introduce the notion of nice
hulls.

A colour-preserving graph homomorphism $f:A\to B$ between two
$\Gamma$-spaces induces a homomorphism $\chi(f): \chi(A)\to \chi(B)$
between
the chamber systems of flags of $A$ and $B$. It is easy to see that
this defines an isomorphism between the category of $\Gamma$-spaces
and the category of dual quasi-buildings. (\emph{cf.} Theorem
\ref{T:biinterpretierbar})

\begin{definition}\label{D:contring}
  Suppose that both $A$ and $B$
  are \SC. Given $X\subset \chi(A)$ and $Y\subset \chi(B)$,
  the map $\phi: X\to Y$ is \emph{contracting} \iff for any
  two flags $F$ and $G$ in $X$,
  \[ \w_B(\phi(F),\phi(G))\preceq \w_A(F,G).\]
\noindent We say that $\phi$ is an isometry if \[
\w_B(\phi(F),\phi(G)) =
\w_A(F,G).\]
\end{definition}

The following is easy to see:

\begin{lemma}\label{L:Corresp}
  If $A$ and $B$ are \SC, a map $\chi(A)\to\chi(B)$ is contracting
  \iff it is an homomorphism of chamber systems.
\end{lemma}

\begin{definition}\label{D:spec}
  If $A\subset B$ are $\Gamma$-spaces, a \emph{specialisation} from
  $B$ to $A$ is a homomorphism $f:B\to A$ such that
  $\rest{f}{A}=\mathrm{Id}_A$.
\end{definition}

\begin{remark}\label{R:Spec_collapse}
 If $B=A\cup F$ is a simple extension of $A$ of type $(s,G)$, then
the map $f$ fixing all elements of $A$ which sends the
$\gamma$-vertex of $F$ to the $\gamma$-vertex of $G$, for $\gamma$
in $s$, is a specialisation.
\end{remark}

\begin{lemma}\label{L:spec_nice}
  Given $\Gamma$-spaces $A\subset B$, with $B$ \SC, the subspace
$A$ is nice in $B$ \iff $B$ can be specialised to $A$.
\end{lemma}

\begin{proof}
  Suppose that $B$ can be specialised to $A$, and let $P:F\op{u}G$
be a reduced path in $B$ connecting two flags $F$ and $G$ in $A$. The
specialisation maps $P$ to a connecting
  path $P'$ in $A$ between $F$ and $G$ with word $u'\preceq u$, so $A$
is nice in $B$.

  If $A$ is nice in $B$, observe that $A$ specialises to itself.
Choose then a maximal specialisation $f:C\to A$, with $C\subset B$
nice in $B$. If $C\neq B$, Proposition \ref{P:nice-bandf} yields a
proper simple extension $C'$ of $C$, which is again nice in $B$. By
the remark \ref{R:Spec_collapse}, there is a specialisation $C'\to
C$. Composing it  with $f$ contradicts the maximality of $f$.
\end{proof}

Given two nice subsets $N_1$ and $N_2$ of our fixed saturated model
$M$, an \emph{isomorphism} means a bijection $f:N_1\to N_2$ such that
both $f$ and $f\inv$ are homomorphisms of $\Gamma$-spaces. If
$N_1$ and $N_2$ have  a common subset $A$ , we say that they are
$A$\emph{-isomorphic}, denoted by $N_1\simeq_A N_2$, if there is an
isomorphism between $N_1$ and $N_2$ fixing all elements in $A$.

\begin{definition}\label{D:incompressible_hull}
 Let $N\subset M$ be nice and $A$ be some subset of $N$.
  \begin{enumerate}
  \item We say that $N$ is a \emph{nice hull} of $A$ if every nice
subset of $M$ containing $A$ has a nice subset $N'$ which is
$A$-isomorphic to $N$.
  \item\label{D:incompressible_hull:inc} The $\Gamma$-space $N$ is
\emph{incompressible} over $A$ if every $A$-homomorphism $f:N\to N$ is
an automorphism of $N$.
  \item The nice subset $N$ is \emph{strongly incompressible} over $A$
if every $A$-homomorphism $f:N\to M$ induces an isomorphism of $N$
with a nice subset of $M$.
  \end{enumerate}
\end{definition}

We will see in Proposition \ref{P:starr} that, if $N$ is
incompressible over $A$, then the only $A$-endomorphism of $N$ is the
identity.

Lemma \ref{L:spec_nice} implies the following easy observation.
\begin{remark}\label{R:incomp_nosubset}
If $N$ is incompressible over $A$, then it contains no proper nice
subset $A\subset N'\subsetneq N$. Likewise if $N$ is strongly
incompressible over $A$.
\end{remark}

\begin{lemma}\label{L:incompressible_hull}\par\indent
  \begin{enumerate}
  \item\label{L:incompressible_hull:strong} If the nice set $N$ is
    strongly incompressible over $A$, then it is incompressible and a
    nice hull of $A$.
  \item\label{L:incompressible_hull:iso} If $N$ is a nice hull of $A$
    and $N'$ is incompressible over $A$, then $N\simeq_A N'$.
  \end{enumerate}
\end{lemma}
\begin{proof}
  For $(\ref{L:incompressible_hull:strong})$, let $N$ be
  strongly incompressible over $A$. To show that $N$ is incompressible
over $A$, consider an $A$-homomorphism $f:N\to N$. This must induce
  an $A$-isomorphism with a nice subset $N_1$ of $N$ containing $A$,
so $N_1=N$ by Remark \ref{R:incomp_nosubset}.

 Let us now show that $N$ is a nice hull. Given any nice subset $N'$
of $M$ containing $A$, choose a specialisation
$f:M\to N'$, by Lemma \ref{L:spec_nice}. The map $\rest{f}{N}$ must
then induce an $A$-isomorphism of $N$ with a nice subset of $N'$, as
desired.

 Suppose now $N$ and $N'$ are as stated in
$(\ref{L:incompressible_hull:iso})$. Since $N$ is a nice hull, there
is some
nice subset $N''$ of $N'$ which is $A$-isomorphic to $N$. We conclude
that $N''=N'$ by the Remark \ref{R:incomp_nosubset}.
\end{proof}

We now have all the necessary ingredients to conclude the following
result.
\begin{theorem}\label{T:Nice_hulle}
  Every subset $A$ of $M$ has a unique, up to $A$-isomorphism,
strongly incompressible extension. If $A$ is finite, so is this
extension.
\end{theorem}

\begin{proof}
  We give the proof for $A$ finite, and leave the general case to
  reader.

  Proceed by induction over $|A|$. If $A$ is empty, then any flag
  is strongly incompressible over $\emptyset$. Otherwise, write
  $A=A_0\cup\{a\}$ and choose, by
  induction, a finite strongly incompressible extension $N_0$ of
$A_0$. Among all flags passing through $a$, choose one, say $F$,
with $\preceq$-minimal word $u=\w_M(F,G)$, where $G$ is some
base-point of $F$ over $N_0$. Assume furthermore that $u$ is
$\preceq$-minimal
among all possible $A_0$-copies of $N_0$. Let $P:F\op{u}G$ be a
reduced flag path  connecting $F$ to $G$. Set $N=N_0\cup P$, which is
a finite nice subset of $M$, by Lemma \ref{L:alpha_kette}. In order to
show that $N$ is strongly incompressible over $A$, consider an
$A$-homomorphism $f:N\to M$. By induction, the map $\rest{f}{N_0}$
induces an $A_0$-isomorphism between $N_0$ and the nice set $f(N_0)$,
so $f(N_0)$ is also strongly incompressible over $A_0$. The map $f$ is
contracting,
by Lemma \ref{L:Corresp}, so $\w(f(F),f(G))\preceq u$. Since $a$ is
contained in $f(F)$, minimality of $u$ implies that
$\w(f(F),f(G))=u$, thus $f(G)$ is a base-point
  of $f(F)$ over $f(N_0)$. Therefore, the set $f(P)$ determines a
reduced path from $f(F)$ to $f(G)$. Hence, the set $f(N)$ is nice, by
Lemma \ref{L:alpha_kette}, and $f$ is an $A$-isomorphism, as desired.
\end{proof}

Together with Lemma \ref{L:incompressible_hull}, we conclude:
\begin{cor}\label{C:nicehull}
  Every set $A$ has a nice hull $\nh(A)$, which is
incompressible and unique, up to $A$-isomorphism.  If $A$ is finite,
then so is $\nh(A)$.
\end{cor}

\noindent Thus, the three notions in
Definition \ref{D:incompressible_hull} coincide.

\begin{cor}\label{C:acl_endlich}
  The algebraic closure $\acl(A)$ of a finite set $A$ is finite and
  contained in $\nh(A)$.
\end{cor}

\begin{proof}
  Let $\nh(A)$ be the nice hull of a finite set $A$. Nice sets are
  algebraically closed, by Corollary \ref{C:nice_acl}. Thus, the set
  $\acl(A)$ is contained in $\nh(A)$, which is finite.
\end{proof}

\begin{prop}\label{P:starr}
  The nice hull $\nh(A)$ is rigid over $A$, that is, its only
  automorphism fixing $A$ pointwise is the identity.
\end{prop}
\begin{proof}
  Again, we leave the case of infinite $A$ to the reader and assume
  $A$ finite.

  If $A$ is empty, recall that $\nh(\emptyset)$ consists of a single
  flag, so the result is obvious. By the proof of Theorem
  \ref{T:Nice_hulle}, if $A=A_0\cup\{a\}$, then $\nh(A)=N_0\cup P$,
  where $N_0$ is a nice hull of $A_0$ and $P$ is a reduced flag path
  connecting a flag $F$ containing $a$ to its basepoint $G$ over
  $N_0$. Furthermore, the word $u=\w(F,G)$ is $\preceq$-minimal among
  all words $\w(F,G')$, where $G'$ has the same type as $G$ over
  $A_0$. If $\gamma$ denotes the colour of $a$, then $u$ has a unique
  beginning $s$, which contains $\gamma$.

  By Remark \ref{R:Spec_collapse}, choose a specialisation
  $\phi:N_0\cup P\to N_0$, collapsing the whole path $P$ onto $G$. Let
  now $f$ be some automorphism of $\nh(A)$ fixing $A$. The map
  $\rest{(\phi\circ f)}{N_0}$, which is an $A_0$-automorphism, by
  strong incompressibility of $N_0$, must then be the identity, by
  induction on $|A|$. Thus $f$ is the identity on $N_0\setminus G$.

  Since $f$ is an automorphism, we have that
  $\w(f(F),f(G))=\w(F,G)=u$. As the flag $f(F)$ lies in $N_0\cup P$,
  Lemma \ref{L:newflag2} yields, after permutation, flags $K$ with
  $P:F\op{u_1} K\op{u_2} G$ and $G'$ in $N_0$ such that
  $u_2=\w(f(F),G')$ commutes with $v=\w(K,f(F))=\w(G,G')$. If $P'$
  denotes the subpath $K\op{u_2} G$, the flag $K$ is a base-point of
  $F$ over $N_0\cup P'$. Notice that the flag $f(F)$ lies in $N_0\cup
  P'$. Thus, the word $\w(F,f(F))$ equals $[u_1\cdot v]$.
  Since $a=f(a)$ is contained in $f(F)$, the flags $F$ and $f(F)$ are
  $(\Gamma\setminus \{\gamma\})$-equivalent, so $\gamma$ does not
  occur in $\w(F,f(F))$. If $u_1$ is not trivial, then $s$ must be its
  beginning, for $u$ has only $s$ as a beginning, and hence, either
  $s$ or a larger letter containing it must occur in $\w(F,f(F))$.
  Thus, we conclude that $u_1=1$ and the word $v=\w(F,f(F))$ commutes
  with $u$.

  As the flag $f(G)$ lies in $N_0\cup P$, by Lemma \ref{L:newflag2} there
  are flags $\tilde K$ with $P:F\op{\tilde u_1} \tilde K\op{\tilde
    u_2} G$ and $\tilde G$ in $N_0$ such that
  $\tilde v=\w(\tilde K,f(G))$ commutes with $\tilde u_2=\w(f(G),
  \tilde G)$. As above, we see that $\w(F, f(G)) = [\tilde u_1\cdot
  \tilde v]$. We deduce the following:
  $$\tilde u_1\cdot \tilde u_2\cdot v\approx v\cdot u \approx
  v\cdot \w(f(F),f(G)) \approx [\tilde u_1\cdot \tilde v].$$
  Corollary \ref{C:Red_inkompr} implies that
  $\tilde u_2\cdot v$ is a final subword of $\tilde v$, which commutes
  with $\tilde u_2$. We conclude that $\tilde u_2=1$ and hence
  $f(G)=\tilde G$ lies in $N_0$. The
  map $f$ maps $N_0$ to itself, so it is the identity on $N_0$. Since $f$
  induces a permutation of $P$,  Corollary \ref{C:Perm} implies that $f$ is
  the identity on $N$.
\end{proof}
\begin{cor}\label{C:acl_rigid}
  The algebraic closure $\acl(A)$ of $A$ is rigid over $A$, so it equals
$\dcl(A)$.
\end{cor}
\begin{proof}
  Note that $\acl(A)$ is contained in $\nh(A)$, by
Corollary \ref{C:acl_endlich}, so $\nh(A)=\nh(\acl(A))$. Every
$A$-automorphism of $\acl(A)$ extends to an automorphism of $\nh(A)$,
which must then be the identity, by Proposition \ref{P:starr}.
\end{proof}
\begin{prop}\label{P:Stat}
  All types are stationary.
\end{prop}
This needs no longer hold if we consider types over subsets of
$M^\mathrm{eq}$.
\begin{proof}
  We need only prove the statement for $1$-types, and thus it suffices
to prove it for types of a single flag. Let $p$ be the type of a flag
$F$ over the parameter set $A$ and consider $q$ a global non-forking
extension of $p$ to $M$. Since $q=\p_u(G)|M$ for some flag $G$ in $M$, its
canonical parameter is $B=G/\sr(u)$, by Corollary \ref{C:cb_type_pG},
which is interdefinable with a set of real elements. Since $B$ is
algebraic over $A$, it is  hence definable over $A$, by Corollary
\ref{C:acl_rigid}. The type $p$ is thus stationary.
\end{proof}

We will show that the theory $\psg$ is equational, by proving that the
type of finitely many flags is determined by the collection of words
connecting each pair of flags. For two flags this follows
from Lemma \ref{L:alpha_kette}: A reduced path $P:F\op{u}G$ determines a
nice subset, whose type is determined by $u$. For three flags $F$, $G$
and $H$ as below:

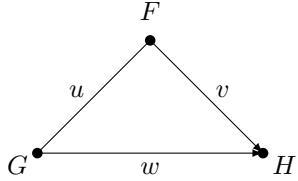
\begin{figure}[h]
   \centering

   \begin{tikzpicture}[>=latex,text height=1ex,text depth=1ex]

     \fill  (0,0)  node [below left]  {$G$}  circle (2pt);
     \fill   (3,0)  node [below right]  {$H$} circle (2pt);
     \fill   (1.5,1.5)  node [above] {$F$} circle (2pt);

     \draw[->] (0,0) -- node [left] {$u$} (1.5,1.5) -- node [right] {$v$}
(3,0) ;
          \draw[->] (0,0)-- node [below]
          {$w$} (3,0) ;
   \end{tikzpicture}

 \end{figure}

\noindent we need the following Proposition.

\begin{prop}\label{P:Canred}
  \begin{enumerate}[(1)]
  \item\label{P:Canred:woerter} Given $u$, $v$ and $w$ reduced words
    with $u\cdot v\str w$, there is a decomposition:
    \begin{align*}u&\approx u_1\cdot \alpha\inv \cdot c\inv,
      \\ v&\approx c \cdot \beta \cdot v_1, \\ w&\approx u_1\cdot
      x\cdot v_1,
    \end{align*}
    \noindent such that $\alpha$, $\beta$ and $x$ pairwise commute,
    the word $x$ is properly right-absorbed by $c$, the word $\alpha$
    is properly left-absorbed by $v_1$, and $\beta$ is right-absorbed
    by $u_1$. The words $u_1$, $v_1$, $c$, $x$, $\alpha$ and $\beta$
    are unique up to permutation.
  \item\label{P:Canred:flagge} Assume further that $F$, $G$ and $H$
    are flags such that $\w(G,F)=u$, $\w(F,H)=v$ and $\w(G,H)=w$. Then
    there is a reduced path $P:G \op{w} H$, and a base-point $K$ of
$F$
    over $P$ such that:
    \begin{align*} \w(G,K)&= u_1\\
    \w(K,H)&=x\cdot v_1, \\
\w(F,K)&=c\cdot\alpha\cdot\beta
    \end{align*}
  \end{enumerate}
\end{prop}

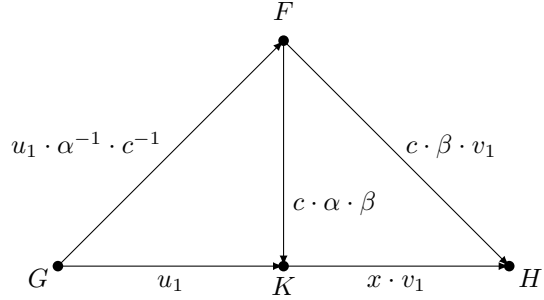
\begin{figure}[h]
   \centering

   \begin{tikzpicture}[>=latex,text height=1ex,text depth=1ex]

     \fill  (0,0)  node [below left]  {$G$}  circle (2pt);
     \fill   (6,0)  node [below right]  {$H$} circle (2pt);
     \fill   (3,3)  node [above] {$F$} circle (2pt);
      \fill   (3,0)  node [below=3pt] {$K$} circle (2pt);

     \draw[->] (0,0) -- node [left] {$u_1\cdot \alpha\inv \cdot c\inv$}
(3,3);
     \draw[->] (3,3)-- node [right] {$c \cdot \beta \cdot v_1$}
     (6,0) ;
     \draw[->] (3,3)-- node [pos=.75, right]
          {$ c\cdot\alpha\cdot\beta$} (3,0) ;
          \draw[->] (3,0)-- node [below] {$x\cdot v_1$} (6,0) ;
          \draw[->] (0,0)-- node [below]
          {$u_1$} (3,0) ;
   \end{tikzpicture}

 \end{figure}

 Observe that the existence of
a decomposition as in (\ref{P:Canred:woerter}) implies $u\cdot v\str
w$. Indeed, since  $x$ is properly right-absorbed by $c=t_m\cdots t_1$,
Lemma \ref{L:reordering} implies that $x\approx x_1\cdots x_m$, where
$x_i$ is
a splitting $t_i$ and commutes with $t_j$, whenever $j<i$.
Thus,
\begin{align*}
 c\inv\cdot c=t_1\cdots t_m\cdot t_m\cdots t_1 &\str t_1\cdots
t_{m-1}\cdot x_m\cdot t_{m-1}\cdots t_1 \str \\
&\str t_1\cdots
t_{m-1}\cdot t_{m-1}\cdot x_m\cdot t_{m-2}\cdots t_1 \str \cdots
\end{align*}

\noindent can be reduced to $x$. Therefore,

\begin{align*}
u\cdot v \approx u_1\cdot \alpha\inv \cdot c\inv \cdot c\cdot \beta\cdot
v_1 &\str u_1\cdot \alpha\inv\cdot x\cdot \beta\cdot v_1 \str \\
&\str u_1\cdot\beta \cdot x\cdot \alpha\inv\cdot v_1 \str u_1\cdot x\cdot
v_1 \approx w.
\end{align*}

\noindent By Lemma \ref{L:Flagfusspunkt}, we may assume
in (\ref{P:Canred:flagge}) that the flag $K$ occurs in the path $P$.

\begin{proof}
  Let us first prove the existence of such a decomposition. By Remark
  \ref{R:paths_exists}, find flags $F$, $G$ and $H$ such that $G\op{u}
  F \op{v} H$ and $G\op{w}H$. Let $P$ be a reduced flag path
  connecting $G$ to $H$. By Lemma
  \ref{L:Flagfusspunkt}, we may choose a base-point $K$ of $F$ over $P$
 which occurs in the path $P$. Set $\w(G,K)= w_1$,
$\w(K,H)=w_2$, and $\w(F,K)=y$. Assume that $y$ is $\preceq$-minimal among
all choices of $P$, which implies

  \[ \text{No end of $y$ is contained in } \wob(w_1,w_2) \tag{$\ddag$}.\]

  \noindent Indeed, if $y=y'\cdot s$, with $s\subset
\wob(w_1,w_2)$, decompose $F\op{y'}K'\op{s}K$ for some flag $K'$.
Observe that $K$ is also a base-point for $K'$ over the nice set
determined by $P$, so by Proposition  \ref{P:fusspunkt} the
reduction of $s\cdot w_1\inv$, resp.\  $s\cdot w_2$, is non-splitting and
equals $w_1\inv$, resp.\ $w_2$.  Replacing  $K$ by $K'$, we obtain a
permutation $P'$ of $P$, so that  $F$ connects to $P'$ with word $y'$,
contradicting the minimality  of $y$.

  Proposition \ref{P:fusspunkt} implies that $[ y\cdot w_1\inv]=u\inv$ and
  $[y\cdot w_2]=v$. By Proposition \ref{P:Dec}, up to permutations
of $w_1$ and $u$,  write:
  \begin{align*} w_1&=u_1\cdot x', \\ y&\approx c_1\inv \cdot \beta
    ,\\ u&=u_1\cdot c_1
  \end{align*}
  \noindent where $x'$ and $\beta$ commute, the word $x'$ is properly
  left-absorbed by $c_1$ and $\beta$ is right-absorbed by $u_1$.
  Let $K'$ be a flag in $P$ such that $G\op{u_1}K'\op{x'}K$. Since
  $x'$ is right-absorbed by $y$, the flag $K'$ is also a base-point of
  $F$ over $P$ by Corollary \ref{C:Fusspunktgleich}. Replacing $K'$ by
$K$, we may assume that $x'=1$.

  Likewise, write
  \begin{align*}w_2&= x\cdot v_1, \\ y&\approx c_2\cdot \alpha,\\ v&=
c_2\cdot v_1
  \end{align*}
  \noindent where $x$ and $\alpha$ commute, the word $x$ is properly
  right-absorbed by $c_2$ and $\alpha$ is left-absorbed by $v_1$.

 Note that $y=c_2\cdot \alpha \approx c_1\inv\cdot \beta$. However, no end
  of $\alpha$ can be an end of $\beta$, by $(\ddag)$. Thus, every
 end of $\alpha$ commutes with $\beta$ and Lemma \ref{L:initial}
  yields that $\alpha$ commutes with $\beta$ and is a final subword
  of $c_1\inv$. Likewise for $\beta$ and $c_2$. After possible
  permutations of $c_1$ and $c_2$, write:
  \begin{align*}
    c_1\inv&=c\cdot \alpha, \\
    c_2&=c\cdot \beta.
  \end{align*}

  Let us now show that $\beta$ and $x$ commute. Otherwise, write
  $x=x_1\cdot s\cdot x_2$, where $x_1$ and $\beta$ commutes, but $\beta$
and $s$ do not. Since
  $x$ is right-absorbed by $c_2=c\cdot\beta$, the letter $s$
  must be absorbed by $\beta$, and  therefore right-absorbed by $u_1$.
Since $s$ commutes with $x_1$, the word $w=u_1\cdot
  x\cdot v_1$ is not reduced, which is a contradiction.
\noindent  Hence, the word $x$ is properly right-absorbed by $c$.

  We have now
  \begin{align*}
    u&= u_1\cdot\alpha\inv\cdot c\inv,\\
    v&= c\cdot\beta\cdot v_1,\\
    w&=w_1\cdot w_2=u_1\cdot x\cdot v_1,\\
    y&=c\cdot\alpha\cdot \beta.
  \end{align*}
  The only property left to show is that $\alpha$ is properly
  left-absorbed by $v_1$. Otherwise, apply Corollary \ref{C:proper_absorb}
to  produce, up to permutation, the following decompositions:
  \begin{align*}
    \alpha&=\alpha'\cdot\omega &\omega\cdot v_2&=v_1,
  \end{align*}
  where $\omega$ is a commuting word, the word $\alpha'$ is properly
  left-absorbed by $v_2$, and $\alpha'$ and $\omega$ commute. Since
  $w_2\approx \omega\cdot x\cdot v_2$, there is a flag $K'$ in some
permutation of  $P$
  such that $K\op{\omega}K'\op{x\cdot v_2}H$. Now, the word $\omega$ is
  right-absorbed by $y\approx c\cdot \beta\cdot \alpha'
\cdot \omega$, so the flag $K'$ is also a base-point of $F$ over $P$, by
Corollary \ref{C:Fusspunktgleich}.  Replacing $K$ by $K'$, and
substituting:
  \begin{align*}
    u_1&\rightsquigarrow u_1\cdot\omega\\
    v_1&\rightsquigarrow v_2\\
    \alpha&\rightsquigarrow\alpha'\\
    \beta&\rightsquigarrow\beta\cdot\omega
  \end{align*}
\noindent gives that the new words $u_1$, $v_1$, $c$, $x$,
  $\alpha$, $\beta$ and the new base-point $K$ have the desired
  properties.

  For uniqueness, let $u_1\cdot a$ be the largest common initial
  subword of $u=u_1\cdot\alpha\inv\cdot c\inv$ and $w=u_1\cdot x\cdot v_1$
  (\emph{cf.} the discussion after Definition \ref{D:beginning}). Since
$\alpha$ and $x$ commute and each one is properly left
  absorbed by $v_1$, resp.\ $c\inv$, the word $a$ must commute
with $\alpha$ and $x$, and is an initial subword of both $c\inv \approx
a\cdot c'$ and $v_1\approx a\cdot v'_1$. In particular, the words
$\alpha\cdot c'$ and $x\cdot v_1'$ are uniquely determined, up to a
  permutation, for $u_1\cdot a$ is.

  Observe that $c'$ is the largest common final subword of $\alpha\cdot
c'$ and $v\inv={v_1'}\inv\cdot
  a\inv\cdot\beta\inv\cdot a\cdot c'$, for otherwise, the largest
common
  final subword would then contain a letter $s$ from $\alpha$, which is
  an end of ${v_1'}\inv\cdot a\inv\cdot\beta\inv\cdot a$,
contradicting that $\alpha$ is properly left-absorbed by $v_1'$.
Therefore, the words $\alpha$ and
  $a\inv\cdot\beta\cdot a\cdot v_1'$ are uniquely determined.

  Since $x$ and $a\inv\cdot\beta\cdot a$ commute, the word
  $v_1'$ is the largest common final subword of $x\cdot v_1'$ and
  $a\inv\cdot\beta\cdot a\cdot v_1'$, so $x$ and
  $a\inv\cdot\beta\cdot a$ are uniquely determined. The result now follows
by applying the following auxiliary result to the words $u_1\cdot a$
and $a\inv\cdot\beta\cdot a$, in order to determine $u_1$, $a$ and
$\beta$, as desired.

\begin{claim}
Given reduced words $e\approx u\cdot a$ and $f\approx a\inv\cdot
b\cdot a$, where $u$ right-absorbs
  $b$, then $a$, $u$ and $b$ are uniquely determined by $e$ and $f$.

\end{claim}

\begin{proof}[{\it Proof of the claim:}]
\renewcommand{\qed}{\noindent\emph{End of the proof of the
claim.}}

We proceed by induction on the length of $f$. Clearly, if
$f=1$, then $b=a=1$ and thus $e=u$.

Otherwise, note that $e$ right-absorbs $f$
\iff $a=1$, in which case $e=u$ and $f=b$.
Therefore, if $e$ does not right-absorb $f$,  write
$a=a'\cdot s$. In particular, the word $e$ has an end
which is simultaneously a beginning and an end of $f$.

We show first that the only possible ends $t$ of $e$ which are both a
beginning and an end of
$f$ are exactly the ends of $a$. If not, the end $t$ must be an end of
$u$  which commutes with $a$. Likewise $b\approx t\cdot b'\cdot t$,
which contradicts that $u$ right-absorbs $b$, since $b'$ and $t$ do
not commute.

Removing $s$ yields the reduced words $u\cdot a'$ and ${a'}\inv \cdot
b\cdot a'$, and the result now follows by induction.

\end{proof}
\end{proof}

\begin{cor}\label{C:Flagtypedist}
  The type $\tp(F,G,H)$ of three flags $F$, $G$ and $H$ is uniquely
  determined by $\w(F,G)$, $\w(G,H)$ and $\w(F,H)$.
\end{cor}

\begin{proof}
  Proposition \ref{P:Canred} yields a reduced path
  $P:G\op{w_1}K\op{w_2}H$ from $G$ to $H$, where  $K$ is a
  base-point of $F$ over the nice set determined by $P$,  such that
$w_1$, $w_2$ and $\w(F,K)$ are uniquely determined, up to
permutation, by the words $\w(F,G)$, $\w(G,H)$ and
  $\w(F,H)$. By Corollary \ref{C:fusspunkttyp}, the type of $F$ over $P$ is
  uniquely determined by $\w(F,K)$ and $K$. Note that $H$ is a base-point
of $G$ over the nice set $H$, so  Lemma \ref{L:alpha_kette} yields that
the type of $P=G,\ldots, K, \ldots, H$ is determined by
  the word $w_1\cdot w_2$. Thus, the type
  of $GKH$ is determined by the equivalence classes of $w_1$ and
  $w_2$, as desired.
\end{proof}

To extend this result to arbitrary sets of flags, we need the following
lemma. Given a $\Gamma$-space $A$, recall that $\chi(A)$ denotes the
chamber system of the flags in $A$ (\emph{cf.}\ the discussion before
Definition \ref{D:contring}).

\begin{lemma}\label{L:nice_flagmenge}
  A non-empty collection $X$ of flags of $M$ equals $\chi(A)$ for some nice
  subset $A$ of $M$ \iff, whenever $F$ and $G$ in $X$ are connected by a
  reduced word $u$ in $M$, then there is path in $X$ with word $u$
  connecting $F$ to $G$.
\end{lemma}
\noindent If $X=\chi(A)$, the remark
after Definition \ref{D:nice} implies that $A$ is the union of the
flags in $X$.
\begin{proof}
  One direction is an equivalent definition of niceness, so we need
  only show that, if $X$ satisfies the right-hand condition, then
  $X=\chi(A)$ for some nice subset $A$. Set $A$ the collection of all
  vertices of flags in $X$. In order to show that $X=\chi(A)$ and that
  $A$ is nice, it suffices to show that any flag in $A$ is a flag
  from $X$, that is, we need only show, by induction on $|S|$, that
  given a collection of elements $S$ in $A$ lying in a flag $H$ in
  $M$, there is a flag in $X$ containing $S$. If $S$ is a singleton,
  there is nothing to prove. Otherwise, enumerate $S=\{a_1,\ldots,
  a_r\}$ and set $T\subset \Gamma$ the collection of colours of
  $a_1,\ldots,a_{r-1}$, and $\gamma$ the colour of $a_r$. Choose a
  flag $F$ in $X$ containing $a_r$ and, by induction, a flag $G$ in
  $X$ containing $a_1,\ldots,a_{r-1}$. Observe that

  \[ F\sim_{\Gamma\setminus\{\gamma\}} H \sim_{\Gamma\setminus T} G,\]
  \noindent so there are reduced words $u$ and $v$, such that $\gamma$
  does not occur in $u$, the support of $v$ is disjoint from $T$ and
  $F\op{u}H\op{v}G$. Reducing this path in $M$ gives a word $w_1\cdot
  w_2$, where $\gamma$ does not occur in $w_1$ and each letter in
  $w_2$ is disjoint from $T$. By assumption, there is a reduced path
  in $X$ of the form $F\op{w_1} H'\op{w_2} G$. Clearly, the path $H'$
  lies in $X$ and contains the set $S$, as desired.
\end{proof}
We will now extend Corollary \ref{C:Flagtypedist} to arbitrarily many
flags.

\begin{theorem}\label{T:Eq}
  The type $\tp(F_1, \ldots,F_n)$ of a sequence of flags is uniquely
  determined by $\w(F_i, F_j)$ for $1\leq i\neq j\leq n$.
\end{theorem}

\begin{proof}
  Recall the notion of isometry, as in Definition \ref{D:contring}.
  Given two collections of flags $X$ and $X'$, we will show that a
  surjective isometry $\phi:X\to X'$ induces an elementary map between
  nice subsets of $M$.

  Suppose first that $X$ satisfies the conditions of Lemma
  \ref{L:nice_flagmenge}. Then so does $X'$, and there are nice
  subsets $A$ and $A'$ such that $X=\chi(A)$ and $X'=\chi(A')$. In
  particular, the map $\phi$ induces a graph isomorphism $f:A\to A'$,
  for two flags $F$ and $G$ in $A$ are $\gamma$-equivalent \iff
  $\w(F,G)$ does not contain $\gamma$. By Remark \ref{R:nice}
  $(\ref{R:nice:space})$, nice subsets are $\Gamma$-spaces, so $f$ is
  an elementary map, by Corollary \ref{C:nicetype}.

  Thus, we need only show that $\phi$ can be extended to some supersets of
  flags, which satisfy the conditions of \ref{L:nice_flagmenge}. Given
  flags $F$ and $G$ in $X$, with respective images $F'$ and $G'$ in
  $X'$, if $\w(F,G)=u$, then choose reduced paths $P:F\op{u}G$ and
  $P':F'\op{u'}G'$ with
  \[X \ind\limits_{F,G} P\quad \text{ and }\quad X'
  \ind\limits_{F',G'} P'.\] Let $\psi$ be the isometry $P\to P'$ which
  maps $(F,G)$ to $(F',G')$. In order to show that $\phi\cup\,\psi$ is
  a well-defined isometry between $X\cup P$ and $X'\cup P'$, consider
  some flag $H$ in $X$. It suffices to show that $\phi\cup\psi$
  induces an isometry $H\cup P\to\phi(H)\cup P'$. By Corollary
  \ref{C:Flagtypedist}, the map $\rest{\phi}{H,F,G}$ is elementary. So
  is $\psi$, by Lemma \ref{L:alpha_kette}. As the type $\tp(H/F,G)$ is
  stationary, by Proposition \ref{P:Stat}, it follows that
  $\rest{\phi}{H,F,G}\cup\,\psi$ is an elementary map and thus an
  isometry.

  Iterating the above process countably many times, we obtain
  the desired superset satisfying \ref{L:nice_flagmenge}.
\end{proof}

The above result together and Corollary \ref{C:P_uEq}
yield the following, by compactness.

\begin{cor}\label{C:Eq}
  The theory $\psg$  is equational.
\end{cor}
Equationality, or rather, Theorem \ref{T:Eq} allows us to show total
triviality and hence weak elimination of imaginaries.
\begin{prop}[{\emph{cf.}\ \cite[Lemma 7.22]{BMPZ13}}]
  \label{P:totally_trivial}
  The theory $\psg$ is totally trivial: over any set of parameters $D$,
  given tuples $a$, $b$ and $c$ such that $a$ is independent both from
  $b$ and from $c$ over $D$, then $a$ is independent from $b,c$ over
  $D$.
\end{prop}
By induction on the length of the tuples (\emph{cf.} \cite[Lemma
4]{jBG91}), it is easy to see that it suffices to check total triviality
for singletons $a$, $b$ and $c$.
\begin{proof}
  By taking a non-forking extension to a small submodel containing
$D$, we may assume that $D$ is nice. Since every element is
contained in a flag, it suffices to show total
triviality when $a$, $b$ and $c$ enumerate three flags $F$, $H$ and
$K$.

  Let $G$ be a base-point of $F$ over $D$. In particular, we have
  $F\ind_GH$ and $F\ind_G K$.

  Choose another realisation $F'$ of $\tp(F/G)$ such that
  $F'\ind_G H,K$. Then $\w(F',G)=\w(F,G)$ and
  \[\w(F',H)=[\w(F',G)\cdot \w(G,H)]=[\w(F,G)\cdot
    \w(G,H)]=\w(F,H),\]

\noindent  since $F\ind_G H$. Likewise for $\w(F,K)$. Theorem \ref{T:Eq}
implies
that $F$ and $F'$ have the same type over $G,H,K$, so $F\ind_G H,K$, as
desired.
\end{proof}
Since $\psg$ is $\omega$-stable, we conclude by \cite[Proposition
7]{jBG91} the following.
\begin{cor}\label{C:perfect}
  The theory $\psg$ is perfectly trivial, that is, given given any set
of parameters $D$ and tuples $a$, $b$ and $c$ such that $a$ and $b$
are both independent over $D$, then so are they  over $D\cup\{c\}$.
\end{cor}
An $\omega$-stable theory $T$ has \emph{weak elimination of
  imaginaries} if the canonical base of every stationary type can be
chosen to be a subset of the real sort. In this case, types over
algebraically closed sets are always stationary.

\begin{cor}\label{C:WEI}
  The theory $\psg$ has weak elimination of imaginaries, and given
any stationary type  $\tp(a_1,\ldots, a_n/B)$, then
  \[\cb(\tp(a_1\ldots a_n/B))=\cb(\tp(a_1/B))\cup\cdots\cup
\cb(\tp(a_n/B)).\]
\end{cor}
\begin{proof}
  Considering a small elementary substructure $N$, we may assume that
the stationary type $p=\tp(a_1,\ldots,a_n/N)$. Suppose that the
canonical base of every unary type over $N$ is  interdefinable with a
finite subset of the real sort. Thus, we may choose finite subsets
$C_1,\ldots,C_n$ such that $\cb(\tp(a_i/N))$ is interdefinable with
$C_i$. Total triviality (Proposition \ref{P:totally_trivial}) implies
that $p$ does not fork over $C=C_1\cup\dotsb\cup C_n$. By Proposition
\ref{P:Stat}, the restriction of $p$ to $C$ is stationary, so $C$ is
interdefinable with $\cb(p)$.

Therefore, we need only show that the canonical base of every unary
$\tp(a/N)$ over the nice set $N$ is interdefinable with a
finite subset of the real sort. Given a flag $F$ containing $a$
independent from $N$, the type $\tp(F/a)$ is stationary, so
$\cb(\tp(a/N))$ is interdefinable with $\cb(\tp(F/N))$. Corollary
\ref{C:cb_type_pG} yields now the desired result.

\end{proof}

\begin{cor}\label{C:2based}
  The canonical base of a type is algebraic over two
independent realisations.
\end{cor}
\begin{proof}
Again by Proposition \ref{P:totally_trivial} and taking small elementary
substructures, we need only consider a unary type $p$ over some nice set
$A$. Let $a$ and $a'$ be two independent
realisations of $p$. Choose a flag $F$ containing $a$ independent
from $A$ over $a$. Since the type of $F$ over $a$ is stationary, it
follows that $\cb(a/A)$ is interdefinable with $\cb(F/A)$. By
automorphisms, we may find a flag $F'$ containing $a'$ with the same
type than $F$ over $A$ and $$ F'\ind_{a'} AF.$$
\noindent In particular, we have that $$ F F'\ind_{a,a'} A$$ and $$
F\ind_A F',$$ thus, if $\cb(F/A) \subset \acl(F,F')\cap A$, then
$\cb(a/A)=\cb(F/A)\subset \acl(a,a')$, as desired.

Therefore, we need only prove the statement for the type of a flag $F$
over $A$.  Let $F'\ind_A F$ and choose a base-point $G$ of $F$ over
$A$. Suppose that $P:F\op{u} G$ and let $\tilde u$ be the final
segment of $u\approx u_1\cdot \tilde u$. Then $w(F,F')= [u\cdot
u\inv]=u_1\cdot \tilde u \cdot {u_1}\inv$, by
Corollary \ref{C:hinundher}.

We obtain the following diagram:

\begin{figure}[h]
\centering

\begin{tikzpicture}[>=latex,text height=.25ex,text depth=0.25ex]

\fill node [below right] (0,0) (G) {} circle
(2pt);
\fill (-1,1)  node [left] (K) {} circle (2pt);
\fill (-2,2) node (F) {}  circle (2pt);
\fill (2,2)  node (FA) {} circle (2pt);
\fill (1,1) node [right] (KA) {} circle (2pt);

\node (G1) [below right of=G, node distance=4mm] {$G$};
\node (K1) [below left of=K, node distance=4mm] {$K$};
\node (KA1) [below right of=KA, node distance=4mm] {$K'$};
\node (FA1) [right of=FA, node distance=4mm] {$F'$};
\node (F1) [left of=F, node distance=4mm] {$F$};

 \draw[->] (0,0) -- (1,1) node[pos=.5, below right] {$\tilde u$};
  \draw[->] (1,1) -- (2,2) node[pos=.5, below right] {$u_1\inv$}
;
 \draw[->] (-2,2)  -- (-1,1) node[pos=.5, below left] {$u_1$}
;
 \draw[->] (-1,1) -- (0,0) node[pos=.5, below left] {$\tilde u$};
  \draw[->] (-1,1) -- (1,1) node[pos=.5, above] {$\tilde u$};
\end{tikzpicture}
\end{figure}

Since $\wob(u_1,\tilde u\cdot
u_1\inv)\subset \sr(u)$, Lemma \ref{L:Wob} implies that
$K/\sim_{\sr(u)}$ lies in  $\acl(F,F')$. Observe that
$K\sim_{\tilde u} G$ and $\tilde u\preceq \sr(u)$, by Corollary
\ref{C:sr}. Thus, the canonical base $\cb(F/A)$ is contained in
$\acl(F,F')$ as well.

\end{proof}

\section{Ampleness}\label{S:Ampleness}

Recall the definition of $n$-ampleness \cite{Pi00,Ev03}.

\begin{definition}\label{D:CM} A stable theory $T$ is
  \emph{$n$-ample} if, working inside a sufficiently saturated model and
possibly over parameters, there are  real tuples $a_0,\ldots,
a_n$ satisfying the following conditions:
  \begin{enumerate}
  \item\label{D:CM:acl}
    $\aclq(a_0,\ldots,a_i)\cap\aclq(a_0,\ldots,a_{i-1},a_{i+1})=
    \aclq(a_0,\ldots,a_{i-1})$ for every $0\leq i<n$,
  \item\label{D:CM:ind} $a_{i+1} \ind_{a_i} a_0,\ldots, a_{i-1}$ for
    every $1\leq i<n$,
  \item\label{D:CM:dep} $a_n \nind a_0$.
  \end{enumerate}
\end{definition}

Note that $T$ is $n$-ample \iff $T^\mathrm{eq}$ is \cite[Corollary
  2.4]{BMPZ13}. Furthermore, if $T$ is $n$-ample, it is $(n-1)$-ample.
A theory is $1$-based \iff it is not $1$-ample. It is CM-trivial \iff
is not $2$-ample.

In order to find an upper bound for the ample degree of $\psg$, we
will use the following result.

\begin{lemma}\label{L:redmod}(\cite[Remarks 2.3 and 2.5]{BMPZ13})
  If $T$ is $n$-ample, there are tuples $a_0,\ldots,a_n$
  enumerating small elementary substructures of an ambient saturated
  model such that for every $0\leq i<n-1$
  \renewcommand{\theenumi}{\alph{enumi}}
  \begin{enumerate}
  \item\label{L:redmod:ind} $a_n \ind_{a_{i-1}} a_i$.
  \item\label{L:redmod:inter} $\aclq(a_i,a_{i+1})\cap\aclq(a_i,a_n)=
    \aclq(a_i)$.
  \item\label{L:redmod:nind} $a_n \nind\limits_{\aclq(a_i)\cap
    \aclq(a_{i+1})} a_i$.
  \end{enumerate}
\end{lemma}
Recall that, given a subset $X$ (of some cartesian power) of a structure
$M$, the \emph{induced structure} on $X$ is the set of all relations on
every cartesian power of $X$ which are definable in $M$ without
parameters.
\begin{lemma}\label{L:ample_persistent}
  Let $X$ be a subset, definable without parameters, of a model $M$ of a
stable theory $T$. If the theory of $X$ equipped with the induced
structure is $n$-ample, then so is $T$.
\end{lemma}
\begin{proof}
  \nc{\peq}{X^{\mathrm{eq}}} We may assume that $M$ is sufficiently
  saturated. Let $a_0,\ldots,a_n$ in $X$ witness that $X$ is $n$-ample
  for some $n$ in $\N$. Since $X$ is equipped with the full induced
  structure from $M$, properties $(\ref{D:CM:ind})$ and
  $(\ref{D:CM:dep})$ of Definition \ref{D:CM} hold, when we consider
  these tuples in $M$. Furthermore, working in $M$, we have that
\[\aclq(a_0,\ldots,a_i)\cap\aclq(a_0,\ldots,a_{i-1},a_{i+1})\cap
  \peq= \aclq(a_0,\ldots,a_{i-1})\cap\peq,\] where $\peq$ denotes those
  imaginary elements of $M^\mathrm{eq}$ having some representative in
$X$. The following result together with $(\ref{D:CM:acl})$ yield the
desired result.
  \begin{claim}
    If $A$ and $B$ are subsets of $X$, then
    \[\aclq(A)\cap\aclq(B)=
    \aclq\left(\aclq(A)\cap\aclq(B)\cap\peq\right).\]
  \end{claim}
  \begin{proof}[{\it Proof of the claim:}]
    \renewcommand{\qed}{\noindent\emph{End of the proof of the
        claim.}}   If $e$
lies in $\aclq(A)\cap\aclq(B)$, it is witnessed by finite
definable sets $E_a=\varphi(x,a)$, resp. $F_b=\psi(x,b)$, with $a$ in $A$,
resp. $b$ in $B$.  The canonical parameter $d$ of the finite set
$E_a\cap F_b$, which contains $e$, belongs to $\peq$, since $\peq$ is
definably closed in $M^\mathrm{eq}$. It lies in
$\aclq(A)$, resp.\ $\aclq(B)$, because, if $E_a$, resp. $F_b$, is
fixed, there are only finitely many possibilities for the subset $E_a\cap
F_b$.

 \end{proof}
\end{proof}

As in the previous sections, we will work inside a sufficiently
saturated $\Gamma$-space $M$, a model of $\psg$.

A subgraph $Y\subset \Gamma$ is \emph{full} if, whenever two vertices
$x$ and $y$ in $Y$ are adjacent in $\Gamma$, then so are they in $Y$.

\begin{cor}\label{C:Unternice}
  Let $\Gamma'$ be a full subgraph of $\Gamma$ and $F$ a fixed flag in
  $M$. Consider the $\Gamma'$-residue of $F$
  \[X=\{ F' \text{ flag in } M
  \,| \, F'\sim_{\Gamma'} F \}.\] The set $X$ is the collection of
  flags of some nice set $A$. The restriction $M'=\A_{\Gamma'}(A)$ to
  vertices with colours in $\Gamma'$ is a model of
  $\mathrm{PS}_{\Gamma'}$.

  If $\mathrm{PS}_{\Gamma'}$ is $n$-ample, so is $\psg$.
\end{cor}

\begin{proof}
  If two flags from $X$ are connected by a reduced flag-path $P$ with
  word $u$, the $|u|$ is a subset of $\Gamma'$, by simple
  connectedness. Thus, all flags in $P$ belong to $X$, so $X=\chi(A)$,
  for some nice subset $A$, by Lemma \ref{L:nice_flagmenge}.

  The restriction $G\mapsto G'=\rest{G}{\Gamma'}$ is a bijection
  between the flags in $A$ (i.e.\ the elements of $X$) and the flags
  of $M'$. Since $\Gamma'$ is full, the Coxeter group generated by
  $\Gamma'$ is a subgroup of $(W,\Gamma)$, so a reduced word on the
  letters of $\Gamma'$ remains so as a word in $\Gamma$. It follows
  that $M'$ is a simply connected $\Gamma'$-space. For every
  $\gamma\in \Gamma'$ and $G$ in $X$, every flag in $M$ which is
  $\gamma$-equivalent to $G$ belongs again to $X$. Thus $M'$ is a
  model of $\mathrm{PS}_{\Gamma'}$.

  In order to show that, if $\mathrm{PS}_{\Gamma'}$ is $n$-ample, so
  is $\psg$, we need only show, by Lemma \ref{L:ample_persistent},
  that the induced structure on $M'$ (as an definable subset of $M$
  with parameters in $F$) coincides with the structure of $M'$ as a
  model of $\mathrm{PS}_{\Gamma'}$. That is, that every definable
  relation on $M'$ which is definable in $M$ over $F$ is then
  $F$-definable in the $\Gamma'$-space $M'$, or equivalently, that a
  type in $M'$ over $F$ determines a unique type in $M$ over $F$.
  Assume therefore that the tuples $c_1$ and $c_2$ have the same type
  in $M'$ over $F'$. Choosing nice sets $D'_i$ in $M'$, for $i=1,2$
  containing $F', c_i$, there is an elementary map $f':D'_1\to D'_2$
  in $M'$, which maps $c_1$ to $c_2$ and is the identity on $F'$. The
  sets $D_i=D'_i\cup (F\setminus M')$ are clearly nice in $A$, and
  hence in $M$. The map $f'$ extends to an elementary map between
  $D_1$ and $D_2$, which is the identity on $F$. Thus the tuples $c_1$
  and $c_2$ have the same type over $F$ in $M$, as desired.
\end{proof}

Recall (Definition \ref{D:beginning}) that a letter $s$ is an end of the
word $u$ if $u\approx v\cdot s$. The final segment of $u$, denoted by
$\tilde u$, is the commuting subword consisting of all ends of $u$. When a
word $w$ is commuting, we will identify it with its support $|w|$.

\begin{lemma}\label{L:Endstueckwirdgroesser}(\cite[Lemma 8.2 and
    Proposition 8.3]{BMPZ13}) Consider nice sets $A$ and $B$ and a
  flag $F$ such that $F\ind_B A$ and $\acl(AB)\cap\acl(AF)=
  \acl(A)$. Let $u = u_B$ (resp.\ $u_A$) be the minimal word
  connecting $F$ to a flag $G_B$ in $B$ (resp.\ $G_A$ in $A$).
  Consider the reduced word $v$ which connects $G_B$ to $G_A$ and the
  associated symmetric decomposition
  \begin{align*}
    u&=u_1\cdot u'\cdot w& w\cdot v'\cdot v_1&=v,
  \end{align*}
  as in Proposition \ref{P:Dec}. Then the word $w\cdot v_1$ is
  commuting. Furthermore, if $$F\nind\limits_{A\cap B} A,$$\noindent
  then
  $$|v'| \nsubseteq \tilde u \subsetneq \tilde u_A,$$
  \noindent where $\tilde u$ and  $\tilde u_A$ are the final segments
of $u$ and $u_A$, respectively.
\end{lemma}

\begin{proof}
  By transitivity of non-forking, we have that $F\ind_{G_B} G_A$,
so $$u_A= [u\cdot v]=u_1\cdot w\cdot v_1.$$

  We will first show that $w\cdot v_1$ is commuting. Considering its
final segment, write $w\cdot v_1\approx a\cdot b$ where $b$ is a commuting
    word. In order to show that $a$ is trivial, we need only show that
$b$ and $a$ commute.

Now, the word $u'\cdot w$ is left absorbed by $w\cdot
    v_1$, so write $u'\cdot w\approx u_a\cdot u_b$, by Lemma
\ref{L:reordering},  where $u_a$ is left-absorbed
    by $a$ and $u_b$ commutes with $a$ in order to be left-absorbed by $b$.
Choose a flag path $P:G_B\op{v'\cdot a}H\op{b}G_A$, for some flag $H$.
Choosing $P$ independent from $F$ over $G_B,G_A$, we may assume that
$F\ind_{G_B} P$, so  $G_B$ is a base-point of $F$ over the nice set
determined by $P$. Therefore

\[\w(F,H)=[u\cdot v'\cdot a]=u_1\cdot u'\cdot w\cdot v'\cdot a
    =u_1\cdot u_b\cdot u_a\cdot v'\cdot a=u_1\cdot a\cdot u_b.\]

\noindent Choose now a flag $K$ with $F\op{u_1\cdot a}K\op{u_b}H$. Observe
that $\w(K,G_A)\preceq [u_b\cdot b]=b$. On the other hand, the word
    $\w(F,G_A)=[u\cdot v]=u_1\cdot w\cdot v_1\approx u_1\cdot a\cdot
    b$, so $\w(K,G_A)=b$.

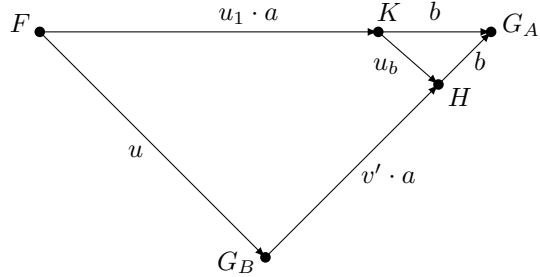
    \begin{figure}[h]
    \centering

    \begin{tikzpicture}[>=latex,text height=.25ex,text depth=0.25ex]

      \fill node [below left] (0,0) (FB) {$G_B$} circle (2pt);
      \fill (2.3,2.3)  node (H) {} circle (2pt);
      \fill (1.5,3) node (K) {}  circle (2pt);
      \fill (3,3)  node (FA) {} circle (2pt);
      \fill (-3,3) node [left] (F) {$F$} circle (2pt);

      \node (H1) [below right of=H, node distance=4mm] {$H$};
      \node (K1) [above right of=K, node distance=2mm] {$K$};
      \node (FA1) [right of=FA, node distance=4mm] {$G_A$};

      \draw[->] (0,0) -- (2.3,2.3) node[pos=.5, below right] {$v'\cdot
        a$};
      \draw[->] (2.3,2.3) -- (3,3) node[pos=.5, below right]
           {$b$} ;
      \draw[->] (-3,3) -- (0,0) node[pos=.5, below left]
           {$u$} ;
      \draw[->] (-3,3) -- (1.5,3) node[pos=.5, above
        right] {$u_1\cdot a$};
      \draw[->] (1.5,3) -- (3,3)
      node[pos=.5, above] {$b$};
      \draw[->] (1.5,3) -- (2.3,2.3)
           node[pos=.5, below left] {$u_b$};
    \end{tikzpicture}
  \end{figure}

Set $\Cent(a)$ to be the set of colours commuting with $a$. Clearly
$\Cent(a)\cap\sr(a)=\emptyset$ and both $\wob(v'\cdot a, b)$ and
$\wob(u_1\cdot a, b)$ are
  subsets of $S=\sr(a)\cup\Cent(a)$. Lemma \ref{L:Wob} implies
that $H/S$ lies in $\acl(AB)$ and $K/S$ in $\acl(AF)$. Since
 $u_b$ commutes with $a$ and $K\op{u_b} H$, we have that
  \[K/S=H/S\in\acl(AB)\cap\acl(AF)=\acl(A).\]
    Lemma \ref{L:basept} implies that $b\subset S$. However, no letter
    of $b$ is contained in $\sr(a)$, since $a\cdot b$ is reduced and $b$
is commuting, so  $b$ and $a$ commute, as desired.

Let us now show that $\tilde u\subset\tilde u_A$. Since
$a=1$ and $b=w\cdot v_1$, the previous diagram yields the
following picture:

  \begin{figure}[h]
    \centering

    \begin{tikzpicture}[>=latex,text height=.25ex,text depth=0.25ex]

      \fill node [below left] (0,0) (FB) {$G_B$} circle (2pt);
      \fill (1.5,1.5)  node (H) {} circle (2pt);
      \fill (0,3)  node (K) {} circle (2pt);
      \fill (-1.5,1.5) node (F1) {}  circle (2pt);
      \fill (3,3)  node (FA) {} circle (2pt);
      \fill (-3,3) node [left] (F) {$F$} circle (2pt);

      \node (HA) [below right of=H, node distance=4mm] {$H$};
      \node (KA) [above right of=K, node distance=2mm] {$K$};
      \node (GA1) [right of=FA, node distance=4mm] {$G_A$};

      \draw[->] (-3,3)  -- (-1.5,1.5) node[pos=.5, below left]
           {$u_1$} ;
      \draw[->]  (-1.5,1.5)-- (0,0) node[pos=.5, below left]
           {$u'\cdot w$};
      \draw[->] (-3,3) -- (0,3) node[pos=.5, above right]
            {$u_1$} ;
      \draw[->] (0,3) -- (3,3) node[pos=.6, above]
           {$w\cdot v_1$} ;

      \draw[->]  (-1.5,1.5) -- (0,3) node[pos=.5, below left]
           {$v'$} ;
      \draw[->] (0,0) -- (1.5,1.5) node[pos=.5, below right]
           {$v'$} ;

      \draw[->]   (1.5,1.5) -- (3,3) node[pos=.5, below right]
           {$w\cdot v_1$} ;

      \draw[->]  (0,3) -- (1.5,1.5)   node[pos=.6, below left]
           {$u'\cdot w$} ;

    \end{tikzpicture}
  \end{figure}

  \noindent The final segment $\tilde u=\hat u_1\cdot \tilde u'\cdot w$ for
  some commuting final subword $\hat u_1$ of $u_1$, which commutes with
  $u'\cdot w$. Since $v'$ is (properly) right-absorbed by $u_1$, we
  have $\sr(v')\subset\sr(u_1)\subset T= \hat u_1\cup
  \Cent(\hat u_1)$, so
  \[\wob(v',w\cdot v_1)\subset\wob(u_1,w\cdot v_1)\subset T.\]
  Note that $|u'\cdot w|\subset T$ and as above, the flag $K/T$ lies in
  $\acl(A)$. Lemma \ref{L:basept} implies that $|w\cdot v_1|\subset
  T$. Since the word $\hat u_1\cdot w\cdot v_1$ is reduced, no letter
of the commuting word $w\cdot v_1$ is contained in $\hat u_1$.
  Therefore $w\cdot v_1$ and $\hat u_1$ commute, so
  \[\tilde u=\hat u_1\cdot w\cdot\tilde u'\subset
  \hat u_1\cdot w\cdot v_1=\tilde u_A.\]

  To conclude, notice that if $|v'|\subset \sr(u_A)$, then
$G_A/\sr(u_A)=G_B/\sr(u_A)$, which would imply $F\ind_{A\cap B} A$ by
Corollary \ref{C:cb_type_pG}. Therefore, if $F\nind_{A\cap B} A$,
then neither $\tilde u=\tilde u_A$ nor $|v'|\subset \tilde  u$: for the
first case, if $\tilde u=\tilde u_A$, then $\tilde u'=v_1$ and hence
$u'=v_1=1$, since $\tilde
  u'$ is properly absorbed by $v_1$. For the latter, if $|v'|\subset \tilde
 u$, then $|v'|\subset \tilde  u_A\subset\sr(u_A)$.
\end{proof}

If the graph $\Gamma$ has no edges, the theory $\psg$ is the theory of
an infinite set $M$ partitioned into $|\Gamma|$ many infinite sets
$\mathcal A_\gamma$. This is a trivial theory of Morley rank $1$
(and degree $|\Gamma|$) which is easily seen not to be
$1$-ample.

For a graph with at least one edge, we define its \emph{minimal valency}
as the minimum of the valencies of non-isolated vertices. In particular, it
is at least $1$.

\begin{theorem}\label{T:amplebounds}
  Let $\Gamma$ be a graph with at least one edge. Let $r$ be its
minimal valency and $n$ in $\N$ be maximal such that the graph $[0,n]$:

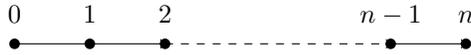
 \begin{figure}[h]
   \centering

   \begin{tikzpicture}[->=latex,text height=1ex,text depth=1ex]

     \fill  (0,0)  node [above]  {$0$}  circle (2pt);
     \fill   (1,0)  node [above]  {$1$} circle (2pt);
     \fill   (2,0)  node [above] {$2$} circle (2pt);
     \fill   (5,0)  node [above] {$n-1$} circle (2pt);
     \fill   (6,0)  node [above] {$n$} circle (2pt);

     \draw  (0,0) --  (1,0) --  (2,0) ;
     \draw[dashed] (2,0) -- (5,0) ;
     \draw (5,0) -- (6,0) ;
   \end{tikzpicture}

 \end{figure}

 \noindent embeds as a full subgraph of\/ $\Gamma$. Then the theory
 $\psg$ is $n$-ample but not $(|\Gamma|-r+1)$-ample.
\end{theorem}

If $\Gamma$ contains a full subgraph isomorphic to $[0,n]$, for some
$n\in \N$, then $n\leq |\Gamma|$ and $r\leq |\Gamma|-n$, since the
graph $[0,n]$ has minimal valency $1$. Thus $|\Gamma|-r+1$ is always
bigger than $n$, as expected. For the graph $[0,n]$, the theorem says
that its associated theory is $n$-ample but not $(n+1)$-ample
(\emph{cf.}\ \cite[Theorem 3.3]{kT14}, \cite[Theorem 8.4]{BMPZ13}),
hence the bounds are best possible, similarly as for the graph
consisting $0,\ldots,n+1$ arranged in a circular way;

\vskip1mm
\hskip3.5cm
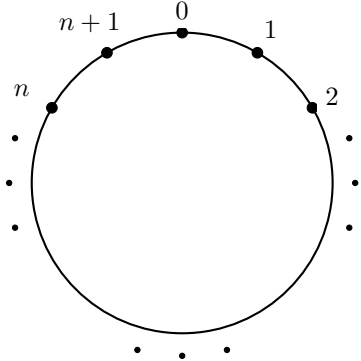
\begin{tikzpicture}[scale=2,cap=round,>=latex]

        \draw[thick] (0cm,0cm) circle(1cm);

        \foreach \x in {30,60,...,150} {

          \filldraw[black] (\x:1) circle(1pt);
        }

        \foreach \x in {-15,0,15, -75,-90,255,165,180,195} {
          \filldraw[black] (\x:1.15cm) circle(.5pt);

        }

        \draw (30:1.15cm)  node[fill=white] {$2$};
        \draw (60:1.18cm)  node[fill=white] {$1$};
        \draw (90:1.15cm)  node[fill=white] {$0$};
        \draw (120:1.23cm)  node[fill=white] {$n+1$};
        \draw (150:1.23cm)  node[fill=white] {$n$};

\end{tikzpicture}
\vskip2mm

\noindent which has valency $2$, so its theory is $n$-ample yet not
$(n+1)$-ample.

In particular, the theory of $\Gamma$ is not $1$-based \iff $\Gamma$
contains at least one edge. The complete graph $\Kom_n$ has minimal
valency $n-1$ and the theory $\mathrm{PS}_{\Kom_n}$ is CM-trivial for
every $n$.

\begin{proof}
  Suppose that $[0,n]$ embeds as a full subgraph in $\Gamma$. Fix some
  flag $F$ and consider the collection of flags $[0,n]$-equivalent to
  $F$. Corollary \ref{C:Unternice} and \cite[Theorem
    8.4]{BMPZ13} imply that $\psg$ is $n$-ample.\\

  Suppose now that $\psg$ is $N$-ample for some natural number $N$,
  and let $a_0,\ldots, a_N$ be enumerations of small models as in
  Lemma \ref{L:redmod}. Total triviality of $\psg$ implies that we may
  replace $a_N$ by a flag $F$. For $0\leq i \leq N-1$, let $u^i$ be
  the reduced word connecting $F$ to a base-point $G_i$ in the nice
  set $a_i$. Lemma \ref{L:Endstueckwirdgroesser} applied to each
  triangle $(F, a_i, a_{i+1})$ implies that the final segment $\tilde
  u^{i+1}$ of $u^{i+1}$ is properly contained in the final segment
  $\tilde u^i$ of $u^i$. In particular $|\tilde u^{1}|\geq N-1$.

  Let $v$ be the reduced word connecting $G_1$ to $G_0$. Lemma
  \ref{L:Endstueckwirdgroesser} implies the existence of a word $v'$
  which is properly right-absorbed by $u^1$ and such that $|v'|$ is
  not contained in $\tilde u^1$. Let $\gamma$ be in $|v'|\setminus
\tilde u^1$. Since $v'$ is absorbed by $u^1$, so is $\gamma$. Thus
$\gamma$  commutes with $\tilde u^1$ and  must be
properly right-absorbed by $u^1$. Hence $\gamma$ is
  not isolated and has the valency at most  $|\Gamma|-|\tilde u^1|-1$.
  So  $r\leq |\Gamma|-N$, that is $N < |\Gamma|-r+1$.
\end{proof}



\end{document}